\documentclass[11pt,letter]{amsart}

\usepackage{amssymb,amsmath,epsfig,graphics,mathrsfs,enumerate,verbatim}
\usepackage[pagebackref,colorlinks=true,linkcolor=blue,citecolor=blue]{hyperref}
\usepackage{fancyhdr}
\usepackage{hyperref}
\hypersetup{
 colorlinks   = true,
 urlcolor     = blue,
 linkcolor    = blue,
 citecolor   = red ,
 bookmarksopen=true
}

\usepackage{amsmath}
\usepackage{amsfonts}
\usepackage{amssymb}
\usepackage{amsthm}
\usepackage{epsfig,graphics,mathrsfs}
\usepackage{graphicx}
\usepackage{dsfont}

\usepackage[usenames, dvipsnames]{color}

\usepackage{hyperref}

\usepackage[a4paper,
bindingoffset=0in,
left=0.8in,
right=0.8in,
top=1.3in,
bottom=1.3in,
footskip=.25in]{geometry}

\def \N {\mathbb{N}}

\def \phi {\varphi}

\def \RN {\mathbb{R}_+^{d+1}}
\def \R {\mathbb{R}}

\def \G{\Gamma}
\newcommand{\Ba}{\mathscr B_z^{(a)}}

\def \vf{\varphi}

\newcommand{\Rd}{\mathbb R^d}

\newcommand{\p}{\partial}

\newcommand{\la}{\lambda}

\numberwithin{equation}{section}

\newcommand{\beq}{\begin{equation}}
\newcommand{\bea}[1]{\begin{array}{#1} }
\newcommand{\eeq}{ \end{equation}}
\newcommand{\ea}{ \end{array}}

\newcommand{\ve}{\varepsilon}

\newcommand{\sa}{\langle}
\newcommand{\da}{\rangle}
\newcommand{\Xs}{\mathcal X}
\newcommand{\XT}{\mathcal{X}_T}

\newcommand{\C}{\mathbb{C}}

\newcommand{\snn}{\mathscr S^+}

\newtheorem{theorem}{Theorem}[section]
\newtheorem{lemma}[theorem]{Lemma}
\newtheorem{proposition}[theorem]{Proposition}
\newtheorem{corollary}[theorem]{Corollary}
\newtheorem{remark}[theorem]{Remark}
\newtheorem{definition}[theorem]{Definition}

\numberwithin{equation}{section}
\begin{document}


\title[]{Strichartz estimates for Schr\"odinger equations with nonlinear boundary interactions}

	
	
\date{}
	
\begin{abstract}
We study a Schr\"odinger equation in the upper half-space with a nonlinear Neumann boundary interaction driven by the Bessel operator $\Ba$; the natural range of the parameter for the boundary interaction is $-1<a<1$, although part of the linear theory holds for all $a>-1$. The problem arises naturally as an extension formulation for a nonlocal NLS with memory and can also be interpreted as a Schr\"odinger evolution with a nonlinear singular source concentrated on a codimension-one interface.

We first develop a complete linear theory for the associated inhomogeneous problem with nonhomogeneous Neumann data. A central ingredient is a new Duhamel representation formula that separates bulk and boundary dynamics and identifies the precise role of the boundary propagator. Using this formula, we establish sharp Strichartz estimates adapted to the geometry of the half-space and the singular structure induced by the Bessel operator. The analysis reveals a basic dichotomy between the regimes $a\ge 0$ and $-1<a<0$: in the former, bulk and boundary exhibit a unified dispersive behavior, whereas in the latter the dispersive structure becomes anomalous, requiring weighted estimates and distinct functional frameworks for the bulk and boundary contributions.

As an application of the linear theory, we prove well-posedness results for the nonlinear problem in the $L_a^2$-mass critical and subcritical regimes. For $a\ge 0$, we obtain global well-posedness for sufficiently small critical data, together with local and global results in the subcritical case. In the anomalous range $-1<a<0$, we establish, on every finite time interval, existence and uniqueness for critical data satisfying a smallness condition whose threshold depends on the length of the interval, as well as local well-posedness in the subcritical regime. The results provide a unified dispersive framework for Schr\"odinger equations with nonlinear boundary interactions and singular extension structures associated with nonlocal-in-time dynamics.
\end{abstract}

\author{Nicola Garofalo}
\address{School of Mathematical and Statistical Sciences\\ Arizona State University}\email[Nicola Garofalo]{nicola.garofalo@asu.edu}

\author{Gigliola Staffilani}
\address{Department of Mathematics\\ Massachusetts Institute of Technology\\
Cambridge, Massachusetts, 02141\\USA}\email[Gigliola Staffilani]{gigliola@math.mit.edu}

\thanks{G. Staffilani is funded in part by the NSF grant DMS-2306378 and the Simons Foundation through the
Simons Collaboration Grant on Wave Turbulence.}
	
\maketitle
	
\tableofcontents

\section{Introduction}

In this paper we study dispersive properties and  well-posedness for a Schr\"odinger equation in a half-space with a nonlinear Neumann boundary interaction. More precisely, given $-1<a<1$, $\mu\in \C$ and $p>1$ (several of the linear results below hold in the larger range $a>-1$, and we will indicate when this is the case), we analyze the following Cauchy problem in the half-space $\RN = \Rd_x \times \R^+_z$: 
\begin{equation}\label{cp02}
\begin{cases}
\p_t U - i(\Delta_x U + \Ba U) = F(X,t), & \ \ X=(x,z)\in \mathbb{R}^{d+1}_+,\ t\in \mathbb R, \\
\displaystyle \lim_{z\to 0^+} z^a \p_z U(X,t) = -\mu |U(x,0,t)|^{p-1}U(x,0,t), & \ \ x\in \Rd,\ t\in \R, \\
U(X,0)=u_0(X), 
\end{cases}
\end{equation}
where we have denoted with $\Ba := \p_{zz}  + \frac az \p_z$
the singular Bessel operator on the half-line $z>0$. 

The interest of \eqref{cp02} is motivated by two complementary interpretations which are further elucidated in Section \ref{S:mot}, where we also provide several references:
\begin{itemize}
\item[(i)] on the one hand, when $a=0$, extending $U$ evenly across the boundary $z=0$ leads to a Schr\"odinger evolution in the whole space with a nonlinear source concentrated on the codimension-one hypersurface $\{z=0\}$. More generally, for arbitrary $a>-1$, we obtain a nonlinear singular interaction supported on the interface  coupled to a singular Bessel operator; 
\item[(ii)] on the other hand, and more fundamentally for the purposes of this work, \eqref{cp02} arises naturally as the extension problem associated with a nonlocal NLS with memory, in a manner reminiscent of the extension procedure of Caffarelli and Silvestre for fractional Laplacians.
In a forthcoming work we will use the theory developed here to analyze this nonlocal problem.
\end{itemize}

The present paper develops a complete dispersive framework for \eqref{cp02}. The main difficulty is that the interaction between the boundary geometry and the Bessel operator produces two genuinely different dispersive regimes. When $a\ge 0$, the equation behaves in many respects like a classical Schr\"odinger evolution: the propagator obeys the natural scaling $(X,t)\to (\la X,\la^2 t)$ of the equation, bulk and boundary exhibit compatible dispersion, and the corresponding Strichartz estimates fit into a unified framework. In contrast, when $-1<a<0$, the dispersive structure becomes anomalous. The bulk propagator no longer behaves according to the above scaling laws, weighted estimates become unavoidable, and the natural functional spaces involve sums and intersections of Lebesgue spaces. One of the principal goals of the paper is to identify and quantify this dichotomy.

Our analysis begins with the linear problem with nonhomogeneous Neumann data,
\begin{equation}\label{nozeroin}
\begin{cases}
\partial_tU-i(\Delta_x U+\Ba U)=F(X,t),\\
\displaystyle \lim_{z\to0^+} z^a\partial_zU(X,t)=\Phi(x,t),\\
U(X,0)=u_0(X).
\end{cases}
\end{equation}
A central contribution of the paper is the derivation of a new Duhamel representation formula for \eqref{nozeroin}. This formula separates the evolution into bulk and boundary components and identifies the precise role played by the boundary propagator. More precisely, we show that solutions can be represented as
\[
U= \mathbb T_a^\star(u_0)+\mathbb D_a(F)+\Theta_a^\star(\Phi),
\]
where the operators $\mathbb T_a^\star$ and $\mathbb D_a$ describe the bulk evolution, while $\Theta_a^\star$ encodes the boundary interaction. This decomposition lies at the heart of the paper and provides the mechanism through which the bulk and boundary dispersive behaviors can be analyzed separately.

Building on this representation formula, we establish sharp Strichartz estimates adapted to the geometry of the problem. The estimates reveal a striking asymmetry between the bulk and boundary dynamics. In the range \(a\geq0\), the bulk theory obeys the natural scaling relation
\[
\frac{2}{q}+\frac{d}{r}+\frac{a+1}{m}
=\frac{d+a+1}{2},
\]
leading to a  family of triples $(q,r,m)$ of admissible exponents. In the anomalous regime \(-1<a<0\), however, the relevant estimates no longer follow the scaling dictated by dimensional analysis. Instead, the bulk propagator behaves as though \(a=0\), but only after introducing a nontrivial weight in the transverse variable \(z\). The resulting estimates involve discrepant dispersive rates and require a substantially different functional framework.

An additional difficulty concerns traces on the boundary. Although the Strichartz estimates provide uniform control in the transverse variable $z$, such bounds do not by themselves guarantee convergence to the boundary manifold $\{z=0\}$. Since the nonlinearity in \eqref{cp02} acts precisely through the boundary trace $U(x,0,t)$, establishing the existence and continuity of traces becomes an essential step in the nonlinear analysis. We resolve this issue through a separate trace theorem which allows the nonlinear fixed-point argument to close in the critical spaces.

As an application of the linear theory, we prove local and global well-posedness results for \eqref{cp02} in the $L_a^2$-mass critical and subcritical regimes. In the range $a\ge 0$, we obtain global well-posedness for sufficiently small critical data, together with local and global results in the subcritical case. In the anomalous range $-1<a<0$, we prove, on every finite time interval, existence and uniqueness for critical data satisfying a smallness condition whose threshold depends on the length of the interval, as well as local well-posedness in the subcritical regime.

The present work extends and substantially generalizes our previous analysis in \cite{GS}, where we studied the one-dimensional problem associated with the Bessel Schr\"odinger equation on the half-line. It is also related to recent investigations of nonlinear Schr\"odinger equations with boundary nonlinearities in the case $a=0$, especially in one spatial dimension, i.e., when $d=1$. However, the singular extension structure considered here leads to a considerably richer dispersive theory and requires new analytical tools that do not appear in the case $a=0$.

\medskip


\subsection{Motivation and references}\label{S:mot}


In the opening discussion we have mentioned the two complementary interpretations  (i) and (ii) for problem \eqref{cp02}. We next provide more specific context for both aspects, along with relevant references. The interpretation (i) is  equivalent to an interface problem in the whole space $\R^{d+1}$ with a nonlinear point source (or sink) concentrated on the codimension-one hypersurface $\{z = 0\}$. For instance, if $a = 0$, and $U$ solves \eqref{cp02}, then the even extension in $z$, $\bar U(x,z,t) := U(x,|z|,t)$, solves the nonlinear interaction problem 
\begin{equation}\label{cp02e}
\begin{cases}
\p_t \bar U - i(\Delta_x \bar U + \p_{zz} \bar U) = \bar F(X,t) + 2 i \mu |\bar U|^{p-1} \bar U\ \delta_{\{z=0\}},  
\\
\bar U(X,0)=\bar u_0(X), 
\end{cases}
\end{equation}
where now $X = (x,z)\in \R^{d+1},\ t\in \R$. In \eqref{cp02e} we have let $\bar F(x,z,t) = F(x,|z|,t)$, $\bar u_0(x,z) = u_0(x,|z|)$, and indicated with $\delta_{\{z=0\}}$ a Dirac delta on the hyperplane $\p \RN\times \R\subset \R^{d+1}\times \R$. When $a\not= 0$, the even extension solves the analogous interface problem for the operator $\Delta_x + \frac{1}{|z|^a}\p_z(|z|^a\p_z\, \cdot\,)$, which combines the classical Schr\"odinger evolution with that of the singular Bessel propagator on the half-line. In this case the interface term must be interpreted in the weak sense with respect to the measure $|z|^a\,dX$: for every test function $\vf$, the interaction contributes the term $2i\mu\int_{\R}\int_{\Rd}|\bar U|^{p-1}\bar U(x,0,t)\,\overline{\vf}(x,0,t)\,dx\,dt$ to the weak formulation.
There exists a large literature on $\delta$-potentials problems, especially in connection with Bose-Einstein condensates and the Gross-Pitaevskii equation (corresponding to $p=3$ in the NLS, see \cite{PS}), but they differ from \eqref{cp02e} in the fact that the Dirac delta plays the role of a potential. By this we mean that the right-hand side is of the type $\gamma |\bar U|^{q-1} \bar U  + i \mu \delta_{\{z=0\}} \bar U$. We only quote some of the relevant references in this connection: \cite{ABD, ABD2, AT, FTDSK, AGHH, HMZ, FOO, HZ, KS, DP, BV, GMW, CFN, CFN23, GI1, GI2, OST3, MSa, GIS}. 
The linear Schr\"odinger equation in a half-space with linear Robin conditions was studied in \cite{Au}, the NLS with linear Dirichlet, Neumann or Robin boundary  conditions on the half-line and in a half-plane was  studied in \cite{SB, Ho, ET, BSZ, RSZ, HK19, HiM1, HiM2, OST2}, see also \cite{CK} for a study of the generalized KdV on a half-line with Dirichlet conditions. For studies of the NLS in the complement of a non-trapping, or a convex obstacle, we refer to \cite{TsuY, BGT,  An, KVZ, KVZ2, Iv}, for exterior domains see \cite{Tsu}. Recently, there has been a growing interest in NLS in a half-line or half-plane with nonlinear boundary conditions (Dirichlet or Neumann). This corresponds to the situation $a=0$ and $d = 0, 1$ in \eqref{cp02}. We refer the reader to the works \cite{AD, HK19, HKO, HOS} and their bibliographies.

The second interpretation (ii) for the problem \eqref{cp02} is connected to the following nonlocal NLS with memory
\begin{equation}\label{ost3}
\begin{cases}
(\p_t u - i \Delta_x u)^{s}  = \la |u|^{p-1} u,\ \ \ \ \ \ \ \ \ \ \ x\in \Rd, t>0,
\\
u(x,t) = u_0(x,t),\ \ \ \ \ \ \ \ \ \ \ \ \ \ \ \ \ \ \ \ \ \ \ x\in \Rd, t\le 0,
\end{cases}
\end{equation}
where $s\in (0,1)$, $\la\in \C$ and $p>1$. 
In a work in preparation, we study Strichartz estimates and the well-posedness of \eqref{ost3} via an \emph{extension problem} in the half-space which is precisely the problem \eqref{cp02}, 
in which we take $a = 1-2s\in (-1,1)$. 
This scenario is clearly reminiscent of the celebrated extension procedure of Caffarelli and Silvestre for the fractional powers of the Laplacian $(-\Delta)^s$, see \cite{CS}. However, as the present work shows, the problem \eqref{cp02} is considerably different since the extension operator is itself of Schr\"odinger type and it thus involves oscillatory phenomena in which positivity plays no role. For related works on extensions involving the wave operator, we refer the reader to \cite{KST, EGV}.

\section{Description of the work and statement of the results} 

In this long section we discuss the more technical aspects of the work and also state the main results. It will be helpful for the reader to become familiar right from the start with some of the relevant notation. We continue to indicate with $X = (x,z), Y = (y,\zeta)$ generic points in the bulk space $\RN$, whereas the notation $X_0 = (x,0), Y_0 = (y,0)$ is used for points of the boundary 
$\p \RN$. Points in $\RN\times \R$ will be denoted by
$(X,t), (Y,\tau)$. 
 As customary, $\mathscr S(\R^{d+1})$ is the Schwartz class in the variables $(x,t)\in \R^{d+1}$, whereas we  denote by $\mathscr S(\RN)$ and $\mathscr S(\RN\times \R)$ the restrictions to the half-spaces $\RN$ and $\RN\times \R$ of the Schwartz functions in $\R^{d+1}_{(x,z)}$ and $\R^{d+1}_{(x,z)}\times \R_t$ respectively. We indicate by  $d\omega_a(X) = z^a dx dz$ the invariant measure for the differential operator $\Delta_x +\Ba$ in $\RN$. We mean by this that given $U, V\in C^\infty_0(\RN)$, we have 
 \[
\int_{\RN} \overline V(\Delta_x U + \Ba U)d\omega_a(X) = \int_{\RN} U\ \overline{(\Delta_x V + \Ba V)}d\omega_a(X).
\]
We denote with $L^m_{a,z} = L^m(\R^+,z^a dz)$, $1\le m<\infty$. When $m=\infty$, we understand that $L^\infty_{a,z} = L^\infty_z$. In this work we routinely use mixed norm spaces $L^m_{a,z} L^r_x$, $L^q_t L^r_x$ and $L^m_{a,z} L^q_t L^r_x$. Such classes are discussed in detail in Section \ref{S:adapted}. When $m=r$ we write $L^m_a(\RN)$ for the class $L^m_{a,z} L^m_x$. For instance, $L^2_a(\RN) = L^2_{a,z} L^2_x  = L^2(\RN,d\omega_a(X))$. When $m=\infty$, we simply write $L^\infty_z L^q_t L^r_x$. The notation $C^b_z L^q_t L^r_x$ indicates the space $C_z L^q_t L^r_x\cap L^\infty_z L^q_t L^r_x$ of continuous bounded functions on $[0,\infty)$ (continuity up to $z=0$ included) with values in the Banach space $L^q_t L^r_x$, normed by the (true) supremum over $z\ge 0$; see \eqref{Cz}.  

\subsection{A key Duhamel formula} To understand the problem \eqref{cp02} we first analyze the linear problem \eqref{nozeroin} with a nonzero Neumann condition.
Our first main result, Theorem \ref{T:U}, shows that, given $u_0\in \mathscr S(\RN)$, $F\in \mathscr S(\RN\times \R)$ and $\Phi\in \mathscr S(\R^{d+1})$, then for any $-1<a<1$, the problem \eqref{nozeroin} admits a unique \emph{mild solution} given by the following generalized Duhamel formula
\begin{align}\label{nicea1in}
U(X,t) & = \int_{\RN} \mathbb S_a(X,Y,t) u_0(Y)d\omega_a(Y) +  \int_0^t \int_{\RN} \mathbb S_a(X,Y,t-\tau) F(Y,\tau) d\omega_a(Y) d\tau
\\
& - i \int_0^{t} \int_{\Rd}  \mathbb S_a(X,Y_0,t-\tau) \Phi(y,\tau) dy d\tau,
\notag
\end{align}
where we have denoted by
\begin{align*}
\mathbb S_a(X,Y,t)
&= \frac{e^{-i\,\operatorname{sgn}(t)\,\frac{(d+a+1)\pi}{4}}}{2^{\frac{2d+a+1}{2}}\pi^{\frac d2}}  
\frac{e^{i\,\frac{z^2+\zeta^2+|x-y|^2}{4t}}}{|t|^{\frac{d+a+1}{2}}}\left(\frac{z\zeta}{2|t|}\right)^{\frac{1-a}{2}}
J_{\frac{a-1}{2}}\!\left(\frac{z\zeta}{2|t|}\right),
\qquad t\neq 0,
\end{align*}
the Schr\"odinger propagator with zero Neumann condition, see Section \ref{S:propa}. Some of the basic properties of this function are:
\begin{itemize}
\item[(1)] $\overline{\mathbb S_a(X,Y,t)} = \mathbb S_a(X,Y,-t)$;
\item[(2)] $\mathbb S_a(\la X,\la Y,\la^2 t) = \mathbb S_a(X,Y,t)$,\ \ \ \ $\la>0$;
\item[(3)]
$\mathbb S_a(X,Y_0,t) :=\underset{\zeta\to 0^+}{\lim} \mathbb S_a(X,Y,t)
=
\frac{ e^{-i\,\operatorname{sgn}(t)\,\frac{(d+a+1)\pi}{4}}}{2^{d+a}\pi^{\frac d2}\Gamma\!\left(\frac{a+1}{2}\right)}
\frac{e^{i \frac{z^2+|x-y|^2}{4t}}}{|t|^{\frac{d+a+1}{2}}},
\qquad t\neq 0$.
\end{itemize}
In particular, unlike what happens for $\mathbb S_a(X,Y,t)$ when $X, Y\in \RN$, the boundary restriction $\mathbb S_a(X,Y_0,t)$ showcases one single dispersive regime $|t|^{-\frac{d+a+1}{2}}$ for all $a>-1$.
This fact is responsible for the surprising uniformity of boundary Strichartz estimates in Section \ref{S:bdry} which, unlike the estimates in the bulk in Section \ref{S:stri}, do not distinguish between the two regimes $-1<a<0$ and $0\le a < 1$.  

Formula \eqref{nicea1in} plays a pervasive role in this work and it will be convenient, henceforth,
 to express the mild solution \eqref{nicea1in} to the linear Cauchy problem \eqref{nozeroin} in the following form
\begin{equation}\label{U0}
U(X,t) = \mathbb T^\star_a(u_0)(X,t) + \mathbb D_a(F)(X,t) + \Theta^\star_a(\Phi)(X,t),
\end{equation}
where the operators $\mathbb T^\star_a, \mathbb D_a$ ($a>-1$), and $\Theta_a^\star$ ($-1<a<1$) are respectively defined by:
\begin{equation}\label{Tstar}
\mathbb T_a^\star(u_0)(X,t) := \int_{\RN} \mathbb S_a(X,Y,t) u_0(Y)d\omega_a(Y),
\end{equation}
\begin{equation}\label{Da}
\mathbb D_a(F)(X,t) := \int_0^t \int_{\RN} \mathbb S_a(X,Y,t-\tau) F(Y,\tau) d\omega_a(Y) d\tau,
\end{equation}
and
\begin{equation}\label{Thetastar}
\Theta_a^\star(\Phi)(X,t) := -i \int_0^{t} \int_{\Rd}  \mathbb S_a(X,Y_0,t-\tau) \Phi(y,\tau) dy d\tau.
\end{equation}
We emphasize that the mild solution \eqref{U0} satisfies the weighted Neumann condition
\[
\underset{z\to 0^+}{\lim} z^a \p_z U(X,t) = \Phi(x,t),\ \ \ \ x\in \Rd,\ t\not= 0,
\]
but it is not a classical solution in general. For instance, if $f\in \mathscr S(\R^{d+1})$ and $f$ is even in $z$, and if $h\in C_0^\infty[0,\infty)$, with $h(z) \equiv 1$ for $0\le z \le 1$, then for the initial datum
\begin{equation}\label{piccolou}
u_0(X) = f(X) + \frac{z^{1-a}}{1-a} h(z)\in  \mathscr S(\RN),\ \ \ \ -1<a<1,
\end{equation}
forcing term $F = 0$ and Neumann datum $\Phi = 0$, the solution given by \eqref{U0} satisfies   
\[
\underset{z\to 0^+}{\lim} z^a \p_z U(X,t) = \begin{cases}0,\ \ \ \text{for}\ t\not= 0,
\\
1,\ \ \ \text{for}\ t = 0.
\end{cases}
\]

\begin{remark}\label{R:signcheck}
It is worth verifying directly that the boundary term in \eqref{U0} produces the Neumann datum; this also explains the factor $-i$ in \eqref{Thetastar}. Take $\Phi\equiv c$ constant (the general case follows by an entirely similar computation) and $t>0$. By \eqref{piccoloSa},
\[
\Theta_a^\star(c)(z,t) = -i\, \frac{2^{-a}e^{-i\frac{(a+1)\pi}4}}{\Gamma(\frac{a+1}2)}\, c \int_0^t \frac{e^{i\frac{z^2}{4\sigma}}}{\sigma^{\frac{a+1}2}}\, d\sigma .
\]
Differentiating in $z$ and substituting $u = \frac{z^2}{4\sigma}$, one finds
\[
z^a\p_z \Theta_a^\star(c)(z,t) = -i\,\frac{2^{-a}e^{-i\frac{(a+1)\pi}4}}{\Gamma(\frac{a+1}2)}\, c\ \frac{i}{2}\, 2^{a+1} \int_{\frac{z^2}{4t}}^\infty u^{\frac{a+1}2-1} e^{iu}\, du\ \underset{z\to 0^+}{\longrightarrow}\ -i\cdot i\, c = c,
\]
where in the limit we have used \eqref{ei} in Lemma \ref{L:cis} with $\mu = \frac{a+1}2$, $b=1$. The bulk terms in \eqref{U0} have vanishing weighted Neumann trace by \eqref{dzpazero}, so the representation \eqref{U0} indeed satisfies $\lim_{z\to0^+}z^a\p_zU = \Phi$.
\end{remark}


\subsection{Admissible triples} We next introduce the relevant notion of admissible triple. As it will be clear from the next definition, in the problem \eqref{nozeroin} there exist two different regimes: $a\ge 0$ and $-1<a<0$. 

\begin{definition}\label{D:ac3}
Let $r, m\ge 2$. The triple $(q,r,m)$ is called \emph{admissible} if it satisfies:  
\begin{equation}\label{ac3}
\frac 2q + \frac dr + \frac{a+1}m = \frac{d +a+1}2,\ \ \ \ \ \text{when}\ \  a\ge 0;
\end{equation}
or
\begin{equation}\label{ac0}
\frac 2q + \frac dr +\frac 1m = \frac{d+1}2, \ \ \ \\ \ \ \ \ \text{when}\ \  -1<a< 0.
\end{equation}
\end{definition}

Note that the triple $(\infty,2,2)$ is admissible for any $a>-1$. Concerning Definition \ref{D:ac3} some comments are in order.
First, in Section \ref{S:ad} we show that \emph{for any} $a>-1$ the condition \eqref{ac3} is necessary for the following Strichartz estimate to hold for the solution to \eqref{nozeroin} with forcing term $F = 0$ and zero Neumann condition $\Phi=0$:
\begin{equation}\label{apriori0}
||\mathbb T^\star_a(u_0)||_{L^m_{a,z} L^q_t L^r_x } \le C ||u_0||_{L^2_a(\RN)}.
\end{equation}
This follows in a standard way from the scalings $(X,t)\to (\la X,\la^2 t)$. However, in the anomalous range $-1<a<0$ the operator $\mathbb T^\star_a$ does not obey these scalings. Instead, the equation \eqref{ac0} suggests that it behaves as if $a = 0$, except that a weight appears in the relevant Strichartz estimates and the spaces $L^m_{a,z} L^q_t L^r_x $ must be replaced by sums and intersections of Lebesgue spaces. 
  
A second comment is that, since we want the solution \eqref{U0} to be continuous and bounded up to the boundary manifold $\p \RN\times \R$, we need to have  $m=\infty$ in \eqref{apriori0}. In such case, 
from \eqref{ac3}, \eqref{ac0} we see that admissibility of the triple $(q,r,\infty)$ means
\begin{equation}\label{acinfty}
a\ge 0:\ \ \ \ \ \frac 2q +\frac dr = \frac{d+a+1}2;\ \ \ \ \ \ \  -1<a<0:\ \ \ \ \ 
\frac 2q +\frac dr = \frac{d+1}2.
\end{equation}
We next discuss these two cases separately. 

\vskip 0.2in

\noindent \underline{Case} $a\ge 0$. Since in such case we must have $d+a-1\ge 0$, two possibilities present:
\begin{equation}\label{a2}
\text{(a)}\ \ \ \ \ d+a-1>0,\ \ \ \ \ \ \  \text{or}\ \ \ \ \ \ \text{(b)}\ \ \ \ d+a-1=0.
\end{equation}
\begin{itemize}
\item[(a)] If $d+a-1>0$ and $q>2$, then if $(q,r,\infty)$ is admissible, the first equation in \eqref{acinfty} implies that it must be $2\le r < \frac{2d}{d+a-1}$. In such situation we are able to establish unified Strichartz estimates (bulk$+$boundary). If $q = 2$, then the triple $(2,r,\infty)$ is admissible iff $r = \frac{2d}{d+a-1}$. The open end-point $(2,\frac{2d}{d+a-1},\infty)$ is the counterpart of the Keel-Tao result for $(2,\frac{2d}{d-2})$ when $d\ge 3$, see \cite{KT, Taodis}, but the situation is complicated by the fact that in the $z$-variable we have $m=\infty$ in \eqref{apriori0}.  
\item[(b)] If $d+a-1=0$, we must necessarily have $d = 1$ and $a=0$, and the first equation in \eqref{acinfty} becomes
\[
\frac 2q +\frac 1r = 1.
\] 
This situation has been treated in the interesting recent work \cite{OST3}, and our results recover theirs for all triples $(q,r,\infty)$ satisfying $q>2$, or equivalently $r<\infty$. At the end-point $(2,\infty,\infty)$ the estimate \eqref{apriori0} fails in view of the Montgomery-Smith counterexample in \cite{MS}.  
\end{itemize}

\vskip 0.2in

\noindent \underline{Case} $-1<a<0$. Again, we have two possibilities:
 \begin{equation}\label{a20}
\text{(c)}\ \ \ \ \ d>1,\ \ \ \ \ \text{or}\ \ \ \ \ \text{(d)}\ \ \ \ \ d=1.
\end{equation}
\begin{itemize}
\item[(c)] If $d>1$ and $q>2$, then the second equation in \eqref{acinfty} implies that, if $(q,r,\infty)$ is admissible, we must have $2\le r < \frac{2d}{d-1}$. In such situation, our best possible Strichartz estimates contain a weight in the $z$ variable and present a discrepancy between bulk and boundary.
If instead $q=2$, then for $(2,r,\infty)$ to be admissible, we must have $r = \frac{2d}{d-1}$. The end-point $(2,\frac{2d}{d-1},\infty)$ when $-1<a<0$ is presently open. 
\item[(d)] If $d =1$, the second equation in \eqref{acinfty} implies that $(q,r,\infty)$ is admissible if
\[
\frac 2q + \frac 1r = \frac 12.
\] 
Since $q = 2$ is not allowed, our result cover all admissible triples $(q,r,\infty)$.
\end{itemize}

\begin{remark}\label{R:rr}
For future use, it will be important to keep in mind that the triple $(r,r,\infty)$ is always admissible, and we have: 
\begin{equation}\label{same}
q = r = \frac{2(d+2)}{d+a+1},\ \ \ \ \ \ \ \ \ \ \ q' = r' = \frac{2(d+2)}{d+3-a},\ \ \ \ \text{when}\ \ a\ge 0,
\end{equation}
and
\begin{equation}\label{same0}
q = r = \frac{2(d+2)}{d+1},\ \ \ \ \ \ \ \ \ \ \ q' = r' = \frac{2(d+2)}{d+3},\ \ \ \ \text{when}\ \ -1<a< 0.
\end{equation}
\end{remark}

\vskip 0.2in

\subsection{Strichartz estimates} We next  state our main results concerning the linear problem \eqref{nozeroin}, beginning with the case $a\ge 0$. In both Theorems \ref{T:main} and \ref{T:main2} below, given $u_0\in \mathscr S(\RN)$, $F \in \mathscr S(\RN\times \R)$ and $\Phi\in \mathscr S(\R^{d+1})$, we indicate with $U$ the unique mild solution to the problem \eqref{nozeroin} constructed in Theorem \ref{T:U}, see formulas \eqref{nicea1in} or \eqref{U0}.

\begin{theorem}\label{T:main}
Let $0\le a<1$ and suppose that the triple $(q,r,\infty)$ is admissible, i.e., that the first equation in \eqref{acinfty} holds. If $d+a-1>0$ we require that $r<\frac{2d}{d+a-1}$, if instead $d+a-1=0$, we assume that $(q,r,\infty)\not= (2,\infty,\infty)$.
 Then there exists a constant $C = C(d,a,r)>0$ such that the following Strichartz estimates hold for the solution $U$ to the problem \eqref{nozeroin}:
\begin{align}\label{Umain}
||U||_{L^\infty_t L^2_a(\RN)} \le C\ \left[||u_0||_{L^2_a(\RN)} + ||F||_{L^1_t L^2_{a}(\RN)} + ||\Phi||_{L^{q'}_t L^{r'}_x} \right],
\end{align}
and also
\begin{align}\label{Umain2}
||U||_{L^\infty_z L^q_t L^r_x } \le C\ \left[||u_0||_{L^2_a(\RN)} +  ||F||_{L^1_{a,z} L^{q'}_t L^{r'}_x } + ||\Phi||_{L^{q'}_t L^{r'}_x} \right].
\end{align}
\end{theorem}

Our next result covers the anomalous range $-1<a<0$.
We mention that in Theorem \ref{T:main2} we are using the short notation $L^{q+q_\infty}_t$ for the sum space $L^{q}_t+L^{q_\infty}_t$, and $L^{q'\cap q'_\infty}_t$ for the intersection space $L^{q'}_t \cap L^{q'_\infty}_t$.
Moreover, the weight $k$ is the one in Definition \ref{D:mn}, see also Proposition \ref{P:dis}. 

\begin{theorem}\label{T:main2}
Let $-1<a<0$. Assume that $(q_\infty,r,\infty)$ and $(q,r,\infty)$ respectively satisfy the first and second equations in \eqref{acinfty}, i.e., 
\begin{equation}\label{uffaa}
\frac{2}{q_\infty} + \frac dr = \frac{d+a+1}2,\ \ \ \ \text{and}\ \ \ \ \ \frac 2{q} + \frac dr = \frac{d+1}2. 
\end{equation} 
Moreover, if $d>1$ we require that $2\le r < \frac{2d}{d-1}$, whereas when $d = 1$ we assume that $(q,r,\infty)\not= (2,\infty,\infty)$.
Then there exists a constant $C = C(d,a,r)>0$ such that the following Strichartz estimates hold for the solution $U$ to the problem \eqref{nozeroin}:
\begin{align}\label{Umain20}
||U||_{L^\infty_t L^2_a(\RN)} \le C\ \left[||u_0||_{L^2_a(\RN)} + ||F||_{L^1_t L^2_{a}(\RN)} + ||\Phi||_{L^{q'_\infty}_t L^{r'}_x} \right],
\end{align}
and also
\begin{align}\label{Umain22}
||U\ k||_{L^\infty_z L^{q+q_\infty}_t L^r_x } \le C\ \left[||u_0||_{L^2_a(\RN)} +  ||F \ k^{-1}||_{L^1_{a,z} L^{q'\cap q'_\infty}_t L^{r'}_x } + ||\Phi||_{L^{q'_\infty}_t L^{r'}_x} \right].
\end{align}
\end{theorem}

Concerning Theorem \ref{T:main2} we explicitly remark that the hypothesis $-1<a<0$  implies $0<a+1<1$, and therefore  the equations in \eqref{uffaa} give $q<q_\infty$. 
Since by Definition \ref{D:mn} we have $0<k(z)\le 1$ for every $z>0$, and therefore $k^{-1}\ge 1$, it ensues that the second inequality in Theorem \ref{T:main2} is a stronger estimate than the one in which the norm $||\Phi||_{L^{q'_\infty}_t L^{r'}_x}$ in the right-hand side is replaced by the larger one $||\Phi k^{-1}||_{L^{q'\cap q'_\infty}_t L^{r'}_x}$.

\subsection{Well-posedness in the nonlinear Cauchy problem} We finally turn to the question of well-posedness in the problem \eqref{cp02}. Before stating the appropriate definition of mild solution to \eqref{cp02} we recall the representation \eqref{U0}.

\begin{definition}\label{solution}
Let $\mu\in \C\setminus\{0\}$. Assume that $0\le a<1$ and that $(q,r,\infty)$ is admissible, with $(q,r,\infty)\not= (2,\infty,\infty)$ when $d=1$ and $a = 0$. Given $u_0\in L^2_a(\RN)$ and $F\in L^1_t L^2_a(\RN) \cap L^1_{a,z} L_t^{q'} L^{r'}_x$, we say that a function $U$ is a \emph{mild solution} to the Cauchy problem \eqref{cp02} if $U\in C(\R,L^2_a(\RN))$, $z\to U(\cdot,z,\cdot)$ belongs to $C^b_z L^q_tL^r_x$, and it satisfies the identity 
\begin{align}\label{sol1}
U(X,t) & = \mathbb T^\star_a(u_0)(X,t) +  \mathbb D_a(F)(X,t) - \mu\, \Theta^\star_a(|U|^{p-1} U(\cdot,0,\cdot))(X,t).
\end{align}
Note that, in view of \eqref{Thetastar} and the Neumann property of the representation \eqref{U0}, the identity \eqref{sol1} encodes precisely the boundary condition $\lim_{z\to 0^+} z^a\p_z U = -\mu|U|^{p-1}U$ in \eqref{cp02}.
If instead $-1<a<0$, given $0<T<\infty$, $u_0\in L^2_a(\RN)$ and $F\in L^1_T L^2_a(\RN)$ such that 
$F  k^{-1}\in L^1_{a,z} L^{q'\cap q'_\infty}_T L^{r'}_x $, we say that $U$ is a \emph{mild solution} on $[0,T]$ if $U\in C([0,T],L^2_a(\RN))$, $z \to U(\cdot,z,\cdot)k(z)$ belongs to $C^b_z L^q_T L^r_x$, and $U$ satisfies the identity \eqref{sol1}.
\end{definition}

We next introduce a key notion for the well-posedness of problem \eqref{cp02}. For its explanation, we refer the reader to Section \ref{S:well}.

\begin{definition}\label{D:mass}
We say that the problem \eqref{cp02} is $L^2_a$-\emph{mass critical} if
\begin{equation}\label{pcpos}
p = p_c := 1 + \frac{2(1-a)}{d+a+1},\ \ \ \text{when}\ \ \ 0\le a<1,  
\end{equation}
and
\begin{equation}\label{pcneg}
p = p_c := 1 + \frac{2}{d+1},\ \ \ \ \ \text{when}\ \ \ -1<a<0.
\end{equation}
We say it is \emph{subcritical} if $1<p< p_c$ and  \emph{supercritical} if $p> p_c$.
\end{definition}

We note in passing that, when $a=0$ and $d =1$, then according to \eqref{pcpos} the problem \eqref{cp02} is $L^2$ critical when $p_c=2$. This recovers the critical exponent in \cite[Theorems 1.1 \& 1.2]{OST3}.   

\begin{remark}\label{R:pc}
When the triple $(r,r,\infty)$ is admissible, the exponent $q=r$ is uniquely identified by the equations \eqref{same} or \eqref{same0}. One verifies that in either case we have $q=r = p_c+1$, the critical exponent in \eqref{pcpos}, \eqref{pcneg}. We also note that, if we set $\tilde q := p_c q'$ and $\tilde r := p_c r' $, we have (see also Lemma \ref{L:mass})
\[
\tilde q = \tilde r = p_c+1 = q = r.
\]
This will be important in the proof of $L^2_a$-mass critical case in Theorems \ref{T:well} and \ref{T:nega}.
\end{remark}

We are ready to state our main results about well-posedness for the problem \eqref{cp02}.

\begin{theorem}\label{T:well}
Let $0\le a<1$ and, given $\mu\in \C$ and $1<p\le p_c$, consider the problem \eqref{cp02}. Assume that $r = p+1$, and the triple $(q,r,\infty)$ is admissible with $q>2$. Suppose that $u_0\in L^2_a(\RN)$ and that  $F\in L^1_t L^2_a(\RN)\cap L^1_{a,z} L^{q'}_t L^{r'}_x $.
\begin{itemize}
\item[(1)] \underline{\emph{The $L^2_a$-mass critical case}}. If $p = p_c$ in \eqref{pcpos}, then with $q = r = p_c+1$ given by \eqref{same} there exists $\ve_0 = \ve_0(d,a,\mu)>0$ such that, if
\[
||u_0||_{L^2_a(\RN)} + ||F||_{L^1_t L^2_a(\RN)} + ||F||_{L^1_{a,z} L^{r'}_t L^{r'}_x} \le \ve_0,
\]
then there exists a global in time mild solution 
$U\in C(\R,L^2_a(\RN)) \cap C^{b}_z L^r_t L^r_x $ 
of the problem \eqref{cp02}, unique in the ball $\{||U||_{\Xs}\le 2\ve_0\}$ of the space $\Xs$ in \eqref{space}.
\item[(2)] \underline{\emph{The $L^2_a$-mass subcritical case}}. If $1<p<p_c$, then for any $\mu\in \C$ there exists a small $T = T(d,a,p,|\mu|,||u_0||_{L^2_a(\RN)},F)>0$ such that the problem \eqref{cp02} admits a mild solution $U\in C([0,T],L^2_a(\RN)) \cap C^{b}_z L^q_T L^r_x$, unique in $\XT$. 
If $F= 0$ and $\Im(\mu) = 0$, such $U$ can be extended to a mild solution on $[-M,M]$ for every $M>0$.
\end{itemize}
\end{theorem}

\vskip 0.2in

Next, we consider the regime $-1<a<0$. In the statement of Theorem \ref{T:nega}, given a time interval $[0,T]$, the reader should keep in mind the mixed norms defined by \eqref{LmixT}.

\begin{theorem}\label{T:nega}
Let $-1<a<0$ and, given $0<T<\infty$, $\mu\in \C$ and $1<p\le p_c$, consider the problem \eqref{cp02} on $\RN\times[0,T]$. Let $r = p+1$ and suppose that the triples $(q_\infty,r,\infty)$ and $(q,r,\infty)$ satisfy the equations \eqref{uffaa}, with $q>2$. Let $u_0\in L^2_a(\RN)$ and $F\in L^1_T L^2_a(\RN)$ be such that 
$F  k^{-1}\in L^1_{a,z} L^{q'\cap q'_\infty}_T L^{r'}_x $.
\begin{itemize}
\item[(1)] \underline{\emph{The $L^2_a$-mass critical case}}. 
If $p=p_c$, there exists $\ve_0 = \ve_0(d,a,\mu,T)>0$ such that, if   
\[
 C \max\{1,T^{\frac{1}{q}-\frac{1}{q_\infty}}\} \left[\|u_0\|_{L^2_a(\RN)} + ||F||_{L^1_T L^2_a(\RN)} + (1+T^{\frac{1}{q}-\frac{1}{q_\infty}}) ||F k^{-1}||_{L^1_{a,z} L^{q'\cap q'_\infty}_T L^{r'}_x}\right]\le \varepsilon_0,
\]
then there exists a mild solution $U\in C([0,T],L^2_a(\RN))$ to the problem \eqref{cp02}, such that $U k\in C^b_z L^q_T L^r_x$, unique in the ball $\{||U||_{\XT}\le 2\ve_0\}$ of the space $\XT$ in \eqref{Xneg}.
\item[(2)] \underline{\emph{The $L^2_a$-mass subcritical case}}. If $1<p<p_c$, then for any $\mu\in \C$ it is possible to select a small $0<T_0\le T$, depending only on $d, a, p, |\mu|, \|u_0\|_{L^2_a(\RN)}$ and the norms of $F$, for which there exists a mild solution $U\in C([0,T_0],L^2_a(\RN))$ to the problem \eqref{cp02} in $\RN\times [0,T_0]$, such that $U k\in C^b_z L^q_{T_0} L^r_x$, unique in $\mathcal X_{T_0}$.
\end{itemize}
\end{theorem}

\begin{remark}\label{R:critnega}
Since $\gamma := \frac 1q - \frac{1}{q_\infty}>0$, the threshold $\ve_0$ in part (1) of Theorem \ref{T:nega} degrades as $T\to\infty$; on the other hand, for fixed data the smallness condition is automatically satisfied once $T$ is small enough. Part (1) should therefore be read as a quantified local result, in which the guaranteed existence time is controlled by the size of the data. This is a manifestation of the fact, discussed after Definition \ref{D:ac3}, that in the range $-1<a<0$ the problem has no exact scaling.
\end{remark}

\vskip 0.2in

\noindent {\bf Organization of the paper.}
The paper is organized as follows. In Section \ref{S:prep} we collect preliminary material from our work \cite{GS}, including properties of the Bessel operator, the associated weighted functional framework, dispersive estimates on the half-line, and a generalized Hardy--Littlewood--Sobolev lemma that plays a key role in the anomalous regime $-1<a<0$.

In Section \ref{S:propa} we study the Schr\"odinger propagator associated with homogeneous Neumann conditions in the upper half-space. We establish its basic structural properties, including symmetry, scaling, unitarity, and the corresponding Duhamel formula. We also discuss the distinction between mild and classical solutions and identify the admissible scaling relations underlying the Strichartz theory.

Section \ref{S:cpn} is devoted to the linear problem with nonhomogeneous Neumann boundary data. There we derive the fundamental representation formula \eqref{nicea1in}, equivalently written as \eqref{U0}, which separates the bulk evolution from the boundary forcing term. This formula constitutes the backbone of the paper and is the starting point for all subsequent dispersive estimates.

In Section \ref{S:stri} we establish bulk Strichartz estimates for the propagator. The analysis reveals a sharp distinction between the cases \(a\geq0\) and \(-1<a<0\): in the former one obtains estimates compatible with the natural scaling of the equation, whereas in the latter weighted estimates and mixed sum/intersection spaces become unavoidable.

Section \ref{S:bdry} is devoted to boundary Strichartz estimates. In contrast with the bulk theory, the boundary propagator exhibits a remarkably uniform dispersive behavior for all $a>-1$. Exploiting this structure, we prove estimates that are independent of the dichotomy governing the bulk dynamics.

The existence of traces on the boundary is addressed in Section \ref{S:cont}. Although the estimates of Section \ref{S:bdry} provide uniform control in the transverse variable $z$, they do not by themselves guarantee convergence to the boundary. The main result of this section establishes the existence and continuity of boundary traces in the appropriate functional setting.

In Section \ref{S:main} we combine the estimates obtained in the previous sections to prove Theorems \ref{T:main} and \ref{T:main2}. Finally, Section \ref{S:well} is devoted to the nonlinear problem. There we establish local and global well-posedness results in the mass-critical and subcritical regimes, completing the proof of Theorems \ref{T:well}
 and \ref{T:nega}. In the closing Section \ref{S:res} we use our Strichartz estimates to prove a theorem a la Tomas-Stein for the restriction of the Fourier-Hankel transform to a paraboloid.


\vskip 0.2in


\section{Preparatory material}\label{S:prep}

In this section we collect some material from our work \cite{GS} that will be used in the rest of the paper. We also introduce the relevant notation and function spaces.

We denote by 
\begin{equation}\label{Ba}
\mathscr B_z^{(a)} = \p_{zz}  + \frac az \p_z = z^{-a}\p_z(z^a\p_z),\ \ \ \ \ \ \ a>-1,
\end{equation}
the Bessel operator on the half-line $\R^+_z$, with its associated invariant measure $d\omega_a(z) = z^a dz$. As it is well-known,  $-\Ba$ is nonnegative on $L^2_a = L^2(\R^+, d\omega_a)$ when restricted to $C^\infty_0(\R^+)$, as one has
\begin{equation}\label{pos}
\sa -\Ba \psi,\psi\da = - \int_0^\infty \bar \psi\ \Ba \psi\  d\omega_a = \int_0^\infty |\p_z \psi|^2 d\omega_a.
\end{equation}
It is also known, see e.g. \cite{GS}, that 
\begin{equation}\label{sa}
\Ba\ \text{is self-adjoint in}\ L^2_a\ \ \Longleftrightarrow\  
a\in (-\infty,-1] \cup [3,\infty).
\end{equation}
For $a$ in the range in \eqref{sa} no boundary conditions are necessary. If instead $-1<a<3$, we consider the Cauchy problem 
\begin{equation}\label{cp0}
u_t - i \Ba u = F(z,t),\ \ \ \ \ u(z,0) = \vf(z),
\end{equation}
with the additional Neumann condition 
\begin{equation}\label{neu}
\underset{z\to 0^+}{\lim}\ z^a \p_z u = 0.
\end{equation}

The following fundamental solution for \eqref{cp0}, \eqref{neu} was found in \cite[Prop. 1.1]{GS}: 
\begin{equation}\label{SaS0}
S_a(z,\zeta,t)
=
\frac{e^{-i\,\operatorname{sgn}(t)\,\frac{(a+1)\pi}{4}}}{(2|t|)^{\frac{a+1}{2}}}
\left(\frac{z\zeta}{2|t|}\right)^{\frac{1-a}{2}}
J_{\frac{a-1}{2}}\!\left(\frac{z\zeta}{2|t|}\right)
\,e^{i \frac{z^2+\zeta^2}{4t}}, \quad t \neq 0.
\end{equation}
Notice that $S_a(z,\zeta,t)=S_a(\zeta,z,t)$. Moreover, $S_a(z,\zeta,t)$ satisfies the Neumann condition 
\begin{equation}\label{dzpazero}
\underset{z\to 0^+}{\lim} z^a \p_z S_a(z,\zeta,t) = 0,
\end{equation}
for any $\zeta>0$ and $t\not=0$. Keeping in mind that for any $\nu\in \R$ it holds: 
\begin{equation}\label{asy0}
\underset{z\to 0}{\lim}\ z^{\nu} J_{-\nu}(z) = \frac{2^{\nu}}{\G(-\nu+1)},
\end{equation}
we have for $a>-1$,
\begin{equation}\label{piccoloSa}
S_a(z,0,t):=\lim_{\zeta\to 0} S_a(z,\zeta,t)
=
\frac{2^{-a}}{\Gamma\!\left(\frac{a+1}{2}\right)}
\frac{e^{-i\,\operatorname{sgn}(t)\,\frac{(a+1)\pi}{4}}}{|t|^{\frac{a+1}{2}}}
\,e^{i \frac{z^2}{4t}}, \quad t \neq 0.
\end{equation}

Using the kernel \eqref{SaS0}, for every $t\in \R$ we define a linear operator $S_a(t)$ according to the rule
\begin{equation}\label{Sat}
S_a(t)\vf(z) =
\int_0^\infty S_a(z,\zeta,t) \vf(\zeta) d\omega_a(\zeta).
\end{equation}
Since $S_a(z,\zeta,t) = S_a(\zeta,z,t)$, letting $z\to 0^+$ in \eqref{Sat} and using \eqref{piccoloSa} we obtain the \emph{boundary evaluation} of $S_a(t)$,
\begin{equation}\label{Sat0}
S_a(t)\vf(0) = \int_0^\infty S_a(z,0,t)\, \vf(z)\, z^a dz,\ \ \ \ \ \ t\not= 0.
\end{equation}
Elementary as it is, the identity \eqref{Sat0} will be used repeatedly in Section \ref{S:bdry}, and we have therefore recorded it once and for all.

From \cite[Prop. 1.1]{GS} we have
\begin{proposition}\label{P:repa}
The solution of the Cauchy problem \eqref{cp0} admits the Duhamel representation 
\begin{equation}\label{kernel}
u(z,t) =  S_a(t)\vf(z) + \int_0^t S_a(t-\tau)(F(\cdot,\tau))(z) d\tau,
\end{equation}
\end{proposition}

The following is \cite[Prop. 3.1]{GS}.

\begin{proposition}\label{P:uni}
For every $t\in \R$, we have $S_a(t): L^2_a\to L^2_a$ unitarily, in the sense that for any $t\in \R$
\[
||S_a(t)\vf||_{L^2_a} = ||\vf||_{L^2_a}.
\]
Moreover, $S_a(t)$ is a group.
\end{proposition}

We also need the following. 
\begin{definition}\label{D:mn}
For $-1<a<0$, we define 
\[
k(z) = \min\{1,z^{\frac a2}\},\ \ \ \ \ \text{and}\ \ \ \ \ \ \max\{z^{\frac a2},z^a\} = k(z)^{-1} z^a.
\]  
\end{definition}

Notice that $0<k(z)\le 1$ for any $z>0$, and that moreover $k\equiv 1$ on the interval $[0,1]$.
We recall the following \cite[Prop. 3.4]{GS}.

\begin{proposition}[Dispersive estimate on $\R^+_z$]\label{P:dis}
Suppose that $a\ge 0$. There exists $C(a)>0$ such that for any $t\in \R\setminus\{0\}$ we have 
\begin{equation}\label{good}
||S_a(t) \vf||_{L^\infty(\R^+)}\le C(a) |t|^{-\frac{a+1}2} ||\vf||_{L^1_a}.
\end{equation}
If instead $-1<a<0$, then there exists $C(a)>0$ such that for any $t\in \R\setminus\{0\}$ we have 
\begin{align}\label{weight}
||S_a(t)\vf\ k||_{L^\infty(\R^+)} & \le C(a)\ \left\{|t|^{-\frac{a+1}2} + |t|^{-\frac{1}2}\right\} ||\vf\ k^{-1}||_{L^1_a(\R^+)}.
\end{align} 
\end{proposition}

Finally, we denote by
\begin{equation}\label{S}
S(x,y,t) =  (4\pi i t)^{-\frac d2} e^{i \frac{|x-y|^2}{4t}}
\end{equation}
the kernel of the Schr\"odinger propagator in $\Rd$.
This means that, given $\vf\in \mathscr S(\Rd)$, the function
\[
f(x,t) = S(t)\vf(x) := \int_{\Rd} S(x,y,t) \vf(y) dy,
\]
represents the unique solution of the Cauchy problem 
\[
\p_t f - i \Delta f = 0,\ \ \ \ \ \ f(x,0) = \vf(x).
\]
We recall the well-known dispersive estimate satisfied by $\{S(t)\}_{t\in \R}$, see \cite{GV, GV2, SuSu, Caze}.

\begin{proposition}\label{P:fundis}
For every $r\ge 2$ there exists $C(d,r)>0$ such that
\[
||S(t)\vf||_{L^r(\Rd)}\le C(d,r)|t|^{-d(\frac 12 - \frac 1r)}\ ||\vf||_{L^{r'}(\Rd)},
\]
for every $\vf\in L^{r'}(\Rd)$.
\end{proposition}

\vskip 0.2in


\subsection{Function spaces}\label{S:adapted} Given exponents $1\le q, r < \infty$, we will denote by $L^q_t L^r_x$ the Banach space of measurable functions $\Phi:\Rd\times \R\to \overline \C$ such that 
\[
||\Phi||_{L^q_t L^r_x} = \left(\int_\R ||\Phi(\cdot,t)||^q_{L^r_x} dt\right)^{\frac 1q} = \left(\int_\R \left(\int_{\Rd} |\Phi(x,t)|^r dx\right)^{\frac qr} dt\right)^{\frac 1q} < \infty.
\] 
If $1\le r, m< \infty$ we will denote by $L^m_{a,z}L^r_x $ the Banach space of measurable functions $f:\RN\to \overline{\C}$ such that 
\begin{equation}\label{Lrp}
||f||_{L^m_{a,z}L^r_x } := \left(\int_0^\infty \left(\int_{\Rd} |f(x,z)|^r dx\right)^{\frac mr} d\omega_a(z)\right)^{\frac 1m} = \left(\int_0^\infty ||f(\cdot,z)||^m_{L^r_x} d\omega_a(z)\right)^{\frac 1m} < \infty.
\end{equation}
When $r = m$ we write $L^r_a(\RN)$ instead of $L^r_{a,z}L^r_x $. In such case, the norm is given by
\[
||u||_{L^r_a(\RN)} = \left(\int_{\RN} |u(x,t)|^r d\omega_a(X)\right)^{\frac 1r},
\]
where, by slightly abusing the notation, we have let
\begin{equation}\label{me}
d\omega_a(X) = z^a dx dz = z^a dX.
\end{equation}
Since the context will be clear, there will be no risk of confusion between the two measures $d\omega_a(z)$ and $d\omega_a(X)$. 
In a similar way, given $1\le q, r, m< \infty$ we will denote by $L^m_{a,z} L^q_t L^r_x$ the Banach space of measurable functions $F:\RN\times \R\to \overline{\C}$ such that  
\begin{equation}\label{Lqrp}
||F||_{L^m_{a,z} L^q_t L^r_x} :=  \left(\int_0^\infty ||F(\cdot,z,\cdot)||^m_{L^q_t L^r_x} z^a dz\right)^{\frac 1m} < \infty.
\end{equation}
By standard tools one can show that, when none of the exponents equals $\infty$, the spaces $\mathscr S(\RN)$ and $\mathscr S(\RN\times \R)$ are respectively dense in $L^m_{a,z} L^r_x$ and $L^m_{a,z} L^q_t L^r_x$. 
When one or more of the exponents $q, r, m$ equals $\infty$, the corresponding norms must be replaced by essential suprema. When $m=\infty$, we understand that $L^\infty_{a,z} = L^\infty_z$. We will also indicate with 
\begin{equation}\label{Cz}
C^b_z L^q_t L^r_x =  C_z L^q_t L^r_x \cap L^\infty_z L^q_t L^r_x
\end{equation}
the space of the functions $F:\RN\times \R\to \overline{\C}$ such that $z\to F(\cdot,z,\cdot)$ is continuous and bounded in $z\ge 0$, with values in the Banach space $L^q_t L^r_x$. Such space is endowed with the norm $\underset{z\ge 0}{\sup} ||F(\cdot,z,\cdot)||_{L^q_t L^r_x} = ||F||_{L^\infty_z L^q_t L^r_x}$.  

In the proof of Theorem \ref{T:nega} we will need to restrict the Lebesgue spaces in \eqref{Lqrp} to a time interval $[0,T]$. For this purpose, we will use the notation $L^m_{a,z} L^q_T L^r_x$ to indicate the Banach space of measurable functions $F:\RN\times [0,T]\to \overline{\C}$ such that
\begin{equation}\label{LmixT}
||F||_{L^m_{a,z} L^q_T L^r_x} :=  \left(\int_0^\infty ||F(\cdot,z,\cdot)||^m_{L^q_T L^r_x} z^a dz\right)^{\frac 1m} < \infty,
\end{equation}
where we have let 
\[
||F(\cdot,z,\cdot)||_{L^q_T L^r_x} = \left(\int_0^T ||F(\cdot,z,t)||^q_{L^r_x} dt\right)^{\frac 1q},
\]
and continue to use the notation \eqref{Lqrp} for the mixed norms on the whole line $t\in \R$.

In the anomalous regime $-1<a<0$ the linear theory produces estimates in the sum space $L^{q+q_\infty}_t = L^q_t + L^{q_\infty}_t$, whereas the fixed point argument for the nonlinear problem is run in $L^q_T$. On a finite time interval the two are the same space. Since this is the reason why Theorem \ref{T:nega} is confined to a finite interval, and the source of the degeneration of its constants as $T\to \infty$, we record the elementary fact explicitly.

\begin{lemma}\label{L:sumT}
Let $0<T<\infty$ and $2<q<q_\infty<\infty$, and set $\gamma = \frac 1q - \frac 1{q_\infty}>0$. Then, for every Banach space $E$ and every $f$ measurable on $[0,T]$ with values in $E$, one has $L^{q+q_\infty}_T E = L^q_T E$, with
\begin{equation}\label{sumT}
||f||_{L^{q+q_\infty}_T E} \le ||f||_{L^{q}_T E} \le \max\{1,T^{\gamma}\}\ ||f||_{L^{q+q_\infty}_T E}.
\end{equation}
\end{lemma}

\begin{proof}
The first inequality follows by choosing the decomposition $f = f + 0$ in the definition of the sum norm. For the second one, let $f = g+h$ with $g\in L^q_T E$ and $h\in L^{q_\infty}_T E$. Since $q<q_\infty$, H\"older inequality on $[0,T]$ with exponents $\frac{q_\infty}{q}$ and $\left(\frac{q_\infty}{q}\right)'$ gives
\[
||h||_{L^q_T E} \le T^{\frac 1q - \frac 1{q_\infty}}\ ||h||_{L^{q_\infty}_T E} = T^{\gamma}\ ||h||_{L^{q_\infty}_T E}.
\]
Therefore
\[
||f||_{L^q_T E} \le ||g||_{L^q_T E} + T^{\gamma} ||h||_{L^{q_\infty}_T E} \le \max\{1,T^{\gamma}\}\left[||g||_{L^q_T E} + ||h||_{L^{q_\infty}_T E}\right],
\]
and it suffices to take the infimum over all such decompositions of $f$.
\end{proof}

\noindent
We stress that the constant in the right-hand side of \eqref{sumT} blows up as $T\to \infty$, and that no such equivalence holds on the whole line.

It will be useful to keep in mind how the norms \eqref{Lrp} and \eqref{Lqrp} respectively change under the following change of scale
\[
f_\la(X) = f(\la X) = f(\la x,\la z),\ \ \ \ \ \ \ \ \ F_\la(X,t) = F(\la X,\la^2 t).
\]  
We have
\begin{equation}\label{rescale}
||f_\la||_{L^m_{a,z} L^r_x} = \la^{- (\frac dr + \frac{a+1}m)}\ ||f||_{L^m_{a,z} L^r_x},\ \ \ \ \ \ \ ||F_\la||_{L^m_{a,z} L^q_t L^r_x } = \la^{- (\frac 2q + \frac dr + \frac{a+1}m)}\ ||F||_{L^m_{a,z} L^q_t L^r_x }.
\end{equation} 

Finally, we use the brief notation $L^{q_1+q_2}$ to indicate the
Banach space $L^{q_{1}} + L^{q_{2}}$ of measurable functions $f$ such that $f = f_{1} + f_{2}$, with $f_{1}\in L^{q_{1}}$ and
$f_{2} \in L^{q_{2}}$, endowed with the norm
\begin{align*}
\|f\|_{L^{q_{1}+q_{2}}}=\inf_{f=f_{1}+f_{2},f_{1}\in L^{q_{1}},f_{2} \in L^{q_{2}}}\|f_{1}\|_{L^{q_{1}}}+\|f_{2}\|_{L^{q_{2}}}.
\end{align*}
We also use the short notation $L^{p\cap q}$ to denote the Banach space $L^{p}\cap L^{q}$, endowed with the norm
\begin{align*}
\|f\|_{L^{p\cap q}}=\|f\|_{L^{p}}+\|f\|_{L^{q}}.
\end{align*}
We recall that $(L^{p+q})'=L^{p' \cap q'}$. We respectively indicate by
$L^m_{a,z} L^{q_1+q_2}_t L^r_x$ and $L^m_{a,z} L^{q_1\cap q_2}_t L^r_x $ the space of measurable functions $F:\RN\times \R\to \overline \C$ such that
\begin{align}\label{LmixmixT}
||F||_{L^m_{a,z} L^{q_1+q_2}_t L^r_x} & :=  \left(\int_0^\infty ||F(\cdot,z,\cdot)||^m_{L^{q_1+q_2}_t L^r_x} z^a dz\right)^{\frac 1m} < \infty
\\
||F||_{L^m_{a,z} L^{q_1\cap q_2}_t L^r_x} & :=  \left(\int_0^\infty ||F(\cdot,z,\cdot)||^m_{L^{q_1\cap q_2}_t L^r_x} z^a dz\right)^{\frac 1m} < \infty,
\notag
\end{align}
For measurable functions $U:\RN\times [0,T]\to \overline \C$, similarly to \eqref{LmixT} we will write $L^m_{a,z} L^{q_1+q_2}_T L^r_x$ and $L^m_{a,z} L^{q_1\cap q_2}_T L^r_x$.

\subsection{A generalized Hardy-Littlewood-Sobolev lemma}\label{S:genHLS}

 We close this section with a generalization of the Hardy-Littlewood-Sobolev theorem. For its proof see for instance \cite[Lemma 5.1]{BG}. 

\begin{lemma}\label{L:HLS}
Let $0<\gamma_{1}, \gamma_\infty <1$, $C_1,C_2>0$. Let $K:\mathbb{R}\to \mathbb{R}$ be such that
\begin{align*}
|K(t)|\le \begin{cases}
\frac{C_1}{|t|^{\gamma_{1}}} \ \ \ \ \text{if }\ |t|\le 1,
\\
\frac{C_2}{|t|^{\gamma_{\infty}}} \ \ \ \ \text{if}\ |t|\ge1.
\end{cases}
\end{align*}
Let $1<p<q<\infty$ , $1< p_{\infty}, q_{\infty} <\infty$ be exponents such that
\begin{align*}
\gamma_{1}=1+\frac{1}{q}-\frac{1}{p}, \ \ \ \text{and}\ \ \ \ \gamma_{\infty}\ge1+\frac{1}{q_{\infty}}-\frac{1}{p_{\infty}}.
\end{align*}  
Then one has
\begin{align}\label{meglio}
\|f\star K\|_{L^{q+q_\infty}}\le C \|f\|_{L^{p\cap p_{\infty}}}.
\end{align}
\end{lemma}

We will apply Lemma \ref{L:HLS} with the choices $p = q'$ and $p_\infty = q'_\infty$. This means that, given $\gamma_1, \gamma_\infty$, the exponents $q, q_\infty$ are determined by the equations
\begin{equation}\label{qq}
\frac 2{q} = \gamma_1,\ \ \ \ \ \ \frac 2{q_\infty} = \gamma_\infty.
\end{equation}

\vskip 0.2in

\subsection{Band-limited functions}\label{S:band}

We close this section with an elementary decay property that will be used repeatedly in Section \ref{S:cont}. Recall from \cite[(2.3)]{GS} the modified Hankel transform of order $\nu = \frac{a-1}2$ on $(\R^+, z^a dz)$,
\[
\mathcal H_{\nu}(\psi)(\la) = \int_0^\infty (\la z)^{-\nu} J_{\nu}(\la z)\, \psi(z)\, z^a dz.
\]
By \cite[Theorem 2.1 \& (2.7)]{GS}, the transform $\mathcal H_\nu$ is an involution, $\mathcal H_\nu\circ \mathcal H_\nu = \operatorname{Id}$, and an isometry of $L^2_{a,z}$ onto itself, and it diagonalises the Bessel operator $\Ba$ with the Neumann condition \eqref{neu}, see \cite[Theorem 2.2]{GS}. We say that a function $\psi$ on $\R^+$ is \emph{band-limited} if $\mathcal H_{\nu}(\psi)\in C_0^\infty((0,\infty))$.

\begin{lemma}[Decay of band-limited functions]\label{L:band}
Let $-1<a<1$, $\nu = \frac{a-1}2$, and let $\psi$ be band-limited, with $\hat\psi := \mathcal H_\nu(\psi)$ supported in $[\alpha,\beta]\subset(0,\infty)$. Then $\psi$ is a smooth function of $z^2$ on $[0,\infty)$ and, for every $k, M\in \N\cup\{0\}$, one has
\begin{equation}\label{banddecay}
|\p^k_z\psi(z)| \le C\, (1+z)^{-M},\ \ \ \ \ \ z\ge 0,
\end{equation}
with $C = C(a,k,M,\alpha,\beta,||\hat\psi||_{C^{M+k+1}})>0$. In particular, $\psi$ is the restriction to $[0,\infty)$ of a function in $\mathscr S(\R)$, and $\psi\, h\in L^1_{a,z}$ for every measurable $h$ of polynomial growth. Finally, the same conclusions hold for $\mathcal H_\nu(g)$ whenever $g\in C_0^\infty((0,\infty))$.
\end{lemma}

\begin{proof}
By the involution property, $\psi(z) = \int_0^\infty (z\la)^{-\nu}J_\nu(z\la)\, \hat\psi(\la)\, \la^a d\la$. Since $w\mapsto w^{-\nu}J_\nu(w)$ is an even entire function, and since $\hat\psi$ has compact support, $\psi$ is a smooth function of $z^2$, differentiation under the integral sign being justified; this also gives \eqref{banddecay} for $0\le z\le 1$. For $z\ge 1$ we write $J_\nu = \frac12 \left(H^{(1)}_\nu + H^{(2)}_\nu\right)$ and use the classical asymptotic representations of the Hankel functions,
\[
H^{(1,2)}_\nu(w) = \left(\frac{2}{\pi w}\right)^{\frac 12} e^{\pm i \left(w - \frac{\nu\pi}2 - \frac{\pi}4\right)}\, \omega^{\pm}_\nu(w),\ \ \ \ \ \ \left|\frac{d^j}{dw^j}\,\omega^{\pm}_\nu(w)\right| \le C(j,\nu,\alpha)\, w^{-j},\ \ \ w\ge \alpha,
\]
see e.g. \cite[8.451]{GR}; the expansions may be differentiated term by term. Since $\la\in [\alpha,\beta]$ forces $z\la\ge \alpha z$, we obtain
\[
\psi(z) = \sum_{\pm} c_{\pm}(\nu)\, z^{-\nu-\frac12} \int_0^\infty e^{\pm i z\la}\, m_\pm(z,\la)\, d\la,\ \ \ \ \ \ m_\pm(z,\la) := \la^{a-\nu-\frac 12}\, \omega^{\pm}_\nu(z\la)\, \hat\psi(\la),
\]
where $|c_\pm(\nu)| = (2/\pi)^{1/2}/2$. The amplitudes $m_\pm(z,\cdot)$ are supported in $[\alpha,\beta]$ and satisfy $|\p^j_\la m_\pm(z,\la)|\le C(j,a,\alpha,\beta,||\hat\psi||_{C^j})$ uniformly in $z\ge1$, because each $\la$-derivative falling on $\omega^\pm_\nu(z\la)$ produces $z\, (\omega^\pm_\nu)'(z\la) = O(\la^{-1})$. Integrating by parts $M$ times in $\la$, all boundary terms vanishing, we conclude $|\psi(z)|\le C\, z^{-\nu-\frac12-M}$, which is \eqref{banddecay} for $k=0$. For the derivatives, the identity $\frac{d}{dw}\left(w^{-\nu}J_\nu(w)\right) = - w^{-\nu}J_{\nu+1}(w)$ gives $\p_z\psi(z) = -\int_0^\infty \la\, (z\la)^{-\nu}J_{\nu+1}(z\la)\, \hat\psi(\la)\,\la^ad\la$, and the same argument applies with $\nu$ replaced by $\nu+1$; higher derivatives are handled inductively. The membership $\psi\, h \in L^1_{a,z}$ follows from \eqref{banddecay} with $M$ large, since $a>-1$. Finally, if $g\in C_0^\infty((0,\infty))$, then $f:= \mathcal H_\nu(g)$ satisfies $\mathcal H_\nu(f) = g\in C_0^\infty((0,\infty))$ by the involution property, so that $f$ is band-limited and the previous conclusions apply to it.

\end{proof}

\begin{remark}\label{R:bandt}
The constant in \eqref{banddecay} depends on $\hat\psi$ only through its support and finitely many of its derivatives. Since $\mathcal H_\nu\left(S_a(t)\psi\right)(\la) = e^{-it\la^2}\, \mathcal H_\nu(\psi)(\la)$, see \cite{GS} and \eqref{hank} below, and since the $C^{j}$ norms of $\la \mapsto e^{-it\la^2}\hat\psi(\la)$ are bounded uniformly for $t$ in compact subsets of $\R$, the evolution $S_a(t)\psi$ of a band-limited function is band-limited with the same spectral support, and satisfies \eqref{banddecay} with constants which are locally uniform in $t$. The same applies to the bulk evolution $\mathbb S_a(t) = S(t)\circ S_a(t)$ of \eqref{prop2} acting on products $\vf(x)\psi(z)$ with $\vf\in \mathscr S(\Rd)$ and $\psi$ band-limited.
\end{remark}


\vskip 0.3in

\section{The Schr\"odinger propagator with a zero Neumann condition}\label{S:propa}

Given $a>-1$ and $d\in \mathbb N$, we next consider the following generalization of the Cauchy problem \eqref{cp0}, \eqref{neu}:
\begin{equation}\label{eppn}
\begin{cases}
\p_t U - i \left\{\mathscr B_z^{(a)} U +  \Delta_x U\right\} = F(X,t),\ \ \ (X,t)\in \R_+^{d+1}\times \R
\\
\underset{z\to 0^+}{\lim} z^a \p_z U(X,t) = 0,\ \  \ \ U(X,0) = u_0(X).
\end{cases}
\end{equation}

It is well-known (see for instance formula (2.10) in \cite{BDGP} for the case of the heat equation) that the fundamental solution $\mathbb S_a(X,Y,t)$ of the Cauchy problem \eqref{eppn} is given by  
\begin{equation}\label{Sa}
\mathbb S_a(X,Y,t) =  S_a(z,\zeta,t) S(x,y,t).
\end{equation}
In \eqref{Sa} we have denoted by $S_a(z,\zeta,t)$ the fundamental solution \eqref{SaS0}, whereas $S(x,y,t)$ is the kernel \eqref{S}.
Note that \eqref{neu} implies the Neumann condition
\begin{equation}\label{Sazeron}
\underset{z\to 0^+}{\lim} z^a \p_z \mathbb S_a(X,Y,t) = 0.
\end{equation}

From \eqref{SaS0} and \eqref{S} we have
\begin{align}\label{grandeSaS0}
\mathbb S_a(X,Y,t)
= \frac{e^{-i\,\operatorname{sgn}(t)\,\frac{(d+a+1)\pi}{4}}}{2^{\frac{2d+a+1}{2}}\pi^{\frac d2}}  
\frac{e^{i\,\frac{z^2+\zeta^2+|x-y|^2}{4t}}}{|t|^{\frac{d+a+1}{2}}}\left(\frac{z\zeta}{2|t|}\right)^{\frac{1-a}{2}}
J_{\frac{a-1}{2}}\!\left(\frac{z\zeta}{2|t|}\right),
\qquad t\neq 0.
\end{align}

We notice the following consequence of \eqref{grandeSaS0}:
\begin{equation}\label{conj}
\overline{\mathbb S_a(X,Y,t)} = \mathbb S_a(X,Y,-t).
\end{equation}
This identity is justified by the fact that $J_{\frac{a-1}2}(\frac{z\zeta}{2t})$ is real-valued for $z, \zeta, t>0$. 
Using \eqref{piccoloSa}, we obtain
\begin{equation}\label{grandeSa}
\mathbb S_a(X,Y_0,t) :=\underset{\zeta\to 0^+}{\lim} \mathbb S_a(X,Y,t)
=
\frac{ e^{-i\,\operatorname{sgn}(t)\,\frac{(d+a+1)\pi}{4}}}{2^{d+a}\pi^{\frac d2}\Gamma\!\left(\frac{a+1}{2}\right)}
\frac{e^{i \frac{z^2+|x-y|^2}{4t}}}{|t|^{\frac{d+a+1}{2}}},
\qquad t\neq 0.
\end{equation}

In the special case $a=0$,  \eqref{grandeSaS0} takes a simple form, see \cite[Prop. 2.6]{GS}.

\begin{proposition}\label{P:azero}
When $a = 0$ we have
\[
\mathbb S_0(X,Y,t) = \frac{e^{-i\operatorname{sgn}(t)\frac{(d+1)\pi}{4}}}{2^{\frac{2d+3}{2}}\pi^{\frac{d+1}{2}}}
\frac{e^{i\frac{z^2+\zeta^2+|x-y|^2}{4t}}}{|t|^{\frac{d+1}{2}}}
\left(
e^{i\frac{z\zeta}{2|t|}} + e^{-i\frac{z\zeta}{2|t|}}
\right),
\qquad t\neq 0.
\]
\end{proposition}

\subsection{Mild versus classical solutions}

We recall that $\mathscr S(\RN)$ and $\mathscr S(\RN\times \R)$ respectively indicate the restrictions to the half-spaces $\RN$ and $\RN\times \R$ of the Schwartz functions in $\R^{d+1}_{X}$ and $\R^{d+1}_X\times \R_t$. We also respectively indicate with $\Sigma(\RN)$ and $\Sigma(\RN\times \R)$ the spaces of $C^\infty$ functions $f:\RN\to \C$, and $F:\RN\times \R\to \C$,  such that for any $N\in \mathbb N_0$ one has
\[
||f||_{N} = \underset{k+|\alpha| \le N}{\sup}\ \ \underset{X\in\R^+}{\sup} (1+|X|^2)^{\frac N2} \left|\left(\frac 1z \p_z\right)^k \p^\alpha_x f(X)\right|<\infty,
\]
and similarly
\[
||F||_{N} = \underset{k+|\alpha| +\ell \le N}{\sup}\ \ \underset{(x,t)\in\R^+\times \R}{\sup} (1+|X|^2+t^2)^{\frac p2} \left|\left(\frac 1z \p_z\right)^k \p^\alpha_x \p_t^\ell F(X,t)\right|<\infty.
\]
One has 
\[
\Sigma(\RN)\subset \mathscr S(\RN),\ \ \ \ \text{and}\ \ \ \  \Sigma(\RN\times \R)\subset \mathscr S(\RN\times \R),
\]
but the inclusions are strict. For instance, the function $u_0$ in \eqref{piccolou} belongs to $\mathscr S(\RN)$, but it does not belong to $\Sigma(\RN)$. The reason for this is that for any $u\in \Sigma(\RN)$ we must have for $a>-1$
\[
\underset{X\in \RN}{\sup} |z^a \p_z u(X)|\ \underset{z\to 0^+}{\longrightarrow}\ 0,
\] 
but it is easy to see that $z^a \p_z u_0(X)\to 1$ as $z\to 0^+$.

The classes $\Sigma(\RN)$ and $\Sigma(\RN\times \R)$  can be characterized as the restrictions to the half-spaces $\RN$ and $\RN\times \R$ of Schwartz functions which are \emph{even} in the variable $z$. Although we will not use them in the present work, we have mentioned them since they are the appropriate spaces for \emph{classical} solutions. For the notion of mild solution see \cite[Definition 2.3, p.106]{Pa}.

\begin{theorem}\label{T:Uzero}
Let $a>-1$. Assume that $u_0 \in \mathscr S(\RN)$ and $F \in \mathscr S(\RN\times \R)$.
There exists a unique mild solution to the problem \eqref{eppn} given by the Duhamel formula
\begin{align}\label{seppn}
U(x,t) & = \int_{\RN} \mathbb S_a(X,Y,t) u_0(Y) d\omega_a(Y) + \int_0^t \int_{\RN} \mathbb S_a(X,Y,t-\tau) F(Y,\tau) d\omega_a(Y) d\tau.
\end{align}
If $u_0\in \Sigma(\RN)$ and $F\in \Sigma(\RN\times \R)$, then $U$ is also a classical solution. 
\end{theorem}

\subsection{Conservation of mass} We define the Schr\"odinger propagator for the problem \eqref{eppn}, with forcing term $F(X,t) \equiv 0$, as the operator 
\begin{equation}\label{prop}
\mathbb S_a(t)u_0(X) := \int_{\RN} \mathbb S_a(X,Y,t) u_0(Y) d\omega_a(Y).
\end{equation} 
Keeping \eqref{Tstar},  \eqref{me} and \eqref{Sa} in mind, we see that we can write 
\begin{equation}\label{prop2}
\mathbb S_a(t)u_0(X) = S_a(t)(S(t)(u_0(\cdot,\cdot))(x))(z)= S(t)\big(S_a(t)(u_0(\cdot,\cdot))(z)\big)(x).
\end{equation}
The following result will be important in the sequel.

\begin{proposition} \label{P:Uni}
Let $a>-1$. For every $u_0\in L^2_a(\RN)$, and every $t\in \R$, we have
\[
||\mathbb S_a(t) u_0||_{L^2_a(\RN)} = ||u_0||_{L^2_a(\RN)}.
\]
\end{proposition}

\begin{proof}
Using \eqref{prop2}, Proposition \ref{P:uni}, the classical result 
\[
||S(t) \psi||_{L^2(\Rd)} = ||\psi||_{L^2(\Rd)},
\]
and Tonelli's theorem, we easily obtain 
\begin{align*}
||\mathbb S_a(t) u_0||^2_{L^2_a(\RN)} & = \int_0^\infty \int_{\Rd} |S(t)\big(S_a(t)(u_0(\cdot,\cdot))(z)\big)(x)|^2 dx z^a dz
\\
& = \int_0^\infty \int_{\Rd} |S_a(t)(u_0(x,\cdot))(z)\big)|^2 dx z^a dz =   \int_{\Rd} \int_0^\infty |S_a(t)(u_0(x,\cdot))(z)\big)|^2 z^a dz  dx 
\\
& = \int_{\Rd} \int_0^\infty |u_0(x,z)|^2 z^a dz  dx = ||u_0||^2_{L^2_a(\RN)}.
\end{align*}

\end{proof}

\vskip 0.2in

\subsection{Admissible triples}\label{S:ad}

We next discuss the question of the admissible exponents in our Strichartz estimates. Since in the problem \eqref{nozeroin} the space variables $x\in \Rd$ and $z>0$ play different roles, we use the mixed Lebesgue spaces $L^m_{a,z} L^q_t L^r_x $ with the norms \eqref{Lqrp}. 

Given an initial datum $u_0\in \mathscr S(\RN)$, and forcing term $F\equiv 0$, consider the unique mild solution $U$ to \eqref{eppn}. Comparing \eqref{seppn} with \eqref{Tstar}, we see that $U = \mathbb T^\star_a(u_0)$. We ask for which exponents $1\le q, r, m\le \infty$ does a priori estimate \eqref{apriori0} hold.
If for $\la>0$ we consider the scaled function $U_\la(X,t) = U(\la X,\la^2 t)$, then $U_\la$ also solves \eqref{eppn}, with initial datum $u_{0,\la}(X) = u_0(\la X)$. Inserting $U_\la$ and $u_{0,\la}$ in \eqref{apriori0}, and using \eqref{rescale}, we reach the conclusion that 
\[
\la^{- (\frac 2q + \frac dr + \frac{a+1}m)}\ ||\mathbb T^\star_a(u_0)||_{L^m_{a,z} L^q_t L^r_x } \le C \la^{- \frac{d+a+1}2} ||u_0||_{L^2_a(\RN)}.
\]
From the arbitrariness of $\la>0$, we see that a necessary condition for \eqref{apriori0} to hold for any $a>-1$ is that the triple $(q,r,m)$ satisfy the equation \eqref{ac3} in Definition \ref{D:ac3}. However, as we will see, \eqref{ac3} is sufficient only in the range $a\ge 0$. As we have mentioned already, the propagator does not obey scalings in the anomalous range $-1<a<0$.

Before closing we note that, when $a = 0$ and $d = 1$ condition \eqref{ac3} reads 
\[
\frac 2q + \frac 1r +\frac 1m = 1.
\]
This is the admissibility hypothesis (1.2) for the triple $(q,r,m)$ in the paper \cite{OST3}.

\vskip 0.3in

\section{Cauchy problem with nonzero Neumann condition}\label{S:cpn}

One difficulty with the problem \eqref{cp02} is to have a good representation formula of Duhamel type for the solution 
of the linear problem with nonzero Neumann condition:
\begin{equation}\label{nozero}
\begin{cases}
\p_t U - i \left\{\mathscr B_z^{(a)} U +  \Delta_x U\right\} = F(X,t),\ \ \ (X,t)\in \R_+^{d+1}\times (0,\infty),
\\
\underset{z\to 0^+}{\lim} z^a \p_z U(X,t) = \Phi(x,t),\ \  \ \ U(X,0) = u_0(X).
\end{cases}
\end{equation}

In this section we solve this question by carefully exploiting the non-admissible solution $u(z) = z^{1-a}$ of the Bessel operator (we mean by this that $u$ fails to satisfy \eqref{neu}). Our main result is the following.

\begin{theorem}\label{T:U}
Given $-1<a<1$, let $u_0\in \mathscr S(\RN)$, $F \in \mathscr S(\RN\times \R)$, $\Phi\in \mathscr S(\R^{d+1})$. There exists a unique mild solution to the problem \eqref{nozero}. Such function is represented by the formula
\begin{align}\label{nicea1}
U(X,t) & = \int_{\RN} \mathbb S_a(X,Y,t) u_0(Y)d\omega_a(Y) +  \int_0^t \int_{\RN} \mathbb S_a(X,Y,t-\tau) F(Y,\tau) d\omega_a(Y) d\tau
\\
& - i \int_0^{t} \int_{\Rd}  S_a(z,0,t-\tau)S(x,y,t-\tau) \Phi(y,\tau) dy d\tau,
\notag
\end{align}
where $S_a(z,0,t-\tau)$ is as in \eqref{piccoloSa}.
\end{theorem}

\begin{proof}
In a nutshell, the main idea is to reduce the Cauchy problem \eqref{nozero} to one with homogeneous Neumann condition such as \eqref{eppn}. The implementation of this plan will entail some delicate steps. 
Suppose that $U$ solves \eqref{nozero}. For $R>0$ we fix  a function $h_R\in C^\infty[0,\infty)$ such that:
\begin{equation}\label{hR}
\begin{cases}
 h_R\equiv 1,\ \  \text{for}\ 0\le z\le R,\ \ \text{and}\ \ h_R\equiv 0\ \  \text{for}\ z\ge 2 R,
 \\
|h_R^{(k)}(z)|\le C(k) R^{-k}, \ \ \ \forall k\in \mathbb N.
\end{cases}
\end{equation}
We introduce the auxiliary function
\begin{equation}\label{UV}
U_R(X,t) := U(X,t) - \frac{z^{1-a}}{1-a} h_R(z) \Phi(x,t).
\end{equation}
A computation shows that 
\begin{align}\label{ele}
 \p_t U_R - i \left\{\mathscr B_z^{(a)} U_R +  \Delta_x U_R\right\} & = \p_t U - i \left\{\mathscr B_z^{(a)} U +  \Delta_x U\right\} - \frac{z^{1-a}}{1-a} h_R(z) \left[\p_t \Phi - i  \Delta_x \Phi\right]
\notag
\\
& - i \Phi\left[\frac{z^{1-a}}{1-a}\mathscr B_z^{(a)}h_R(z) + 2 z^{-a} h_R'(z)\right],
\notag
\end{align} 
and we have moreover
\begin{equation}\label{par}
\begin{cases}
\underset{z\to 0^+}{\lim} z^a \p_z U_R(X,t) = \underset{z\to 0^+}{\lim} z^a \p_z U(X,t) - \Phi(x,t) = 0,
\\
\\
U_R(X,0) = U(X,0) - \frac{z^{1-a}}{1-a} h_R(z) \Phi(x,0) = u_0(X) -  \frac{z^{1-a}}{1-a} h_R(z) \Phi(x,0).
\end{cases}
\end{equation}
From \eqref{UV}-\eqref{par} it ensues that if $U$ solves the problem \eqref{nozero}, 
then for any $R>0$ the function $U_R$ solves \eqref{eppn}, with forcing term given by
\begin{equation}\label{data}
F(X,t) - \frac{z^{1-a}}{1-a} h_R(z)\left[\p_t \Phi - i  \Delta_x \Phi\right](x,t) - i k_R(z) \Phi(x,t),
\end{equation} 
where we have let
\[
k_R(z) := \frac{z^{1-a}}{1-a}\mathscr B_z^{(a)}h_R(z) + 2 z^{-a} h_R'(z).
\]
Note that by the support property \eqref{hR} of $h_R$ and the decay of its derivatives, we know that 
\begin{equation}\label{decay}
\operatorname{supp}\ k_R\subset [R,2R]\ \ \ \ \ \text{and}\ \ \ \ \ \underset{z\ge 0}{\sup}\ |k_R(z)|\le C R^{-(a+1)}.
\end{equation} 

From \eqref{seppn}, \eqref{data} we thus obtain
\begin{align}\label{dataeppn}
& U(X,t) - \frac{z^{1-a}}{1-a} h_R(z) \Phi(x,t) = U_R(X,t)
\\
&  = \int_{\RN} \mathbb S_a(X,Y,t) u_0(Y)d\omega_a(Y) - \int_{\RN} \mathbb S_a(X,Y,t) \frac{\zeta^{1-a}}{1-a} h_R(\zeta)\Phi(y,0) d\omega_a(Y)
\notag
\\
&  + \int_0^t \int_{\RN} \mathbb S_a(X,Y,t-\tau) F(Y,\tau) d\omega_a(Y) d\tau
\notag\\
& - \int_0^t \int_{\RN} \mathbb S_a(X,Y,t-\tau) \frac{\zeta^{1-a}}{1-a} h_R(\zeta)\left[\p_\tau \Phi - i  \Delta_y \Phi\right](y,\tau) d\omega_a(Y) d\tau
\notag
\\
& - i \int_0^t \int_{\RN} \mathbb S_a(X,Y,t-\tau) k_R(\zeta) \Phi(y,\tau) d\omega_a(Y) d\tau.
\notag
\end{align}
Rearranging terms in \eqref{dataeppn}, we obtain the following representation for $U$:
\begin{align}\label{U}
U(X,t) & =  \int_{\RN} \mathbb S_a(X,Y,t) u_0(Y)d\omega_a(Y) +  \int_0^t \int_{\RN} \mathbb S_a(X,Y,t-\tau) F(Y,\tau) d\omega_a(Y) d\tau
\\
& - \int_{\RN} \mathbb S_a(X,Y,t) \frac{\zeta^{1-a}}{1-a} h_R(\zeta)\Phi(y,0) d\omega_a(Y) + \frac{z^{1-a}}{1-a} h_R(z) \Phi(x,t) 
\notag
\\
& - i \int_0^t \int_{\RN} \mathbb S_a(X,Y,t-\tau) k_R(\zeta) \Phi(y,\tau) d\omega_a(Y) d\tau
\notag\\
& - \int_0^t \int_{\RN} \mathbb S_a(X,Y,t-\tau) \frac{\zeta^{1-a}}{1-a} h_R(\zeta)\left[\p_\tau \Phi - i  \Delta_y \Phi\right](y,\tau) d\omega_a(Y) d\tau 
\notag
\end{align}

Our next step is to carefully remove the derivatives from $\Phi$ in the last integral in the right-hand side of \eqref{U}: 
\begin{align}\label{remove}
& - \int_0^t \int_{\RN} \mathbb S_a(X,Y,t-\tau) \frac{\zeta^{1-a}}{1-a} h_R(\zeta)\left[\p_\tau \Phi - i  \Delta_y \Phi\right](y,\tau) d\omega_a(Y) d\tau
\\
& = - \underset{\ve\to 0^+}{\lim} \int_{\RN} \frac{\zeta^{1-a}}{1-a} h_R(\zeta) \left\{\int_0^{t-\ve} \mathbb S_a(X,Y,t-\tau) \p_\tau \Phi(y,\tau) d\tau\right\} d\omega_a(Y)
\notag
\\
& + i  \int_0^\infty \frac{\zeta^{1-a}}{1-a} h_R(\zeta) \int_0^t S_a(z,\zeta,t-\tau) \int_{\Rd} S(x,y,t-\tau)\Delta_y \Phi(y,\tau) d\omega_a(Y)d\tau
\notag
\\
& = - \underset{\ve\to 0^+}{\lim} \int_{\RN}  \mathbb S_a(X,Y,\ve) \frac{\zeta^{1-a}}{1-a} h_R(\zeta) \Phi(y,t-\ve)  d\omega_a(Y)
\notag
\\
& + \int_{\RN}  \mathbb S_a(X,Y,t) \frac{\zeta^{1-a}}{1-a} h_R(\zeta) \Phi(y,0) d\omega_a(Y)
\notag
\\
& - \int_{\RN} \frac{\zeta^{1-a}}{1-a} h_R(\zeta) \left\{\int_0^{t} \p_\tau \mathbb S_a(X,Y,t-\tau)  \Phi(y,\tau) d\tau\right\} d\omega_a(Y)
\notag
\\
& + i  \int_0^\infty \frac{\zeta^{1-a}}{1-a} h_R(\zeta) \int_0^t S_a(z,\zeta,t-\tau) \int_{\Rd} \Delta_y S(x,y,t-\tau) \Phi(y,\tau) d\omega_a(Y)d\tau,
\notag
\end{align}
where the integration by parts
\[
\int_{\Rd} S(x,y,t-\tau) \Delta_y  \Phi(y,\tau)dy = \int_{\Rd} \Delta_y S(x,y,t-\tau) \Phi(y,\tau)dy
\]
is justified by the decay at infinity of $y\to \Phi(y,t)$, along with its derivatives.
Using the fact that $\mathbb S_a(X,Y,t)$ is the fundamental solution for the Cauchy problem \eqref{eppn}, we have
\begin{align}\label{term}
- \underset{\ve\to 0^+}{\lim} \int_{\RN}  \mathbb S_a(X,Y,\ve) \frac{\zeta^{1-a}}{1-a} h_R(\zeta)\Phi(y,t-\ve)  d\omega_a(Y) = - \frac{z^{1-a}}{1-a} h_R(z)\Phi(x,t).
\end{align}
Moreover, 
\begin{align*}
- \p_\tau \mathbb S_a(X,Y,t-\tau) & = - i (\mathscr B_\zeta^{(a)}  + \Delta_y) \mathbb S_{a}(X,Y,t-\tau)
\\
& = - i S(x,y,t-\tau) \mathscr B_\zeta^{(a)} S_{a}(z,\zeta,t-\tau) - i S_{a}(z,\zeta,t-\tau)
\Delta_y S(x,y,t-\tau).
\end{align*}
Using this identity, we can convert the last two integrals in \eqref{remove} in the following way
\begin{align}\label{oterm}
& - \int_0^\infty \int_{\Rd} \frac{\zeta^{1-a}}{1-a} h_R(\zeta) \int_0^{t} \p_\tau \mathbb S_a(X,Y,t-\tau)  \Phi(y,\tau) d\tau d\omega_a(Y)
\\
& + i  \int_0^\infty \frac{\zeta^{1-a}}{1-a} h_R(\zeta) \int_0^t S_a(z,\zeta,t-\tau) \int_{\Rd} \Delta_y S(x,y,t-\tau) \Phi(y,\tau) d\omega_a(Y)d\tau
\notag
\\
& = - i \int_0^\infty \int_{\Rd} \frac{\zeta^{1-a}}{1-a} h_R(\zeta) \int_0^{t} S(x,y,t-\tau) \mathscr B_\zeta^{(a)} S_{a}(z,\zeta,t-\tau)  \Phi(y,\tau) d\tau d\omega_a(Y)
\notag
\\
& = - i \int_0^{t} \int_{\Rd}    S(x,y,t-\tau) \Phi(y,\tau)\int_0^\infty \frac{\zeta^{1-a}}{1-a} h_R(\zeta) \mathscr B_\zeta^{(a)} S_{a}(z,\zeta,t-\tau)   d\omega_a(Y) d\tau.
\notag
\end{align}
Substituting \eqref{term}, \eqref{oterm} in \eqref{remove}, we obtain
\begin{align}\label{remove2}
& - \int_0^t \int_{\RN} \mathbb S_a(X,Y,t-\tau) \frac{\zeta^{1-a}}{1-a} h_R(\zeta)\left[\p_\tau \Phi - i  \Delta_y \Phi\right](y,\tau) d\omega_a(Y) d\tau
\\
& = - \frac{z^{1-a}}{1-a} h_R(z)\Phi(x,t)  + \int_{\RN}  \mathbb S_a(X,Y,t) \frac{\zeta^{1-a}}{1-a} h_R(\zeta) \Phi(y,0) d\omega_a(Y)
\notag
\\
&  - i \int_0^{t} \int_{\Rd}    S(x,y,t-\tau) \Phi(y,\tau)\int_0^\infty \frac{\zeta^{1-a}}{1-a} h_R(\zeta) \mathscr B_\zeta^{(a)} S_{a}(z,\zeta,t-\tau)   d\omega_a(Y) d\tau.
\notag
\end{align}

If we now replace \eqref{remove2} in the identity \eqref{U}, 
we finally obtain for every $R>0$
\begin{align}\label{nicea}
& U(X,t) = \int_{\RN} \mathbb S_a(X,Y,t) u_0(Y)d\omega_a(Y) +  \int_0^t \int_{\RN} \mathbb S_a(X,Y,t-\tau) F(Y,\tau) d\omega_a(Y) d\tau
\\
& - i \int_0^{t} \int_{\Rd}    S(x,y,t-\tau) \Phi(y,\tau)\left(\int_0^\infty \frac{\zeta^{1-a}}{1-a}  h_R(\zeta) \mathscr B_\zeta^{(a)} S_{a}(z,\zeta,t-\tau) \zeta^a d\zeta\right)  dy d\tau
\notag
\\
&   - i \int_0^t \int_{\RN} \mathbb S_a(X,Y,t-\tau) k_R(\zeta) \Phi(y,\tau) d\omega_a(Y) d\tau.
\notag
\end{align}

We now make two crucial claims.

\medskip

\noindent{\textbf{Claim 1:}} For every $X\in \RN$ and $t>0$ we have
\begin{equation}\label{claim1}
\underset{R\to \infty}{\lim}\ \int_0^t \int_{\RN} \mathbb S_a(X,Y,t-\tau) k_R(\zeta) \Phi(y,\tau) d\omega_a(Y) d\tau = 0.
\end{equation}

\noindent \emph{Proof of Claim 1}:
By a standard change of variable, we write 
\begin{align}\label{I_R}
I_R & :=\int_0^t \int_{\RN} \mathbb S_a(X,Y,t-\tau) k_R(\zeta) \Phi(y,\tau) d\omega_a(Y) d\tau
\\
& = \int_0^t\int_{\Rd}
S(x,y,\tau) \Phi(y,t-\tau)
\Bigg(\int_0^\infty S_a(z,\zeta,\tau) k_R(\zeta) \zeta^a d\zeta\Bigg)
dy d\tau.
\notag
\end{align}
Since $\Phi\in \mathscr S(\R^{d+1})$, the function $\tau\to \int_{\Rd}
S(x,y,\tau) \Phi(y,t-\tau) dy$ is in $L^1(0,t)$. To establish \eqref{claim1} it thus suffices to prove that 
\begin{equation}\label{supJ}
\underset{\tau\in [0,t]}{\sup}\ L_R(\tau) := \underset{\tau\in [0,t]}{\sup} \int_0^\infty S_a(z,\zeta,\tau) k_R(\zeta) \zeta^a d\zeta\ \underset{R\to \infty}{\longrightarrow}\ 0.
\end{equation}
If for fixed $z>0$, we set
\begin{equation}\label{Atau}
A_\tau(\zeta) := \frac{e^{i \frac{z^2}{4\tau}}}{(2i\tau)^{\frac{a+1}2}}  \left(\frac{z\zeta}{2\tau}\right)^{\frac{1-a}2}J_{\frac{a-1}2}(\frac{z\zeta}{2\tau}),
\end{equation}
then from \eqref{SaS0} and the support property of $k_R$, we can write
\begin{equation}\label{JR}
L_R(\tau)=\int_R^{2R} F_\tau(\zeta) e^{i \frac{\zeta^2}{4\tau}} d\zeta,
\end{equation}
where we have let  
\begin{equation}\label{Ftau}
F_\tau(\zeta) = A_\tau(\zeta)\,k_R(\zeta)\,\zeta^a.
\end{equation}
By the asymptotic  of $S_a(z,\zeta,\tau)$ for large $\zeta$ in \cite{GS}, we know that there exists a constant $C(a,z)>0$ such that when $\zeta>>1$ we have
\begin{equation}\label{Ataubound}
|A_\tau(\zeta)| \le C(a,z) \tau^{-1/2} \zeta^{-\frac a2}.
\end{equation}
From \eqref{decay}, \eqref{Ftau} and \eqref{Ataubound}, we infer for $\zeta\in [R,2R]$ 
\begin{equation}\label{Ftau2}
|F_\tau(\zeta)| \le C(a,z) \tau^{-1/2} R^{-\frac a2} R^{-(a+1)} R^a \cong C(a,z) \tau^{-1/2} R^{-(\frac{a}2+1)}. 
\end{equation}
Substituting this information in \eqref{JR}, we conclude 
\[
|L_R| \le C(a,z) \tau^{-1/2} R^{-\frac{a}2},
\]
but this estimate does not provide decay in the regime $-1<a\le 0$.
 To achieve a better decay of $L_R$, we integrate by parts in the oscillatory integral \eqref{JR}.
 We thus have
\begin{align*}
L_R(\tau) & =\int_R^{2R} \frac{1}{\frac{i\zeta}{2\tau} } F_\tau(\zeta) \frac{d}{d\zeta} e^{i \frac{\zeta^2}{4\tau}} d\zeta = \left[F_\tau(\zeta)  \frac{e^{i \frac{\zeta^2}{4\tau}}}{\frac{i\zeta}{2\tau} }\right]_{\zeta = R}^{\zeta = 2R} - \int_R^{2R} \left[\frac{F'_\tau(\zeta)}{\frac{i\zeta}{2\tau}} - \frac{F_\tau(\zeta)}{\frac{i\zeta^2}{2\tau}}\right] e^{i \frac{\zeta^2}{4\tau}}d\zeta.
\end{align*}
 
 Using \eqref{decay} and \eqref{Ataubound}, we obtain for large $R>>1$
\[
\left|\left[F_\tau(\zeta)  \frac{e^{i \frac{\zeta^2}{4\tau}}}{\frac{i\zeta}{2\tau} }\right]_{\zeta = R}^{\zeta = 2R}\right| \le C \tau^{1/2} R^{-\frac a2} R^{-(a+2)}R^a \cong R^{-(\frac a2+2)}\ \underset{R\to \infty}{\longrightarrow}\ 0.
\]

We next estimate 
\begin{align*}
& \left|\int_R^{2R} F'_\tau(\zeta) \frac{e^{i \frac{\zeta^2}{4\tau}}}{\frac{i\zeta}{2\tau} } d\zeta\right|\le \frac{C\tau}{R} \int_R^{2R} |F'_\tau(\zeta)| d\zeta.
\end{align*}
To estimate $|F'_\tau(\zeta)|$ we set $\nu = \frac{a-1}2\in (-1,0)$, and let $G_\nu(z) = z^{-\nu} J_{\nu}(z)$. Then we can write
\[
A_\tau(\zeta) = \frac{e^{i \frac{z^2}{4\tau}}}{(2i\tau)^{\frac{a+1}2}}  G_\nu\left(\frac{z\zeta}{2\tau}\right).
\]
Since by \cite[(5.3.5) on p.103]{Le} we have 
\[
G'_\nu(z) = - z^{-\nu} J_{\nu+1}(z),
\]
we infer that, for $\zeta>>1$
\begin{align}\label{A'}
A'_\tau(\zeta) & = - \frac{e^{i \frac{z^2}{4\tau}}}{(2i\tau)^{\frac{a+1}2}} \frac{z}{2\tau} \left(\frac{z\zeta}{2\tau}\right)^{\frac{1-a}2} J_{\frac{a+1}2}\left(\frac{z\zeta}{2\tau}\right)
\\
& \cong {(2\tau)^{\frac{a+1}2}} \frac{z}{2\tau} \left(\frac{z\zeta}{2\tau}\right)^{\frac{1-a}2} \left(\frac{z\zeta}{2\tau}\right)^{-\frac 12}
\notag
\\
& \cong C(a,z) \tau^{a-\frac 12} \zeta^{-\frac{a}2}. 
\notag
\end{align} 
From \eqref{decay}, \eqref{Ftau}, \eqref{Ataubound} and \eqref{A'}, we now have for any $-1<a<1$ and $\zeta\in [R,2R]$
\begin{align*}
|F'_\tau(\zeta)| & \le C(z,a) \left[R^{-\frac a2} R^{-(a+1)} R^{a-1} + R^{-\frac a2}R^{-(a+2)} R^{a} + R^{-\frac a2} R^{-(a+1)} R^a\right]
\\
& \cong R^{-(1+\frac{a}{2})}\ \longrightarrow\ 0.
\end{align*}
Finally, we obtain from \eqref{Ftau}
\begin{align*}
& \left|\int_R^{2R} F_\tau(\zeta) \frac{e^{i \frac{\zeta^2}{4\tau}}}{\frac{i\zeta^2}{2\tau} } d\zeta\right|\le \frac{C\tau}{R^2} \int_R^{2R} |F_\tau(\zeta)| d\zeta\cong  C(a,z) R^{-(\frac{a}2+2)}\ \underset{R\to \infty}{\longrightarrow}\ 0.
\end{align*}
These estimates prove \eqref{supJ}, thus completing the proof of \eqref{claim1}. This establishes Claim 1.

\vskip 0.2in

\noindent \textbf{Claim 2:} For every $z>0$ and $t>0$ we have
\begin{equation}\label{claim2}
\lim_{R\to\infty}
\int_0^\infty
\frac{\zeta^{1-a}}{1-a}
h_R(\zeta)
\Ba S_a(z,\zeta,t) \zeta^a d\zeta =
S_a(z,0,t).
\end{equation}

\noindent \emph{Proof of Claim 2}:
Writing the Bessel operator in divergence form
\[
\Ba f
= \zeta^{-a}\partial_\zeta \big(\zeta^a \partial_\zeta f\big),
\]
and by the support property of $h_R$, integrating by parts we have
\begin{align*}
& \int_0^\infty
\frac{\zeta^{1-a}}{1-a}
h_R(\zeta)
\Ba S_a(z,\zeta,t) \zeta^a d\zeta = \underset{\ve\to 0^+}{\lim} \int_\ve^{2R}
\frac{\zeta^{1-a}}{1-a}
h_R(\zeta) \partial_\zeta \big(\zeta^a \partial_\zeta S_a(z,\zeta,t)\big) d\zeta
\\
& = - \underset{\ve\to 0^+}{\lim} \frac{\zeta^{1-a}}{1-a}
h_R(\zeta) \zeta^a \partial_\zeta S_a(z,\zeta,t) - \underset{\ve\to 0^+}{\lim} \int_\ve^{2R}
\zeta^a \p_\zeta\left(\frac{\zeta^{1-a}}{1-a}
h_R(\zeta)\right)  \partial_\zeta S_a(z,\zeta,t) d\zeta,
\end{align*}
where we have used $h_R(2R) = 0$. We now observe that 
\[
\underset{\ve\to 0^+}{\lim} \left[\frac{\zeta^{1-a}}{1-a}
h_R(\zeta) \zeta^a \partial_\zeta S_a(z,\zeta,t)\right]_{\zeta=\ve} =0.
\]
This follows from the fact that $1-a>0$ and from the zero Neumann condition \eqref{dzpazero} satisfied by $S_a(z,\zeta,t)$,
\[
\underset{\zeta\to 0^+}{\lim}  \zeta^a \partial_\zeta S_a(z,\zeta,t) =0.
\]
We thus find
\begin{align*}
& \int_0^\infty
\frac{\zeta^{1-a}}{1-a}
h_R(\zeta)
\Ba S_a(z,\zeta,t) \zeta^a d\zeta = - \underset{\ve\to 0^+}{\lim} \int_\ve^{2R} \left(h_R(\zeta) + \frac{\zeta h'_R(\zeta)}{1-a} \right) \partial_\zeta S_a(z,\zeta,t)d\zeta
\\
& = - \underset{\ve\to 0^+}{\lim} \left[\left(h_R(\zeta) + \frac{\zeta h'_R(\zeta)}{1-a} \right) S_a(z,\zeta,t)\right]_{\zeta = \ve}^{\zeta = 2R} + \underset{\ve\to 0^+}{\lim} \int_\ve^{2R} \p_\zeta \left(h_R(\zeta) + \frac{\zeta h'_R(\zeta)}{1-a} \right)  S_a(z,\zeta,t) d\zeta.
\end{align*}
Since $h_R(2R) = h'_R(2R) = 0$, the contribution of the upper end-point vanishes, whereas at the lower end-point, using that for every fixed $R>0$ we have $h_R(\zeta)\to 1$ as $\zeta\to 0^+$, we obtain the critical information that 
\[
- \underset{\ve\to 0^+}{\lim} \left[\left(h_R(\zeta) + \frac{\zeta h'_R(\zeta)}{1-a} \right) S_a(z,\zeta,t)\right]_{\zeta = \ve}^{\zeta = 2R} =\ +\, S_a(z,0,t).
\]
In conclusion, we have proved that, for every $R>0$ we have
\begin{align}\label{key}
& \int_0^\infty
\frac{\zeta^{1-a}}{1-a}
h_R(\zeta)
\Ba S_a(z,\zeta,t) \zeta^a d\zeta =  S_a(z,0,t) + H(R),
\end{align}
where
\begin{align}\label{HR}
H(R) := \int_R^{2R} \left(\frac{2-a}{1-a}  h'_R(\zeta) + \frac{\zeta h''_R(\zeta)}{1-a} \right)  S_a(z,\zeta,t) d\zeta
\end{align}
Since $|h'(\zeta)|\le C R^{-1}$, and since \eqref{SaS0} gives for $\zeta$ large
\begin{equation}\label{SaS02}
S_a(z,\zeta,t) \cong 
\frac{1}{t^{\frac{a+1}2}}  \left(\frac{z\zeta}{2t}\right)^{\frac{1-a}2}\left|J_{\frac{a-1}2}(\frac{z\zeta}{2t})\right|\cong \frac{1}{t^{\frac{a+1}2}}  \left(\frac{z\zeta}{2t}\right)^{-\frac{a}2}\cong t^{-\frac 12} \zeta^{-\frac{a}2},
\end{equation}
we obtain from \eqref{HR}
\[
|H(R)| \le C(a,z,t) R^{-\frac a2}.
\]
As for the Claim 1, this estimate provides no decay in the regime $-1<a\le 0$. To remediate this aspect, we proceed as in the proof of Claim 1, and write 
\[
S_a(z,\zeta,t) = A_t(\zeta)\ e^{i \frac{\zeta^2}{4t}} = \frac{A_t(\zeta)}{\frac{i\zeta}{2t}}  \frac{d}{d\zeta} e^{i \frac{\zeta^2}{4t}},
\]
where $A_t(\zeta)$ is as in \eqref{Atau}. We thus find from \eqref{HR}
\begin{align*}
H(R) & = \int_R^{2R} \left(\frac{2-a}{1-a}  h'_R(\zeta) + \frac{\zeta h''_R(\zeta)}{1-a} \right) \frac{A_t(\zeta)}{\frac{i\zeta}{2t}}  \frac{d}{d\zeta} e^{i \frac{\zeta^2}{4t}} d\zeta
\\
& = \left[\left(\frac{2-a}{1-a}  h'_R(\zeta) + \frac{\zeta h''_R(\zeta)}{1-a} \right) \frac{A_t(\zeta)}{\frac{i\zeta}{2t}} e^{i \frac{\zeta^2}{4t}}\right]_{\zeta = R}^{\zeta = 2R}
\\
& - \int_R^{2R} \frac{d}{d\zeta} \left[\left(\frac{2-a}{1-a}  h'_R(\zeta) + \frac{\zeta h''_R(\zeta)}{1-a} \right) \frac{A_t(\zeta)}{\frac{i\zeta}{2t}}\right]  e^{i \frac{\zeta^2}{4t}} d\zeta
\end{align*}
By \eqref{Ataubound} we see that for any $-1<a<1$, we have
\[
\left|\left[\left(\frac{2-a}{1-a}  h'_R(\zeta) + \frac{\zeta h''_R(\zeta)}{1-a} \right) \frac{A_t(\zeta)}{\frac{i\zeta}{2t}} e^{i \frac{\zeta^2}{4t}}\right]_{\zeta = R}^{\zeta = 2R}\right| \le \frac{C(a)}{R^{2+\frac a2}}\ \underset{R\to \infty}{\longrightarrow}\ 0.
\]
In a similar fashion, we infer that 
\[
\left|\int_R^{2R} \frac{d}{d\zeta} \left[\left(\frac{2-a}{1-a}  h'_R(\zeta) + \frac{\zeta h''_R(\zeta)}{1-a} \right) \frac{A_t(\zeta)}{\frac{i\zeta}{2t}}\right]  e^{i \frac{\zeta^2}{4t}} d\zeta\right| \le \frac{C(a)}{R^{2+\frac a2}}\ \underset{R\to \infty}{\longrightarrow}\ 0.
\]
The latter two estimates prove that 
\[
H(R)\ \underset{R\to \infty}{\longrightarrow}\ 0.
\]
In view of \eqref{key}, we have established \eqref{claim2}, thus finishing the proof of Claim 2.

\vskip 0.2in

With Claims 1 \& 2 in hands, we now return to the identity \eqref{nicea} and, letting $R\to \infty$, we finally obtain \eqref{nicea1}. This completes the proof of Theorem \ref{T:U}.

\end{proof}


\section{Strichartz estimates in the bulk}\label{S:stri}

In this section we establish various Strichartz estimates in the bulk space $\RN\times \R$ for the  problem \eqref{eppn}. We recall the  operators $\mathbb T_a^\star$ and $\mathbb D_a$, respectively defined by \eqref{Tstar} and \eqref{Da}, and denote by $\mathbb T_a$ the dual of $\mathbb T^\star_a$. The operators $\mathbb T_a^\star$ and $\mathbb T_a$, respectively convert initial values $u_0(X)$ into solutions $U(X,t)$ to \eqref{eppn} in the bulk $\RN\times \R$, and vice-versa. In Section \ref{S:bdry} we will study the boundary counterparts $\Theta_a^\star$, $\Theta_a$ of the operators $\mathbb T_a^\star$, $\mathbb T_a$.

\subsection{The generalized Ginibre-Velo operators}  

 For any $a>-1$ we consider the operator $\mathbb T_a^\star$ in \eqref{Tstar}, mapping functions $u_0:\RN\to \C$ into functions on  $\RN\times\R$.
Since according to Definition \ref{D:ac3} the triple $(\infty,2,2)$ is admissible both for $a\ge 0$ and $-1<a<0$, it is natural to ask whether \eqref{apriori0} holds for such triple. We have the following basic result.   
 
\begin{theorem}\label{T:Unistar}
For any $a>-1$ we have  $\mathbb T^\star_a: L^2_a(\RN)\to L^\infty_t L^2_a(\RN)$, and moreover
\[
||\mathbb T^\star_a u_0||_{L^\infty_t L^2_a(\RN)} = ||u_0||_{L^2_a(\RN)}.
\]
We also have for any $F\in L^1_t L^2_a(\RN)$
\[
||\mathbb D_a(F)||_{L^\infty_t L^2_a(\RN)} \le ||F||_{L^1_t L^2_a(\RN)}.
\]
\end{theorem}

\begin{proof}
The former statement about $\mathbb T^\star_a$ is just a reformulation of the conservation of mass in Proposition \ref{P:Uni} and there is nothing to prove. The statement about $\mathbb D_a$ is a direct consequence of the former and of the definition \eqref{Da}. We have in fact:
\[
\mathbb D_a(F)(X,t) = \int_0^t \mathbb S_a(t-\tau)(F(\cdot,\tau))(X) d\tau.
\]
This gives for every $t\in \R$
\begin{align*}
& ||\mathbb D_a(F)(\cdot,t)||_{L^2_a(\RN)}  \le \int_0^t ||\mathbb S_a(t-\tau)(F(\cdot,\tau))||_{L^2_a(\RN)} d\tau
\\
& = \int_0^t ||F(\cdot,\tau)||_{L^2_a(\RN)} d\tau \le ||F||_{L^1_t L^2_a(\RN)}.
\end{align*}

\end{proof}

Next, we represent the operator $\mathbb T_a: L^1_t L^2_a(\RN)\to L^2_a(\RN)$, whose adjoint is $\mathbb T_a^\star$.

\begin{lemma}\label{L:T}
For every function $F\in \snn$, we have
\begin{equation}\label{Ta}
\mathbb T_a(F)(X) = \int_\R \mathbb S_a(-t)(F(\cdot,t))(X) dt.
\end{equation}
\end{lemma}

\begin{proof}
Denote by $\sa\sa\cdot,\cdot \da\da$ the inner product in $L^2_t L^2_a(\RN)$. For every $u_0\in \mathscr S(\RN)$ we obtain from \eqref{conj} and \eqref{Tstar}  
\begin{align*}
\sa \mathbb T_a(F),u_0\da & = \sa\sa F,\mathbb T_a^\star(u_0)\da\da = \int_\R \int_{\RN} F(X,t) \overline{\mathbb T_a^\star(u_0)(X,t)} d\omega_a(X) dt
\\
& = \int_\R  \int_{\RN} F(X,t) \int_{\RN} \overline{\mathbb S_a(X,Y,t)}\ \overline{u_0(Y)} d\omega_a(Y) d\omega_a(X) dt
\\
& = \int_\R  \int_{\RN} F(X,t) \int_{\RN} \mathbb S_a(Y,X,-t)\ \overline{u_0(Y)} d\omega_a(Y) d\omega_a(X) dt
\\
& = \int_{\RN} \overline{u_0(Y)} \int_\R \int_{\RN} \mathbb S_a(Y,X,-t)F(X,t)d\omega_a(X) dt\ d\omega_a(Y)
\\
& = \int_{\RN} \overline{u_0(Y)} \int_\R \mathbb S_a(-t)(F(\cdot,t))(Y) dt\ d\omega_a(Y)
\end{align*}
This proves \eqref{Ta}.

\end{proof}

We next have the following result which establishes the part of Theorem \ref{T:main} relative to the bulk. We emphasize that, in view of the discussion of \underline{Case} $a\ge 0$ and \eqref{a2} in the introduction, the hypothesis $q>2$ in the statement of  Theorem \ref{T:Tstarbulk} forces $2\le r<\frac{2d}{d+a-1}$. Also, we assume $q<\infty$ since the case $q = \infty$ has already been dealt with by Theorem \ref{T:Unistar}.  

\begin{remark}\label{R:duality}
In this section, and again in Sections \ref{S:bdry} and \ref{S:well}, we use the following duality facts for the mixed-norm classes of Section \ref{S:adapted}: for $1<q, q_\infty, r<\infty$,
\[
\left(L^1_{a,z}L^{q'}_tL^{r'}_x\right)^\star = L^\infty_z L^q_t L^r_x,\ \ \ \ \ \ \left(L^1_{a,z}L^{q'\cap q'_\infty}_tL^{r'}_x\right)^\star = L^\infty_z L^{q+q_\infty}_t L^r_x,
\]
and the norm of an element of either dual space is attained as a supremum of pairings against functions in $\mathscr S(\RN\times\R)$ of unit norm. These identifications follow from the separability and $\sigma$-finiteness of the underlying measure spaces, the Radon-Nikod\'ym property of the reflexive spaces $L^q_tL^r_x$ and $L^{q+q_\infty}_tL^r_x$, and the density of $\mathscr S(\RN\times \R)$ in $L^1_{a,z}L^{q'}_tL^{r'}_x$ and in $L^1_{a,z}L^{q'\cap q'_\infty}_tL^{r'}_x$; see \cite{DU}.
\end{remark}

\begin{theorem}\label{T:Tstarbulk}
Let $a\ge 0$. Suppose that the triple $(q,r,\infty)$ is admissible and that $2<q<\infty$. For any $u_0\in \mathscr S(\RN)$ and $F\in \mathscr S(\RN\times \R)$ the unique mild solution to the problem \eqref{eppn} satisfies the following a priori estimate of Strichartz type
\begin{align}\label{bulk}
||U||_{L^\infty_z L^q_t L^r_x } \le C(d,a,r)\left[||u_0||_{L^2_a(\RN)} + ||F||_{L^1_{a,z} L^{q'}_t L^{r'}_x }\right].
\end{align}
\end{theorem}

\begin{proof}
According to Theorem \ref{T:Uzero} the solution $U$ is represented by 
\begin{equation}\label{U00}
U(X,t) = \mathbb T^\star_a(u_0)(X,t) + \mathbb D_a(F)(X,t),
\end{equation}
see also \eqref{U0}. The estimate  \eqref{bulk} will follow from the following a priori inequalities:
\begin{equation}\label{bulk00}
||\mathbb T^\star_a(u_0)||_{L^\infty_z L^q_t L^r_x } \le C(d,a,r)\ ||u_0||_{L^2_a(\RN)},
\end{equation}
and 
\begin{equation}\label{Dabound}
||\mathbb D_a(F)||_{L^\infty_z L^q_t L^r_x } \le C(d,a,r)\  ||F||_{L^1_{a,z} L^{q'}_t L^{r'}_x }.
\end{equation}

To prove \eqref{bulk00}, we fix $u_0\in \mathscr S(\RN)$. For any $F\in \mathscr S(\RN\times \R)$, the Cauchy-Schwarz inequality gives
\begin{align}\label{check}
\sa\sa \mathbb T^\star_a(u_0),F\da\da & = \sa u_0,\mathbb T_a(F)\da \le ||u_0||_{L_a^2(\RN)} ||\mathbb T_a(F)||_{L_a^2(\RN)}.
\end{align}
Next, H\"older inequality gives
\begin{align}\label{pop}
||\mathbb T_a(F)||^2_{L_a^2(\RN)} & = \sa \mathbb T_a(F),\mathbb T_a(F)\da = \sa\sa \mathbb T_a^\star \mathbb T_a(F),F\da\da \le ||\mathbb T_a^\star \mathbb T_a(F)||_{L^\infty_z L^q_t L^r_x } ||F||_{L^1_{a,z} L^{q'}_t L^{r'}_x }.
\end{align}
We now use \eqref{Tstar}, \eqref{Ta} to write
\begin{align}\label{stars}
\mathbb T_a^\star \mathbb T_a(F)(X,t) & = \int_\R \mathbb S_a(t-\tau)(F(\cdot,\tau))(X) d\tau.
\end{align}
This gives for every $z>0$ and $t\in \R$,
\begin{align}\label{opa}
& ||\mathbb T_a^\star \mathbb T_a(F)(\cdot,z,t)||_{L^r_x}  \le \int_{\R} ||\mathbb S_a(t-\tau)(F(\cdot,\tau))(\cdot,z)||_{L^r_x} d\tau 
\\
& = \int_{\R} ||S(t-\tau)(S_a(t-\tau)(F(\cdot,\tau))(\cdot,z)||_{L^r_x} d\tau
\notag
\\
& \le C(d,r) \int_{\R} \frac{||S_a(t-\tau)(F(\cdot,\tau))(\cdot,z)||_{L^{r'}_x}}{|t-\tau|^{d(\frac 12 - \frac 1r)}} d\tau, 
\notag
\end{align}
where in the last inequality we have used Proposition \ref{P:fundis}.
Keeping \eqref{Tstar} and \eqref{good} in mind, we have for any $z>0$
\begin{align*}
||S_a(t-\tau)(F(\cdot,\tau))(\cdot,z)||_{L^{r'}_x} \le \frac{C(a)}{|t-\tau|^{\frac{a+1}2}} \int_0^\infty ||F(\cdot,\zeta,\tau)||_{L^{r'}_x}\  \zeta^a d\zeta.
\end{align*}
Substituting this estimate in \eqref{opa}, we find for every $z>0$ and $t\in \R$,
\begin{align}\label{opa2}
& ||\mathbb T_a^\star \mathbb T_a(F)(\cdot,z,t)||_{L^r_x} \le C(d,a,r) \int_0^\infty \int_{\R} \frac{||F(\cdot,\zeta, \tau)||_{L^{r'}_x}}{|t-\tau|^{\frac{a+1}2+d(\frac 12 - \frac 1r)}} d\tau \zeta^a d\zeta
\\
& = C(d,a,r)\ \int_0^\infty I_\beta(h(\zeta,\cdot))(t) \zeta^a d\zeta,
\notag
\end{align}
where we have set
\begin{equation}\label{hzetatau}
h(\zeta, \tau) = ||F(\cdot,\zeta,\tau)||_{L^{r'}_x},
\end{equation}
and for a function $f$ on $\R$ we have denoted by
\begin{equation}\label{marcel}
I_\beta(f) = f\star |\cdot |^{-(1-\beta)},\ \ \ \ \ \ 0<\beta<1,   
\end{equation} 
the M. Riesz operator of fractional integration on $\R$, in which we are letting
\begin{equation}\label{unomenobeta}
1-\beta = \frac{a+1}2+d(\frac 12 - \frac 1r) = \frac{d+a+1}2 - \frac dr.
\end{equation}
By the Hardy-Littlewood-Sobolev theorem (see \cite{St}) we know that, if $0<\beta<1$, and $q$ is such that $1<q'<\beta^{-1}$, then we have
\begin{equation}\label{marcello}
I_\beta : L^{q'}_t\ \longrightarrow\ L^q_t,\ \ \ \ \text{provided that}\ \ \ \ \frac{1}{q'} - \frac 1q = 1- \frac 2q = \beta.
\end{equation}
Notice that, since $a\ge 0$ and $(q,r,\infty)$ is admissible, by \eqref{acinfty} we have
\[
1-\beta = \frac 2q = \frac{d+a+1}2 - \frac dr,
\]
which is exactly 
\eqref{unomenobeta}. Since we are assuming $2<q<\infty$, we have $0<\beta<1$, which is equivalent to 
\[
0<1- \frac{d+a+1}2 + \frac dr <1\ \Longleftrightarrow\ \frac{2d}{d+a+1}< r < \frac{2d}{d+a-1}.
\]
Note that, since $r\ge 2$, it is automatic that $\frac{2d}{d+a+1}< r$.
 
From all this and from \eqref{opa2} we conclude that
\begin{align}\label{opa3}
& ||\mathbb T_a^\star \mathbb T_a(F)(\cdot,z,\cdot)||_{L^q_t L^r_x} \le C(d,a,r) \int_0^\infty ||I_\beta(h(\zeta,\cdot)||_{L^{q}_t} \zeta^a d\zeta
\\
& \le C(d,a,r) \int_0^\infty ||h(\zeta,\cdot)||_{L^{q'}_t} \zeta^a d\zeta = C(d,a,r) \int_0^\infty ||F(\cdot,\zeta,\cdot)||_{L^{q'}_t L^{r'}_x} \zeta^a d\zeta
\notag\\
& = C(d,a,r) ||F||_{L^1_{a,z} L^{q'}_t L^{r'}_x}.
\notag
\end{align}
Taking the supremum in $z>0$ we obtain from \eqref{opa3} the following crucial estimate:
\begin{equation}\label{opa4}
||\mathbb T_a^\star \mathbb T_a(F)||_{L^\infty_z L^q_t L^r_x} \le C(d,a,r) ||F||_{L^1_{a,z} L^{q'}_t L^{r'}_x}.
\end{equation}
Substitution of \eqref{opa3} in \eqref{pop} gives for any $F\in \mathscr S(\RN\times \R)$,
\begin{align}\label{pop2}
||\mathbb T_a(F)||_{L_a^2(\RN)} & \le C(d,a,r)\ ||F||_{L^1_{a,z} L^{q'}_t L^{r'}_x }.
\end{align}
Inserting \eqref{pop2} into \eqref{check}, and taking the supremum on all $F\in \mathscr S(\RN\times \R)$,  completes the proof of \eqref{bulk00}.

To establish \eqref{Dabound}, we use \eqref{Da} which gives for $F\in \mathscr S(\RN\times \R)$:
\begin{equation}\label{Daa}
\mathbb D_a(F)(X,t) = \int_0^t \mathbb S_a(t-\tau)(F(\cdot,\tau))(X) d\tau.
\end{equation}
We next want to show that for any $t\in \R$ and $z>0$ we have
\begin{equation}\label{opa5}
||\mathbb D_a(F)(\cdot,z,t)||_{L^r_x} \le C(d,a,r) \int_0^\infty I_\beta(h(\zeta,\cdot))(t) \zeta^a d\zeta, 
\end{equation}
with $h(\zeta,\tau)$ as in \eqref{hzetatau}. Once \eqref{opa5} is proved, we proceed as in the proof of \eqref{opa3}, \eqref{opa4}, and obtain \eqref{Dabound}. To prove \eqref{opa5}, we use \eqref{Sa} and Proposition \ref{P:fundis} to find
\begin{equation}\label{opa6}
||\mathbb D_a(F)(\cdot,z,t)||_{L^r_x} \le C(d,a,r) \int_0^t \frac{||S_a(t-\tau)(F(\cdot,\cdot,\tau))(\cdot,z)||_{L^{r'}_x}}{|t-\tau|^{d(\frac 12 - \frac 1r)}} d\tau.
\end{equation}
Next, we apply \eqref{good} to find for every $x\in \Rd$ and $z>0$
\begin{align*}
& |S_a(t-\tau)(F(\cdot,\cdot,\tau))(x,z)| \le ||S_a(t-\tau)(F(\cdot,\cdot,\tau))(x,\cdot)||_{L^\infty_z} \le \frac{C(a)}{|t-\tau|^{\frac{a+1}2}} \int_0^\infty |F(x,\zeta,\tau)| \zeta^a d\zeta.
\end{align*}
Applying Minkowski's integral inequality, we obtain from this estimate
\begin{align*}
& ||S_a(t-\tau)(F(\cdot,\cdot,\tau))(\cdot,z)||_{L^{r'}_x} \le \frac{C(a)}{|t-\tau|^{\frac{a+1}2}} \int_0^\infty ||F(\cdot,\zeta,\tau)||_{L^{r'}_x} \zeta^a d\zeta
\end{align*}
Inserting this estimate in \eqref{opa6}, we find
\begin{align*}
& ||\mathbb D_a(F)(\cdot,z,t)||_{L^r_x} \le C(d,a,r) \int_0^t \int_0^\infty \frac{||F(\cdot,\zeta,\tau)||_{L^{r'}_x}}{|t-\tau|^{\frac{a+1}2 + d(\frac 12 - \frac 1r)}} \zeta^a d\zeta d\tau
\\
& \le C(d,a,r)  \int_0^\infty \int_\R \frac{||F(\cdot,\zeta,\tau)||_{L^{r'}_x}}{|t-\tau|^{\frac{a+1}2 + d(\frac 12 - \frac 1r)}} d\tau \zeta^a d\zeta,
\end{align*}
which establishes \eqref{opa5}. This completes the proof.

\end{proof}

We next discuss the regime $-1<a<0$. We have the following.

\begin{theorem}\label{T:bulknega}
Let $-1<a<0$ and suppose that the exponents $q, r$ and $q_\infty$ satisfy the conditions
\begin{equation}\label{uffa}
\frac 2{q} +\frac dr = \frac{d+1}2,\ \ \ \ \ \frac{2}{q_\infty} +\frac dr \le \frac{d+a+1}2.
\end{equation}
Moreover, we assume that $q>2$ and that $q_\infty<\infty$. For any $u_0\in \mathscr S(\RN)$ and $F\in \mathscr S(\RN\times \R)$ the unique mild solution to the problem \eqref{eppn} satisfies the following a priori estimate of Strichartz type
\begin{align}\label{bulk22}
||U k||_{L^\infty_z L^{q+q_\infty}_t L^r_x} \le C(d,a,r,q_\infty)\left[||u_0||_{L^2_a(\RN)} + ||F k^{-1}||_{L^1_{a,z} L^{q'\cap q'_\infty}_t L^{r'}_x }\right].
\end{align}
\end{theorem}

\begin{proof}
We proceed as in the beginning of the proof of Theorem \ref{T:Tstarbulk}. In view of \eqref{U00}, the estimate \eqref{bulk22} will follow from the two inequalities
\begin{equation}\label{bulksum}
||\mathbb T^\star_a(u_0)k||_{L^\infty_z L^{q+q_\infty}_t L^r_x } \le C(d,a,r,q_\infty)\ ||u_0||_{L^2_a(\RN)},
\end{equation}
and 
\begin{equation}\label{crucial}
||\mathbb D_a(F) k||_{L^\infty_z L^{q+q_\infty}_t L^r_x} \le C(d,a,r,q_\infty)\  ||F k^{-1}||_{L^1_{a,z} L^{q'\cap q'_\infty}_t L^{r'}_x},
\end{equation}
for any $u_0\in \mathscr S(\RN)$ and $F\in \mathscr S(\RN\times \R)$.

To establish \eqref{bulksum}, we use duality  to estimate
\begin{align}\label{popo}
||\mathbb T_a(F)||^2_{L_a^2(\RN)} & = \sa \mathbb T_a(F),\mathbb T_a(F)\da = \sa\sa \mathbb T_a^\star \mathbb T_a(F) k,F k^{-1}\da\da
\\
&  \le ||\mathbb T_a^\star \mathbb T_a(F) k||_{L^\infty_z L^{q+q_\infty}_t L^r_x } ||F k^{-1}||_{L^1_{a,z} L^{q' \cap q'_\infty}_t L^{r'}_x },
\notag
\end{align}
where $k(z)$ is the weight in Definition \ref{D:mn}. Using \eqref{stars}, we now obtain for any $X = (x,z)$ and $t\in \R$
\begin{align*}
& \mathbb T_a^\star \mathbb T_a(F)(X,t) k(z)  = \int_\R \mathbb S_a(t-\tau)(F(\cdot,\tau))(X) k(z) d\tau
\\
& = \int_\R S(t-\tau)(S_a(t-\tau)(F(\cdot,\tau))(\cdot,z) k(z))(x) d\tau,
\end{align*}
where in the second equality we have used \eqref{prop2}. As in \eqref{opa}, this gives for every $z>0$ and $t\in \R$
\begin{align}\label{opa7}
& ||\mathbb T_a^\star \mathbb T_a(F)(\cdot,z,t) k(z)||_{L^r_x} \le C(d,r) \int_\R \frac{||S_a(t-\tau)(F(\cdot,\tau))(\cdot,z) k(z))||_{L^{r'}_x}}{|t-\tau|^{d(\frac 12 - \frac 1r)}} d\tau,
\end{align}
where as before we have used Proposition \ref{P:fundis}.
By \eqref{weight} in Proposition \ref{P:dis}, for any $x\in \Rd$ and $z>0$ we find 
\begin{align*}
& |S_a(t-\tau)(F(\cdot,\cdot,\tau))(x,z) k(z)|\le ||S_a(t-\tau)(F(\cdot,\cdot,\tau))(x,\cdot) k||_{L^\infty_z}
\\
& \le C(a)\left\{|t-\tau|^{-\frac{a+1}2} + |t-\tau|^{-\frac{1}2}\right\} \int_0^\infty |F(x,\zeta,\tau)| k^{-1}(\zeta) \zeta^a d\zeta.
\end{align*}
Minkowski's integral inequality gives
\begin{align*}
& ||S_a(t-\tau)(F(\cdot,\tau))(\cdot,z) k(z))||_{L^{r'}_x} \le C(a)\left\{(|t-\tau|^{-\frac{a+1}2} + |t-\tau|^{-\frac{1}2})\right\} \int_0^\infty ||F(\cdot,\zeta,\tau)||_{L^{r'}_x} k^{-1}(\zeta) \zeta^a d\zeta.
\end{align*}
Inserting this estimate in \eqref{opa7}, gives
\begin{align}\label{opa7bis}
& ||\mathbb T_a^\star \mathbb T_a(F)(\cdot,z,t) k(z)||_{L^r_x}
\\
& \le C(d,a,r) \int_0^\infty  \int_\R \frac{||F(\cdot,\zeta,\tau)||_{L^{r'}_x} d\tau}{|t-\tau|^{d(\frac 12 - \frac 1r)}\left\{(|t-\tau|^{\frac{a+1}2} + |t-\tau|^{\frac{1}2})\right\}} k^{-1}(\zeta) \zeta^a d\zeta
\notag\\
& = C(d,a,r)\ \int_0^\infty (K\star h(\zeta,\cdot))(t) k^{-1}(\zeta) \zeta^a d\zeta,
\notag
\end{align}
where $h(\zeta,\tau)$ is as in \eqref{hzetatau}, and we have set
\begin{equation}\label{K}
K(t) := \left\{|t|^{-\frac{a+1}2} + |t|^{-\frac{1}2}\right\} |t|^{-d(\frac 12 - \frac 1r)}.
\end{equation}

We now apply Lemma \ref{L:HLS} with 
\[
\gamma_1 = \frac{1}2 + d(\frac 12 - \frac 1r) = \frac{d+1}2 - \frac dr = \frac 2q,\ \ \ \ \ \ \gamma_\infty = \frac{d+a+1}2 - \frac dr,
\]
where the last equality in the expression of $\gamma_1$ is justified by the first equation in \eqref{uffa}.
To apply Lemma \ref{L:HLS} we must know that $0<\gamma_1, \gamma_\infty <1$.  
Note that, since $-1<a<0$, we have $0<\frac{a+1}2 < \frac 12$, and so $\gamma_\infty<\gamma_1$. Since $\gamma_1 = \frac 2q$, the hypothesis $q>2$ implies $\gamma_1<1$, and therefore also $\gamma_\infty<1$. On the other hand, in view of the second inequality in \eqref{uffa}, the hypothesis $q_\infty <\infty$ implies $\gamma_\infty \ge  \frac{2}{q_\infty} >0$, and therefore we also have $\gamma_1>0$. From $\gamma_\infty<\gamma_1$ and \eqref{K}, we have: 
\[
K(t) \le C \begin{cases} 
|t|^{-\gamma_1},  \ \ \ \ |t|\le 1,
\\
|t|^{-\gamma_\infty},\ \ \ \ |t|\ge 1.
\end{cases}
\]
By Lemma \ref{L:HLS} we infer that there exists a constant $C=C(d,a,r,q_\infty)>0$ such that  
\begin{equation}\label{Kstar}
||K\star h(\zeta,\cdot)||_{L^{q+q_\infty}_t}\ \le C\ ||h||_{L^{q'\cap q'_\infty}_t} = C\ ||F(\cdot,\zeta,\cdot)||_{L^{q'\cap q'_\infty}_t L^{r'}_x},
\end{equation}
where in the last equality we have used \eqref{hzetatau}. Inserting \eqref{Kstar} into \eqref{opa7bis}, we find for any $z>0$
\[
||\mathbb T_a^\star \mathbb T_a(F)(\cdot,z,\cdot) k(z)||_{L^{q+q_\infty}_t L^r_x} \le C(d,a,r,q_\infty) \int_0^\infty ||F(\cdot,\zeta,\cdot)||_{L^{q'\cap q'_\infty}_t L^{r'}_x} k^{-1}(\zeta) \zeta^a d\zeta.
\]
By taking the supremum in $z>0$, we finally obtain the following basic estimate
\begin{align}\label{opa8}
||\mathbb T_a^\star \mathbb T_a(F) k||_{L^\infty_z L^{q+q_\infty}_t L^r_x} & \le C(d,a,r,q_\infty)  \int_0^\infty ||F(\cdot,\zeta,\cdot)||_{L^{q'\cap q'_\infty}_t L^{r'}_x} k^{-1}(\zeta) \zeta^a d\zeta
\\
& = C(d,a,r,q_\infty) ||F k^{-1}||_{L^1_{a,z} L^{q'\cap q'_\infty}_t L^{r'}_x}.
\notag
\end{align}

Returning to \eqref{popo}, and using \eqref{opa8},  we finally obtain 
\begin{align}\label{popopo}
||\mathbb T_a(F)||_{L_a^2(\RN)} & \le C(d,a,r,q_\infty) ||F k^{-1}||_{L^1_{a,z} L^{q' \cap q'_\infty}_t L^{r'}_x}.
\end{align}
From \eqref{check} and \eqref{popopo} we find
\begin{align}\label{check3}
\sa\sa \mathbb T^\star_a(u_0),F \da\da & \le\  C(d,a,r,q_\infty) ||u_0||_{L_a^2(\RN)}  ||F k^{-1}||_{L^1_{a,z} L^{q' \cap q'_\infty}_t L^{r'}_x}.
\end{align}
If in \eqref{check3} we now take the supremum on all measurable $F$ such that $||F k^{-1}||_{L^1_{a,z} L^{q' \cap q'_\infty}_t L^{r'}_x}\le 1$, we finally obtain \eqref{bulksum} from Remark \ref{R:duality}.

We next establish \eqref{crucial}. From \eqref{Daa} and Proposition \ref{P:fundis}, arguing as in \eqref{opa6}, for any $t\in \R$ and $z>0$ we find
\[
||\mathbb D_a(F)(\cdot,z,t) k(z)||_{L^r_x} \le C(d,r) \int_0^t \frac{||S_a(t-\tau)(F(\cdot,\cdot,\tau))(\cdot,z) k(z)||_{L^{r'}_x}}{|t-\tau|^{d(\frac 12 - \frac 1r)}}\, d\tau.
\]
By \eqref{weight} in Proposition \ref{P:dis}, for any $x\in \Rd$ and $z>0$,
\[
|S_a(t-\tau)(F(\cdot,\cdot,\tau))(x,z)\, k(z)| \le C(a)\left\{|t-\tau|^{-\frac{a+1}2} + |t-\tau|^{-\frac 12}\right\} \int_0^\infty |F(x,\zeta,\tau)|\, k^{-1}(\zeta)\, \zeta^a d\zeta,
\]
and therefore Minkowski's integral inequality gives
\[
||S_a(t-\tau)(F(\cdot,\tau))(\cdot,z)k(z)||_{L^{r'}_x} \le C(a)\left\{|t-\tau|^{-\frac{a+1}2} + |t-\tau|^{-\frac 12}\right\} \int_0^\infty ||F(\cdot,\zeta,\tau)||_{L^{r'}_x}\, k^{-1}(\zeta)\, \zeta^a d\zeta.
\]
Combining the latter two displays, with $K$ as in \eqref{K} and $h(\zeta,\tau)$ as in \eqref{hzetatau} we obtain, for every $z>0$ and $t\in \R$,
\[
||\mathbb D_a(F)(\cdot,z,t) k(z)||_{L^r_x} \le C(d,a,r) \int_0^\infty (K\star h(\zeta,\cdot))(t)\, k^{-1}(\zeta)\, \zeta^a d\zeta.
\]
The right-hand side is independent of $z$. Taking the $L^{q+q_\infty}_t$ norm, using Minkowski's integral inequality once more and then \eqref{Kstar} -- which rests on Lemma \ref{L:HLS}, applied with the same exponents $\gamma_1, \gamma_\infty$ as above -- we conclude
\[
||\mathbb D_a(F)(\cdot,z,\cdot)k(z)||_{L^{q+q_\infty}_t L^r_x} \le C(d,a,r,q_\infty) \int_0^\infty ||F(\cdot,\zeta,\cdot)||_{L^{q'\cap q'_\infty}_t L^{r'}_x}\, k^{-1}(\zeta)\, \zeta^a d\zeta = C\, ||F k^{-1}||_{L^1_{a,z} L^{q'\cap q'_\infty}_t L^{r'}_x},
\]
and \eqref{crucial} follows by taking the supremum over $z>0$.

\end{proof}


\section{Boundary Strichartz estimates}\label{S:bdry}

In this section we assume that $-1<a<1$ and establish Strichartz estimates for the following problem
\begin{equation}\label{nozero2}
\begin{cases}
\p_t V - i \left\{\mathscr B_z^{(a)} V +  \Delta_x V\right\} = 0,\ \ \ (X,t)\in \R_+^{d+1}\times (0,\infty)
\\
\underset{z\to 0^+}{\lim} z^a \p_z V(X,t) = \Phi(x,t),\ \  \ \ V(X,0) = 0,
\end{cases}
\end{equation}
which corresponds to taking $u_0(X) \equiv 0$, $F(X,t) \equiv 0$ in problem \eqref{nozero}. According to Theorem \ref{T:U}, see also \eqref{Thetastar},
 the mild solution of \eqref{nozero2} is given by the integral operator
\begin{align}\label{V}
V(X,t) & =  \Theta_a^\star(\Phi)(X,t) := - i \int_0^{t} \int_{\Rd}  S_a(z,0,t-\tau)S(x,y,t-\tau) \Phi(y,\tau) dy d\tau
\\
& = - i \int_0^{t} S_a(z,0,t-\tau) \int_{\Rd}  S(x,y,t-\tau) \Phi(y,\tau) dy d\tau,
\notag
\end{align}
where $S_a(z,0,t)$ is as in \eqref{piccoloSa}, and where the constant $-i$ is the one already fixed in \eqref{Thetastar}, see Remark \ref{R:signcheck}. Since every estimate established in this section is one on moduli, this unimodular factor plays no role in the proofs; we nonetheless carry it explicitly, so that the normalisation of the boundary Duhamel operator is the same throughout the paper.
One notable aspect of the results in this section is that, because of the special form \eqref{piccoloSa}, \eqref{V} of the boundary propagator, we will not be forced to distinguish between the regimes $1>a\ge 0$ and $-1<a<0$, as we did for the Strichartz estimates in the bulk in Section \ref{S:stri}. 
It is useful to obtain a representation of the operator $\Theta_a$ of which $\Theta^\star_a$ is the adjoint. 

\begin{proposition}\label{P:rep}
For every $V\in \mathscr S(\RN\times \R)$ we have
\begin{align}\label{T}
\Theta_a(V)(y,\tau) & = i \int_\tau^{\infty} \int_0^\infty   S_a(z,0,\tau-\sigma) \int_{\Rd}   S(y,x,\tau-\sigma) V(x,z,\sigma) dx d\sigma  z^a dz
\\
& = i \int_\tau^{\infty} \int_{\RN} \mathbb S_a(Y_0,X,\tau-\sigma) V(X,\sigma) d\omega_a(X) d\sigma
\notag
\\
& = i \int_\tau^{\infty} \mathbb S_a(\tau-\sigma)(V(\cdot,\sigma))(Y_0) d\sigma.
\notag
\end{align}
\end{proposition}

\begin{proof}
Keeping in mind that, by \eqref{V}, the operator $\Theta_a^\star$ carries the factor $-i$, so that its complex conjugate contributes the factor $\overline{-i} = i$, we have
\begin{align*}
& \sa \Theta_a(V),\Phi\da = \sa\sa V,\Theta_a^\star(\Phi)\da\da = \int_0^\infty \int_{\RN} V(X',t) \overline{\Theta_a^\star(\Phi)(X',t)} d\omega_a(X') dt
\\
& = i \int_0^\infty \int_{\RN} V(X',t) \int_0^{t} \int_{\Rd}  \overline{S_a(z,0,t-\tau)}\ \overline{S(x',y,t-\tau)}\ \overline{\Phi(y,\tau)} dy d\tau d\omega_a(X') dt
\\
& = i \int_0^\infty \int_{\Rd} \overline{\Phi(y,\tau)}  \int_0^\infty  \int_\tau^{\infty} S_a(z,0,\tau-t) \int_{\Rd}   S(y,x',\tau-t) V(x',z,t) dx' dt  z^a dz dy d\tau
\\
& = \int_0^\infty \int_{\Rd} \overline{\Phi(y,\tau)} \left(i \int_\tau^{\infty} \int_0^\infty   S_a(z,0,\tau-\sigma) \int_{\Rd}   S(y,x',\tau-\sigma) V(x',z,\sigma) dx' d\sigma  z^a dz\right) dy d\tau,
\end{align*}
where we have used that $\overline{S_a(z,0,t-\tau)} = S_a(z,0,\tau-t)$ and $\overline{S(x',y,t-\tau)} = S(y,x',\tau-t)$. The passage from the first to the third line of \eqref{T} is the boundary evaluation \eqref{Sat0}, applied in the $z$ variable, together with the factorisation $\mathbb S_a(t) = S(t)\circ S_a(t)$ of \eqref{prop2}.
This identity proves \eqref{T}.

\end{proof}

Before proving the next result, we recall two classical integrals, see, e.g., \cite[formulas 3.761.4 \& 3.761.9, p.420-1]{GR}.

\begin{lemma}\label{L:cis}
Let $b>0$ and $0<\Re \mu<1$. Then, as improper Riemann integrals,
\begin{equation}\label{sincosy}
\int_0^\infty y^{\mu-1} \sin(by) dy = \frac{\G(\mu)}{b^\mu} \sin \frac{\mu \pi}{2},\ \ \ \ \ \ \int_0^\infty y^{\mu-1} \cos(by) dy = \frac{\G(\mu)}{b^\mu} \cos \frac{\mu \pi}{2}.
\end{equation}
Therefore,
\begin{equation}\label{ei}
\int_0^\infty y^{\mu-1} e^{\pm i by} dy = \frac{\G(\mu)}{b^\mu} e^{\pm i\frac{\mu \pi}{2}}.
\end{equation} 
\end{lemma}

With the aid of Lemma \ref{L:cis}, we obtain the following auxiliary result.

\begin{lemma}\label{L:piccoloSa}
Let $-1<a<1$, and suppose that $\tau$ and $\sigma$ both lie strictly between $0$ and $t$, i.e., that either $0<\tau,\sigma<t$, or $t<\tau,\sigma<0$. Then, as an improper (Fresnel) integral,
\begin{align}\label{ohoh}
& \int_0^\infty S_a(z,0,t-\tau)\overline{S_a(z,0,t-\sigma)} z^a dz =
\frac{2^{-a}}{\Gamma\!\left(\frac{a+1}{2}\right)}
\,\frac{\exp\!\left(i\,\frac{(a+1)\pi}{4}\,\operatorname{sgn}(\tau-\sigma)\right)}{|\tau-\sigma|^{\frac{a+1}{2}}}.
\end{align}
\end{lemma}

\begin{proof}
The hypothesis guarantees that $t-\tau$ and $t-\sigma$ are both nonzero and of the same sign, so that in the product $S_a(z,0,t-\tau)\overline{S_a(z,0,t-\sigma)}$ the two constant phases $e^{\mp i\operatorname{sgn}(t-\tau)\frac{(a+1)\pi}4}$, $e^{\pm i \operatorname{sgn}(t-\sigma)\frac{(a+1)\pi}4}$ of \eqref{piccoloSa} cancel, and that $(t-\tau)(t-\sigma)>0$. Using \eqref{piccoloSa}, we thus find
\begin{align*}
& \int_0^\infty S_a(z,0,t-\tau)\overline{S_a(z,0,t-\sigma)} z^a dz
\\
& = \frac{2^{-2a}}{\G(\frac{a+1}2)^2|t-\tau|^{\frac{a+1}2}|t-\sigma|^{\frac{a+1}2}} \int_0^\infty  e^{i z^2\left[\frac{1}{4(t-\tau)} - \frac{1}{4(t-\sigma)}\right]} z^{a+1} \frac{dz}z
\\
&  =  \frac{2^{-2a}}{\G(\frac{a+1}2)^2|t-\tau|^{\frac{a+1}2}|t-\sigma|^{\frac{a+1}2}} \int_0^\infty  e^{i z^2\left[\frac{\tau-\sigma}{4(t-\tau)(t-\sigma)}\right]} z^{a+1} \frac{dz}z.
\end{align*}
Suppose $\tau-\sigma>0$. Then the change of variable $\zeta = z^2 \frac{\tau-\sigma}{4(t-\tau)(t-\sigma)}$, gives
\begin{align*}
& \int_0^\infty  e^{i z^2\left[\frac{\tau-\sigma}{4(t-\tau)(t-\sigma)}\right]} z^{a+1} \frac{dz}z = \frac{2^{a}|t-\tau|^{\frac{a+1}2}|t-\sigma|^{\frac{a+1}2}}{(\tau-\sigma)^{\frac{a+1}2}}\int_0^\infty  e^{i \zeta} \zeta^{\frac{a+1}2-1}d\zeta
\\
& = \frac{2^{a}|t-\tau|^{\frac{a+1}2}|t-\sigma|^{\frac{a+1}2}}{(\tau-\sigma)^{\frac{a+1}2}} \G(\frac{a+1}2) e^{i\frac{(a+1)\pi}{4}},
\end{align*}
where in the last equality we have applied \eqref{ei} in Lemma \ref{L:cis} with $\mu = \frac{a+1}2$ (the change of variable $\zeta = z^2\frac{\tau-\sigma}{4(t-\tau)(t-\sigma)}$ produces the factor $\frac 12$, since $z^a dz = \frac 12 \zeta^{\frac{a+1}2 -1}\left(\frac{4(t-\tau)(t-\sigma)}{\tau-\sigma}\right)^{\frac{a+1}{2}} d\zeta$). Notice that, since we are assuming $-1<a<1$, we have $0<\mu<1$. If instead $\tau-\sigma<0$, then we make the change of variable $\zeta = z^2 \frac{\sigma-\tau}{4(t-\tau)(t-\sigma)}$, obtaining 
\begin{align*}
& \int_0^\infty  e^{i z^2\left[\frac{\tau-\sigma}{4(t-\tau)(t-\sigma)}\right]} z^{a+1} \frac{dz}z =  \int_0^\infty  e^{- i z^2\left[\frac{\sigma-\tau}{4(t-\tau)(t-\sigma)}\right]} z^{a+1} \frac{dz}z
\\
&  = \frac{2^{a}|t-\tau|^{\frac{a+1}2}|t-\sigma|^{\frac{a+1}2}}{(\sigma-\tau)^{\frac{a+1}2}} \G(\frac{a+1}2) e^{-i\frac{(a+1)\pi}{4}},
\end{align*}
again by \eqref{ei}.
Combining these results, we obtain \eqref{ohoh}.

\end{proof}

We next establish the main result of this section. Before stating it, we stress that the assumption \eqref{bdryad} coincides with the admissibility condition \eqref{acinfty} only when $a\ge 0$. 

\begin{theorem}\label{T:V}
Let $-1<a<1$ and suppose that $q, r$
satisfy
\begin{equation}\label{bdryad}
\frac 2q + \frac dr = \frac{d+a+1}2,
\end{equation}
with $r\ge 2$ and 
 $q>2$. Then there exists $C(d,a,r)>0$ such that for any $\Phi\in \mathscr S(\R^{d+1})$ one has
\begin{equation}\label{ok}
||\Theta_a^\star(\Phi)||_{L^\infty_t L^2_a(\RN)} \le C(d,a,r)\ ||\Phi||_{L^{q'}_t L^{r'}_x}.
\end{equation}
Furthermore, we also have
\begin{equation}\label{okk}
||\Theta_a^\star(\Phi)||_{L^\infty_{z} L^q_t L^r_x } \le C(d,a,r)\ ||\Phi||_{L^{q'}_t L^{r'}_x}.
\end{equation}
\end{theorem}

\begin{proof}

We first establish \eqref{ok}. Since, by \eqref{piccoloSa}, the modulus of $S_a(z,0,s)$ does not depend on $z$, the double $z$-integral which arises when the square $|\Theta_a^\star(\Phi)(X,t)|^2$ is expanded is \emph{not} absolutely convergent, and the identity \eqref{ohoh} itself holds only as an improper integral. To justify the computation we insert a Gaussian regularization. For $\ve>0$ we consider
\begin{align}\label{thetareg}
Q_\ve(t) := \int_{\Rd} \int_0^\infty e^{-\ve z^2}\, |\Theta_a^\star(\Phi)(X,t)|^2\, z^a\, dz\, dx.
\end{align}
By \eqref{piccoloSa}, Cauchy-Schwarz in $x$ and the unitarity of $S(\cdot)$ on $L^2(\Rd)$, for every fixed $z>0$
\begin{align*}
& \int_0^t \int_0^t \int_{\Rd} |S_a(z,0,t-\tau)|\, |S_a(z,0,t-\sigma)|\, |S(t-\tau)(\Phi(\cdot,\tau))(x)|\, |S(t-\sigma)(\Phi(\cdot,\sigma))(x)|\, dx\, d\tau d\sigma
\\
& \le C(a) \left(\int_0^t \frac{||\Phi(\cdot,\tau)||_{L^2_x}}{|t-\tau|^{\frac{a+1}2}}\, d\tau\right)^2 < \infty,
\end{align*}
since $\frac{a+1}2<1$ and $\Phi\in \mathscr S(\R^{d+1})$. As $\int_0^\infty e^{-\ve z^2} z^a dz<\infty$, the five-fold integral obtained by expanding the square in \eqref{thetareg} is absolutely convergent, and Fubini's theorem, together with the group property of the Schr\"odinger propagator $S(t)$, gives
\begin{align*}
Q_\ve(t)
& = \int_0^t \int_0^t \left(\int_0^\infty e^{-\ve z^2}\, S_a(z,0,t-\tau)\overline{S_a(z,0,t-\sigma)}\, z^a dz\right) \int_{\Rd} S(\sigma-\tau)(\Phi(\cdot,\tau))(x) \overline{\Phi(x,\sigma)}\, dx\, d\tau d\sigma.
\end{align*}
By \eqref{piccoloSa}, and since for $0<\tau, \sigma<t$ the numbers $t-\tau$ and $t-\sigma$ have the same sign, so that the two constant phases $e^{-i\operatorname{sgn}(t-\tau)\frac{(a+1)\pi}4}$, $e^{+i\operatorname{sgn}(t-\sigma)\frac{(a+1)\pi}4}$ of \eqref{piccoloSa} cancel -- exactly as in the proof of Lemma \ref{L:piccoloSa} -- the inner integral equals
\[
\frac{2^{-2a}}{\Gamma\!\left(\frac{a+1}2\right)^2 \left[(t-\tau)(t-\sigma)\right]^{\frac{a+1}2}}\, \int_0^\infty e^{-(\ve - i \omega) z^2}\, z^{a+1}\, \frac{dz}z,\ \ \ \ \ \ \omega = \omega(\tau,\sigma) := \frac{\tau-\sigma}{4(t-\tau)(t-\sigma)},
\]
and the substitution $u = z^2$ gives, for every $\ve>0$,
\[
\int_0^\infty e^{-(\ve-i\omega)z^2}\, z^{a+1}\, \frac{dz}z = \frac{\Gamma(\frac{a+1}2)}{2}\, (\ve - i \omega)^{-\frac{a+1}2},
\]
the principal branch being understood. Since $|(\ve-i\omega)^{-\frac{a+1}2}|\le |\omega|^{-\frac{a+1}2}$, and
\[
(\ve-i\omega)^{-\frac{a+1}2}\ \underset{\ve\to 0^+}{\longrightarrow}\ |\omega|^{-\frac{a+1}2}\, e^{\,i\,\frac{(a+1)\pi}{4}\operatorname{sgn}(\omega)},
\]
with $\operatorname{sgn}(\omega) = \operatorname{sgn}(\tau-\sigma)$ and $\left[(t-\tau)(t-\sigma)\right]^{-\frac{a+1}2}|\omega|^{-\frac{a+1}2} = 2^{a+1}\, |\tau-\sigma|^{-\frac{a+1}2}$, dominated convergence applies in the $\tau, \sigma$ integrals -- an admissible majorant being $C(a)\, |\tau-\sigma|^{-\frac{a+1}2}\, ||\Phi(\cdot,\tau)||_{L^2_x} ||\Phi(\cdot,\sigma)||_{L^2_x}$ -- while, by monotone convergence, $Q_\ve(t)\uparrow ||\Theta_a^\star(\Phi)(\cdot,t)||_{L^2_a(\RN)}^2$ as $\ve\downarrow 0$. We conclude that
\begin{align}\label{theta}
& ||\Theta_a^\star(\Phi)(\cdot,t)||_{L^2_a(\RN)}^2 = \frac{2^{-a}}{\Gamma\!\left(\frac{a+1}{2}\right)}
\int_0^t \int_0^t  \frac{\exp\!\left(i\,\frac{(a+1)\pi}{4}\,\operatorname{sgn}(\tau-\sigma)\right)}{|\tau-\sigma|^{\frac{a+1}{2}}} \int_{\Rd} S(\sigma-\tau)(\Phi(\cdot,\tau))(x) \overline{\Phi(x,\sigma)} dx d\tau d\sigma,
\end{align}
in accordance with \eqref{ohoh} in Lemma \ref{L:piccoloSa}. Assume now that $q, r$ satisfy \eqref{bdryad} with $r\ge 2$. Since $a>-1$, we have
\begin{equation}\label{qzero}
\frac 2q = \frac{a+1}2 + d(\frac 12 - \frac 1r) > 0.
\end{equation}
H\"older inequality gives 
\begin{align}\label{ohohoh}
\left|\int_{\Rd} S(\sigma-\tau)(\Phi(\cdot,\tau))(x) \overline{\Phi(x,\sigma)} dx\right| \le ||S(\sigma-\tau)(\Phi(\cdot,\tau))||_{L^r_x} ||\Phi(\cdot,\sigma)||_{L^{r'}_x}.
\end{align}
Using \eqref{ohohoh}, we finally obtain from \eqref{theta}
\begin{align}\label{theta2}
& ||\Theta_a^\star(\Phi)(\cdot,t)||_{L^2_a(\RN)}^2  \le C(a) \int_0^t \int_0^t \frac{1}{|\tau-\sigma|^{\frac{a+1}2}} ||S(\sigma-\tau)(\Phi(\cdot,\tau))||_{L^r_x} ||\Phi(\cdot,\sigma)||_{L^{r'}_x} d\tau d\sigma
\\
& \le C(a,d) \int_0^t \int_0^t \frac{1}{|\tau-\sigma|^{\frac{a+1}2}} \frac{1}{|\tau-\sigma|^{d(\frac{1}2-\frac 1r)}}||\Phi(\cdot,\tau)||_{L^{r'}_x} ||\Phi(\cdot,\sigma)||_{L^{r'}_x} d\tau d\sigma
\notag
\\
& \le C(a,d) \int_0^t \left(\int_{\R}  \frac{||\Phi(\cdot,\tau)||_{L^{r'}_x}}{|\tau-\sigma|^{\frac{a+1}2+ d(\frac{1}2-\frac 1r)}} d\tau\right) ||\Phi(\cdot,\sigma)||_{L^{r'}_x}  d\sigma,
\notag
\end{align}
where in the second to the last inequality we have used Proposition \ref{P:fundis}.
If we let
\[
h(\tau) = ||\Phi(\cdot,\tau)||_{L^{r'}_x},
\]
and consider the fractional integration operator \eqref{marcel},
then from \eqref{theta2} we have
\begin{align}\label{theta22}
||\Theta_a^\star(\Phi)(\cdot,t)||_{L^2_a(\RN)}^2  & \le C(a,d) \int_0^t I_\beta(h)(\sigma) ||\Phi(\cdot,\sigma)||_{L^{r'}_x}  d\sigma,
\end{align}
where we have let
\begin{equation}\label{beta}
1-\beta = \frac{a+1}2+ d(\frac{1}2-\frac 1r) = \frac{d+a+1}2 - \frac dr = \frac 2q,
\end{equation}
where in the last equality we have used \eqref{bdryad}. Since we are assuming $q>2$, from \eqref{qzero} and \eqref{beta} we obtain $0<\beta<1$.
Exactly as in \eqref{marcello} we thus conclude that 
\begin{equation}\label{marcellino}
\|I_\beta(h)||_{L^q(\R)} \le C(\beta,q) ||h||_{L^{q'}(\R)} = C(d,a,r) ||\Phi||_{L^{q'}_t L^{r'}_x}.
\end{equation}

Finally, applying H\"older inequality to \eqref{theta22}, we obtain
\begin{align}\label{theta3}
||\Theta_a^\star(\Phi)(\cdot,t)||_{L^2_a(\RN)}^2  & \le C(a,d) \int_\R I_\beta(h)(\sigma) ||\Phi(\cdot,\sigma)||_{L^{r'}_x}  d\sigma
\\
& \le ||I_\beta(h)||_{L^q_t} ||\Phi||_{L^{q'}_t L^{r'}_x} \le C(d,a,r) ||\Phi||^2_{L^{q'}_t L^{r'}_x},
\notag
\end{align}
where in the last inequality we have used \eqref{marcellino}. This establishes \eqref{ok}.

To complete the proof of the theorem we are left with proving \eqref{okk}. We use one more time the duality argument. Let $V\in \mathscr S(\RN\times \R)$, then using \eqref{T} in Proposition \ref{P:rep}, we find
\begin{align*}
\left|\sa\sa \Theta_a^\star(\Phi),V\da\da\right| & = \left|\sa \Phi,\Theta_a(V)\da\right| = \left|\int_{\R} \int_{\Rd} \Phi(y,\tau) \overline{\Theta_a(V)(y,\tau)} dy d\tau\right|
\\
& = \left|\int_{\R} \int_{\Rd} \Phi(y,\tau) \int_\tau^{\infty} \overline{\mathbb S_a(\tau-\sigma)(V(\cdot,\sigma))(Y_0)} d\sigma dy d\tau\right|
\\
& = \left|\int_{\R} \int_\tau^{\infty} \int_{\Rd}\Phi(y,\tau) S(\sigma-\tau)\left(S_a(\sigma-\tau)\left(\overline{V(\cdot,\cdot,\sigma)}\right)(0)\right)(y) dy d\sigma d\tau\right|
\\
& \le \int_{\R} \int_{\R} ||\Phi(\cdot,\tau)||_{L^{r'}_x} ||S(\sigma-\tau)\left(S_a(\sigma-\tau)\left(\overline{V(\cdot,\cdot,\sigma)}\right)(0)\right)||_{L^{r}_x} d\sigma d\tau
\\
& \le C(d,r) \int_{\R} \int_{\R} ||\Phi(\cdot,\tau)||_{L^{r'}_x} \frac{1}{|\sigma-\tau|^{d(\frac 12 - \frac 1r)}} ||S_a(\sigma-\tau)\left(\overline{V(\cdot,\cdot,\sigma)}\right)(0)||_{L^{r'}_x} d\sigma d\tau,
\end{align*}
where in the second equality the unimodular factor $\overline{i} = -i$ of \eqref{T} has been absorbed by the modulus, in the third one we have used again the boundary evaluation \eqref{Sat0} and the factorisation \eqref{prop2}, as in the proof of Proposition \ref{P:rep}, and where in the last inequality we have used  Proposition \ref{P:fundis}.
Since \eqref{Sat0} gives
\begin{align*}
& S_a(\sigma-\tau)\left(\overline{V(\cdot,\cdot,\sigma)}\right)(0) = \int_0^\infty S_a(z,0,\sigma-\tau) \overline{V(\cdot,z,\sigma)} z^a dz,
\end{align*}
from \eqref{piccoloSa} and Minkowski integral inequality we obtain the critical estimate
\begin{align*}
& ||S_a(\sigma-\tau)\left(\overline{V(\cdot,\cdot,\sigma)}\right)(0)||_{L^{r'}_x} \le \frac{C(a)}{|\sigma-\tau|^{\frac{a+1}2}}  \int_0^\infty ||V(\cdot,z,\sigma)||_{L^{r'}_x} z^a dz. 
\end{align*}

We emphasize that this is the second time in the proof where, because the point $Y_0 = (y,0)\in \p \RN$, we are not forced to distinguish between $-1<a<0$ and $a\ge 0$ (the first small miracle has occurred in the proof of \eqref{theta}). 
Substituting this estimate in the above, we infer
\begin{align}\label{dualstar}
\left|\sa\sa \Theta_a^\star(\Phi),V\da\da\right| & \le C(d,a,r) \int_0^\infty \int_{\R} ||\Phi(\cdot,\tau)||_{L^{r'}_x}  \int_{\R} \frac{||V(\cdot,z,\sigma)||_{L^{r'}_x}}{|\sigma-\tau|^{\frac{a+1}2+ d(\frac 12 - \frac 1r)}} d\sigma d\tau z^a dz
\\
& = C(d,a,r) \int_0^\infty \int_{\R} ||\Phi(\cdot,\tau)||_{L^{r'}_x} I_\beta(h(z,\cdot))(\tau) d\tau z^a dz,
\notag
\end{align}
where we have let 
\[
h(z,\sigma) = ||V(\cdot,z,\sigma)||_{L^{r'}_x},
\]
and $\beta$ is the same as in \eqref{beta}. Since as in the proof of \eqref{ok}, we find
\[
||I_\beta(h(z,\cdot))||_{L^q_t} \le C(\beta)\ ||h(z,\cdot)||_{L^{q'}_t} = ||V(\cdot,z,\cdot)||_{L^{q'}_t L^{r'}_x},
\]
applying H\"older inequality to \eqref{dualstar}, we finally have
\begin{align*}
\left|\sa\sa \Theta_a^\star(\Phi),V\da\da\right| &  \le C(d,a,r) ||\Phi||_{L^{q'}_t L^{r'}_x} \int_0^\infty ||I_\beta(h(z,\cdot))||_{L^q_t} z^a dz
\\ 
& \le C(d,a,r) ||\Phi||_{L^{q'}_t L^{r'}_x} \int_0^\infty ||V(\cdot,z,\cdot)||_{L^{q'}_t L^{r'}_x} z^a dz
\\
& = C(d,a,r) ||\Phi||_{L^{q'}_t L^{r'}_x} ||V||_{L^1_{a,z} L^{q'}_t L^{r'}_x}.
\end{align*}
Taking the supremum on all $V\in \mathscr S(\RN\times \R)$, by the density of $\mathscr S(\RN\times \R)$ in $L^1_{a,z} L^{q'}_t L^{r'}_x $ and duality,  we finally obtain \eqref{okk}.

\end{proof}

\section{Traces on the boundary}\label{S:cont}

The boundary Strichartz estimate \eqref{okk} in Theorem \ref{T:V} establishes the uniform boundedness in the orthogonal variable $z>0$ of the solution $U$ to the problem \eqref{nozeroin}. This result does not suffice, however, to guarantee the existence of a trace on the boundary, a property which plays a critical role in the analysis of the nonlinear problem \eqref{cp02}. 

In this section we close this gap, for the boundary term $\Theta^\star_a(\Phi)$ in Theorem \ref{T:gap}, and for the bulk terms $\mathbb T^\star_a(u_0)$ and $\mathbb D_a(F)$ in Theorem \ref{T:gapbulk} below. Together, these results show that the solution $U$ admits a strong trace on $\p\RN\times\R$. Our first result is the following.

\begin{theorem}[Existence of traces]\label{T:gap}
Let $-1<a<1$ and suppose that $q, r$ satisfy \eqref{bdryad}, with $r\ge 2$ and $q>2$. Given $\Phi\in L^{q'}_t L^{r'}_x$,  
the function $\Theta^\star_a(\Phi)$ belongs to the space $C^b_z L^q_t L^r_x$. By this we mean that for every $z_0\ge 0$
\begin{equation}\label{cont}
\underset{z\to z_0}{\lim}\ ||\Theta^\star_a(\Phi)(\cdot,z,\cdot) - \Theta^\star_a(\Phi)(\cdot,z_0,\cdot)||_{L^q_t L^r_x} = 0.
\end{equation}
\end{theorem}

\begin{proof}
For the definition of the space $C^b_z L^q_t L^r_x$ see \eqref{Cz}. We divide the proof into two steps, by first showing that the desired conclusion \eqref{cont} is true when $\Phi\in \mathscr S(\R^{d+1})$. To simplify the notation, we set $V_z(x,t) := \Theta^\star_a(\Phi)(x,z,t)$. In view of \eqref{Thetastar}, given $z\ge 0$ we have
\begin{align*}
V_z(x,t) & = - i \int_0^{t} \int_{\Rd}  \mathbb S_a(X,Y_0,t-\tau) \Phi(y,\tau) dy d\tau = - i \int_0^{t} S_a(z,0,t-\tau) S(t-\tau)(\Phi(\cdot,\tau))(x) d\tau
\\
& = - i \int_0^{t} S_a(z,0,\tau) S(\tau)(\Phi(\cdot,t-\tau))(x) d\tau.
\end{align*}
Keeping \eqref{piccoloSa} in mind, we see that for $t>0$ and every $z, z_0\ge 0$, we have
\begin{equation}\label{Uz}
V_z(x,t) - V_{z_0}(x,t) = c(a) \int_0^t \tau^{-\frac{a+1}{2}} e^{i\frac{z_0^2}{4\tau}}[e^{i \frac{\omega(h)}{4\tau}} - 1] S(\tau)(\Phi(\cdot,t-\tau))(x) d\tau,
\end{equation}
where $c(a)\in \C$ is a number whose explicit value is irrelevant in what follows, and we have let
\[
\omega(h) = h^2 + 2 z_0 h,\ \ \ \ \ \ h = z-z_0.
\]
Without restriction, we can assume that $h\ge 0$. If we set
\[
U_h(x,t) := V_z(x,t) - V_{z_0}(x,t),
\]
it is clear that, to establish \eqref{cont}, we need to prove that 
\begin{equation}\label{Uh}
\underset{h\to 0}{\lim}\ ||U_h||_{L^q_t L^r_x} = 0.
\end{equation}
With this objective in mind, we define for $\tau\ge 0$
\[
g(\tau):=\|\Phi(\cdot,\tau)\|_{L_x^{r'}}.
\]
Because $\Phi\in\mathscr S(\R^{d+1})$, we have  $g\in L^1(\R)\cap L^\infty(\R)$. This is easily seen as follows. Since $2\le r \le \infty$, we know $\frac 12 \le \frac{1}{r'} \le 1$. Fix $N\in \mathbb N$ such that $N> d\ge \frac{d}{r'}$. Then 
\begin{align*}
g(t)^{r'} & = \int_{\Rd} \frac{\left[(1+|x|^2)^{\frac N2} |\Phi(x,t)|\right]^{r'}}{(1+|x|^2)^{\frac{r' N}{2}}} dx \le C(N,r) \int_{\Rd} \frac{dx}{(1+|x|^2)^{\frac{r' N}{2}}} \le C(d,N,r).
\end{align*}
This shows that $g\in L^\infty(\R)$, and similar arguments prove that $g\in L^1(\R)$. 

Next, applying Proposition \ref{P:fundis}, for every $t>0$ we obtain from \eqref{Uz}:
\begin{align}\label{Uzero}
||U_h(\cdot,t)||_{L_x^r} & \le C(a)
\int_0^t \tau^{-\frac{a+1}{2}}
|e^{i\frac{\omega(h)}{4\tau}}-1|
||S(\tau)(\Phi(\cdot,t-\tau))||_{L_x^r} d\tau 
\\
& \le C(d,a,r) \int_0^t \tau^{-\frac{a+1}{2} -d(\frac 12 - \frac 1r)}
|e^{i \frac{\omega(h)}{4\tau}}-1| g(t-\tau) d\tau
\notag 
\\
& \le C(d,a,r) \int_0^\infty \tau^{-\frac 2q} |e^{i \frac{\omega(h)}{4\tau}}-1|  g(t-\tau) d\tau
\notag 
\\
&  = C(d,a,r) K_h \star g(t),
\notag
\end{align}
where we have let 
\[
K_h(\tau):= \tau^{-\frac 2q}|e^{i\frac{\omega(h)}{4\tau}}-1|\,\mathbf 1_{(0,\infty)}(\tau).
\]
We now claim that:
\begin{equation}\label{Kzero}
\underset{h\to 0}{\lim}\ ||K_h||_{L^1(\R)} = 0.
\end{equation}
To prove \eqref{Kzero} we argue as follows:
\begin{align*}
||K_h||_{L^1(\R)} & = \int_0^{\omega(h)} \tau^{-2/q}|e^{i\frac{\omega(h)}{4\tau}}-1| d\tau + \int_{\omega(h)}^\infty \tau^{-2/q}|e^{i\frac{\omega(h)}{4\tau}}-1| d\tau = I_h(\tau) + II_h(\tau).
\end{align*}
To estimate $I_h(\tau)$ we use the trivial bound
$|e^{i\frac{\omega(h)}{4\tau}}-1|\le 2$,
to obtain
\[
I_h(\tau) \le 2 \int_0^{\omega(h)} \tau^{-2/q} d\tau = \frac{2q}{q-2} \omega(h)^{1-\frac 2q}\ \underset{h\to 0}{\longrightarrow}\ 0,
\]
since by assumption $q>2$. To estimate $II_h(\tau)$, instead, we use the bound $|e^{i \theta} - 1| \le |\theta|$, valid for every $\theta\in \R$, to estimate
\[
II_h(\tau) \le \int_{\omega(h)}^\infty \tau^{-2/q} \frac{\omega(h)}{4\tau} d\tau = \frac{\omega(h)}4 \int_{\omega(h)}^\infty \tau^{-2/q -1}  d\tau = \frac q8 \omega(h)^{1-\frac 2q}\ \underset{h\to 0}{\longrightarrow}\ 0.
\]
This proves \eqref{Kzero}. 
Since $g\in L^1(\R)\cap L^\infty(\R)$, Young's inequality now gives
\begin{equation}\label{Y}
\|K_h\star g\|_{L_t^1}\le \|K_h\|_{L^1}\|g\|_{L^1},
\qquad
\|K_h\star g\|_{L_t^\infty}\le \|K_h\|_{L^1}\|g\|_{L^\infty}.
\end{equation}
From \eqref{Uzero}, \eqref{Kzero} and \eqref{Y}, we infer
\begin{equation}\label{UUzero}
\|U_h\|_{L_t^1L_x^r}
\le C\|K_h\|_{L^1}\|g\|_{L^1}\
\underset{h\to 0}{\longrightarrow}\ 0,\ \ \ \ \ \|U_h\|_{L_t^\infty L_x^r}
\le C\|K_h\|_{L^1}\|g\|_{L^\infty}\
\underset{h\to 0}{\longrightarrow}\ 0.
\end{equation}
Finally, we have from \eqref{UUzero}:
\begin{align*}
\|U_h\|^q_{L_t^qL_x^r} = \int_\R \|U_h(\cdot,t)\|^q_{L_x^r} dt = \int_\R \|U_h(\cdot,t)\|^{q-1}_{L_x^r} \|U_h(\cdot,t)\|_{L_x^r}dt \le \|U_h\|^{q-1}_{L_t^\infty L_x^r} \|U_h\|_{L_t^1 L_x^r}\
\underset{h\to 0}{\longrightarrow}\ 0.
\end{align*} 

This proves \eqref{Uh} - and therefore \eqref{cont} -  when $\Phi\in \mathscr S(\R^{d+1})$. Assume now that $\Phi\in L^{q'}_tL^{r'}_x$ and choose $\Phi_n\in \mathscr S(\R^{d+1})$ such that
$\Phi_n\to \Phi$ in $L_t^{q'}L_x^{r'}$.
Denoting $V_z(x,t) = \Theta^\star_a(\Phi)(x,z,t)$ and $V^n_z(x,t) = \Theta^\star_a(\Phi_n)(x,z,t)$, we have
\begin{equation}\label{Vz}
||V_z - V_{z_0}||_{L^q_t L^r_x}\le ||V_z - V^n_z||_{L^q_t L^r_x} + ||V^n_z - V^n_{z_0}||_{L^q_t L^r_x} + ||V^n_{z_0} -  V_{z_0}||_{L^q_t L^r_x}.
\end{equation}

By the boundary Strichartz estimate \eqref{okk} there exists a constant $C(d,a,r)>0$ such that, a priori for a.e.\ $z\ge 0$,
\begin{align*}
||V_z - V_z^n||_{L^q_t L^r_x} & = ||\Theta^\star_a(\Phi - \Phi_n)(\cdot,z,\cdot)||_{L^q_t L^r_x} \le C(d,a,r) ||\Phi - \Phi_n||_{L^{q'}_t L^{r'}_x}.
\end{align*}
We claim that this bound holds in fact for \emph{every} $z\ge0$. Indeed, for $n, m\in \N$ the function $V^n_z - V^m_z = \Theta^\star_a(\Phi_n - \Phi_m)(\cdot,z,\cdot)$ has Schwartz datum, and is therefore continuous in $z$ by the first part of the proof; for a $z$-continuous function the essential supremum in $z$ coincides with the genuine supremum. Consequently \eqref{okk}, applied to $\Phi_n - \Phi_m$, shows that $(V^n)_{n\in\N}$ is a Cauchy sequence with respect to the supremum over $z\ge0$ of the $L^q_tL^r_x$ norm. Its uniform limit is continuous in $z$ and agrees with $\Theta^\star_a(\Phi)(\cdot,z,\cdot)$ for a.e.\ $z\ge0$; it is this representative of $\Theta^\star_a(\Phi)$ that we denote by $V_z$, and for it the above bound holds at every $z\ge0$, by letting $m\to\infty$ in the bound for $V^m_z - V^n_z$.
Given $\ve>0$, we thus choose $n(\ve)\in \mathbb N$ such that
\[
2 C(d,a,r) ||\Phi - \Phi_n||_{L^{q'}_t L^{r'}_x} \le \frac{\ve}2,\ \ \ \ \ \ \forall n\ge n(\ve).
\]
We infer that for every $z, z_0\ge 0$ 
\begin{equation}\label{Vzz}
||V_z - V_z^{n(\ve)}||_{L^q_t L^r_x} + ||V^{n(\ve)}_{z_0} -  V_{z_0}||_{L^q_t L^r_x} \le \frac{\ve}2\ \ \ \ \ \ \forall z, z_0 \ge 0.
\end{equation}
Using \eqref{Uh}, we now choose $\delta(\ve)>0$ such that for every $z, z_0\ge 0$, with $|z-z_0|\le \delta(\ve)$, we have
\begin{equation}\label{Vzos}
||V^{n(\ve)}_z - V^{n(\ve)}_{z_0}||_{L^q_t L^r_x} \le \frac{\ve}2.
\end{equation}
Combining \eqref{Vz}, with \eqref{Vzz} and \eqref{Vzos}, we reach the desired conclusion \eqref{cont}.

\end{proof}

\vskip 0.2in

We next establish the corresponding property for the bulk terms $\mathbb T^\star_a(u_0)$ and $\mathbb D_a(F)$ in \eqref{U0}. In the following statement, when $0\le a<1$ we assume that $(q,r,\infty)$ is admissible with $q>2$ (and $r<\frac{2d}{d+a-1}$ when $d+a-1>0$), whereas when $-1<a<0$ we assume that $(q_\infty,r,\infty)$ and $(q,r,\infty)$ satisfy \eqref{uffaa} with $q>2$ and $q_\infty<\infty$, and $k$ denotes the weight in Definition \ref{D:mn}.

\begin{theorem}[Traces of the bulk terms]\label{T:gapbulk}
Let $0\le a<1$. Then for every $u_0\in L^2_a(\RN)$ and every $F$ with $||F||_{L^1_{a,z}L^{q'}_tL^{r'}_x}<\infty$ we have
\[
\mathbb T^\star_a(u_0),\ \mathbb D_a(F)\ \in\ C^b_zL^q_tL^r_x.
\]
If instead $-1<a<0$, then for every $u_0\in L^2_a(\RN)$ and every $F$ with $||Fk^{-1}||_{L^1_{a,z}L^{q'\cap q'_\infty}_tL^{r'}_x}<\infty$, the maps $z\to \mathbb T^\star_a(u_0)(\cdot,z,\cdot)k(z)$ and $z\to \mathbb D_a(F)(\cdot,z,\cdot)k(z)$ belong to $C^b_zL^{q+q_\infty}_tL^r_x$.
\end{theorem}

\begin{proof}
In view of the uniform bounds \eqref{bulk00}, \eqref{Dabound} (case $0\le a<1$) and \eqref{bulksum}, \eqref{crucial} (case $-1<a<0$), it suffices to establish the continuity statements for data in suitable dense classes. The passage from dense data to general data requires some care, since the estimates just quoted control only essential suprema in $z$, and it is carried out in Step 4 below. For the homogeneous term we choose the class of band-limited data
\[
\mathcal D_a := \operatorname{span}\left\{\vf(x)\,\psi(z)\ \Big|\ \vf\in \mathscr S(\Rd),\ \mathcal H_{\frac{a-1}2}(\psi)\in C_0^\infty((0,\infty))\right\},
\]
see Section \ref{S:band}. Since $\mathcal H_{\frac{a-1}2}$ is an isometry of $L^2_{a,z}$ onto itself and $C_0^\infty((0,\infty))$ is dense in $L^2_{a,z}$, the class $\mathcal D_a$ is dense in $L^2_a(\RN)$. For the inhomogeneous term we choose the class $\mathcal E_a$ of finite sums of elementary tensors
\[
F(X,t) = \vf(x)\, \psi(z)\, \eta(t),\ \ \ \ \ \ \vf\in \mathscr S(\Rd),\ \ \psi \in C_0^\infty((0,\infty)),\ \ \eta\in C_0^\infty(\R).
\]
The class $\mathcal E_a$ is dense in the mixed-norm classes of the statement. Indeed, the norms in question are of the form $\int_0^\infty ||F(\cdot,z,\cdot)||_{E}\, d\mu(z)$, where $E = L^{q'}_tL^{r'}_x$ and $d\mu = z^adz$ when $0\le a<1$, and $E = L^{q'\cap q'_\infty}_tL^{r'}_x$ and $d\mu = k^{-1}(z)\,z^adz$ when $-1<a<0$; in either case $\mu$ is a $\sigma$-finite Borel measure on $(0,\infty)$, finite on bounded sets since $a>-1$ and $k\equiv1$ near the origin. Simple functions $\sum_i \mathbf 1_{A_i}(z)\, \Phi_i(x,t)$, with $A_i\subset(0,\infty)$ bounded Borel sets and $\Phi_i\in E$, are dense in $L^1(d\mu;E)$; each $\mathbf 1_{A_i}$ can be approximated in $L^1(d\mu)$ by functions in $C_0^\infty((0,\infty))$, and each $\Phi_i$ in the norm of $E$ by finite sums of products $\vf(x)\eta(t)$ of the above form, by the density of tensor products of Schwartz functions in the mixed Lebesgue classes.

Three properties of $\mathcal D_a$ will be used: (i) it is invariant under $\Delta_x$, under $\Ba$ and under the propagator $\mathbb S_a(t)$, which acts diagonally on the Fourier--Hankel side, see \eqref{prop2}, \cite{GS} and Remark \ref{R:bandt}; (ii) every $\psi$ with $\mathcal H_{\frac{a-1}2}(\psi)\in C_0^\infty((0,\infty))$ is a smooth function of $z^2$ with rapidly decreasing derivatives, see Lemma \ref{L:band}; consequently $z^a\p_z\psi = O(z^{a+1})\to0$ as $z\to0^+$; (iii) $\left(\Ba\right)^j(1-\Delta_x)^m u_0\in \mathcal D_a\subset L^2_a(\RN)$ for all $j, m\in \mathbb N$ and $u_0\in \mathcal D_a$.

\smallskip
\noindent \emph{Step 1: a trace inequality.} Let $-1<a<1$, $Z>0$, and let $g:(0,2Z]\to L^2(\Rd)$ be of class $C^1$ with $g, \p_zg\in L^2_{a,z}((0,2Z);L^2_x)$. Choosing $w\in [Z,2Z]$ such that $||g(w)||^2_{L^2_x}\le \frac 1Z\int_Z^{2Z}||g(\zeta)||^2_{L^2_x}d\zeta$, we obtain for every $0<z\le Z$
\begin{equation}\label{tracez}
||g(z)||_{L^2_x}\ \le\ ||g(w)||_{L^2_x} + \int_z^{2Z}||\p_\zeta g(\zeta)||_{L^2_x}\, d\zeta\ \le\ C(a,Z)\left(||g||_{L^2_{a,z}L^2_x} + ||\p_zg||_{L^2_{a,z}L^2_x}\right),
\end{equation}
where in the last step we have used Cauchy--Schwarz and $\int_0^{2Z}\zeta^{-a}d\zeta = \frac{(2Z)^{1-a}}{1-a}<\infty$, which requires $a<1$. The same chain shows that $\lim_{z\to0^+}g(z)$ exists in $L^2_x$, so that $g$ extends continuously to $[0,Z]$.

\smallskip
\noindent \emph{Step 2: uniform bounds for data in $\mathcal D_a$.} Let $u_0\in \mathcal D_a$ and set $U(\cdot,\cdot,t) = \mathbb T^\star_a(u_0)(\cdot,\cdot,t)$. Fix $m\in \mathbb N$ with $2m>d$, so that by the Sobolev embedding $||f||_{L^r_x}\le C(d,m)||(1-\Delta_x)^mf||_{L^2_x}$ for every $2\le r\le \infty$. Since $\Delta_x$ and $\Ba$ commute with the propagator, and since the latter is unitary on $L^2_a(\RN)$ (Proposition \ref{P:Uni}), the norms $||(1-\Delta_x)^mU(t)||_{L^2_a(\RN)}$ and $||\Ba(1-\Delta_x)^mU(t)||_{L^2_a(\RN)}$ are constant in $t$. Moreover, by property (ii) above and \eqref{dzpazero}, the function $V(t):=(1-\Delta_x)^mU(t)$ satisfies $\lim_{z\to0^+}z^a\p_zV(\cdot,z,t)=0$, so that the identity \eqref{pos}, integrated in the $x$ variable, gives
\begin{align*}
||\p_zV(t)||^2_{L^2_a(\RN)} & = -\int_{\RN}\overline{V(t)}\,\Ba V(t)\, d\omega_a
\\
& \le\ ||\Ba(1-\Delta_x)^mu_0||_{L^2_a(\RN)}\, ||(1-\Delta_x)^mu_0||_{L^2_a(\RN)}.
\end{align*}
We note that \eqref{pos} was stated for $\psi\in C^\infty_0(\R^+)$, for which the boundary terms in the integration by parts vanish trivially. In the present situation they vanish at $z=0$ because $\lim_{z\to0^+}z^a\p_zV(\cdot,z,t)=0$, and at $z=\infty$ by the rapid decay in $z$ of band-limited data and of their evolutions, see Lemma \ref{L:band} and Remark \ref{R:bandt}. We also stress that no real part is needed in \eqref{pos}, nor in the identity above: since $\p_z\overline V\cdot z^a\p_zV = |\p_zV|^2 z^a$, the integration by parts returns a real, nonnegative quantity.
Applying \eqref{tracez} to $g(z) = V(\cdot,z,t)$ and then the Sobolev embedding, we conclude that for every $Z>0$
\begin{equation}\label{unibulk}
\sup_{t\in \R}\ \sup_{0\le z\le Z}\ ||U(\cdot,z,t)||_{L^r_x}\ \le\ C(a,Z)\, N_m(u_0),
\end{equation}
where $N_m(u_0)$ denotes the sum of the two conserved norms above. On the other hand, writing $\mathbb S_a(t) = S(t)\circ S_a(t)$ as in \eqref{prop2}, and combining Proposition \ref{P:fundis} with the kernel bounds behind \eqref{good} and \eqref{weight}, we obtain for every $z>0$ and $t\neq0$
\begin{equation}\label{decaybulk}
\begin{cases}
||U(\cdot,z,t)||_{L^r_x}\le 
C\,|t|^{-\frac 2q}\, N'(u_0)\ \ (0\le a<1),
\\
||U(\cdot,z,t)\,k(z)||_{L^r_x}\le C\left(|t|^{-\frac 2q}+|t|^{-\frac{2}{q_\infty}}\right)N'(u_0)\ \ (-1<a<0),
\end{cases}
\end{equation}
where $N'(u_0):=\int_0^\infty ||u_0(\cdot,\zeta)||_{L^{r'}_x}\, h(\zeta)\,\zeta^ad\zeta$, with $h\equiv1$ when $a\ge0$ and $h = k^{-1}$ when $a<0$, is finite by Lemma \ref{L:band}, since $h$ has polynomial growth, and where we have used \eqref{acinfty}, respectively \eqref{uffaa}, to identify the time exponents.

\smallskip
\noindent \emph{Step 3: continuity for data in $\mathcal D_a$.} Let $z_n\to z_0$ in $[0,\infty)$ and let $t\neq0$. Writing $U(x,z,t) = S(t)(g_{z,t})(x)$ with $g_{z,t}(x) = \int_0^\infty S_a(z,\zeta,t)\, u_0(x,\zeta)\,\zeta^a d\zeta$, the continuity of $z\mapsto S_a(z,\zeta,t)$ on $[0,\infty)$, the bound $|S_a(z,\zeta,t)|\le C(t,Z)(1+\zeta^{\frac{|a|}2})$ for $0\le z\le Z$, and the rapid decay of $u_0$ give, by dominated convergence (first in $\zeta$, then in $x$) combined with Proposition \ref{P:fundis},
\[
||U(\cdot,z_n,t)-U(\cdot,z_0,t)||_{L^r_x}\le C(d,r)\, |t|^{-d(\frac12-\frac1r)}\, ||g_{z_n,t}-g_{z_0,t}||_{L^{r'}_x}\ \underset{n\to\infty}{\longrightarrow}\ 0 .
\]
When $0\le a<1$, by \eqref{unibulk} and \eqref{decaybulk} the functions $t\mapsto ||U(\cdot,z_n,t)-U(\cdot,z_0,t)||_{L^r_x}$ are dominated, uniformly in $n$, by $C\min\{1, |t|^{-\frac2q}\}\,(N_m(u_0)+N'(u_0))\in L^q_t(\R)$, and dominated convergence in $t$ yields $||U(\cdot,z_n,\cdot)-U(\cdot,z_0,\cdot)||_{L^q_tL^r_x}\to0$. When $-1<a<0$ we argue in the same way for $Uk$, splitting $\R = \{|t|\le1\}\cup\{|t|\ge1\}$: on the first set the majorant $C\, N_m(u_0)$ (recall $0<k\le1$) belongs to $L^q(\{|t|\le1\})$, on the second the majorant $C\, |t|^{-\frac{2}{q_\infty}}N'(u_0)$ belongs to $L^{q_\infty}(\{|t|\ge1\})$, and convergence in $L^q(\{|t|\le1\}) + L^{q_\infty}(\{|t|\ge1\})$ implies convergence in the sum space $L^{q+q_\infty}_tL^r_x$. Since $k$ is continuous and bounded, the continuity of $z\mapsto U(\cdot,z,\cdot)k(z)$ follows.

For $\mathbb D_a(F)$ it suffices, by linearity, to treat a single elementary tensor $F = \vf\,\psi\,\eta \in \mathcal E_a$, with $\psi\in C_0^\infty((0,\infty))$. On such tensors the propagator factorises, $\mathbb S_a(s)(\vf\,\psi) = S(s)\vf\ S_a(s)\psi$, by \eqref{prop2}. We record the following bounds for the two factors. First, we write $S_a(s)\psi(z) = \int_0^\infty (z\la)^{-\nu}J_\nu(z\la)\, e^{-is\la^2}\,\mathcal H_\nu(\psi)(\la)\,\la^ad\la$, with $\nu = \frac{a-1}2$. The function $w\mapsto w^{-\nu}J_\nu(w)$ is bounded near the origin and of size $w^{-\nu-\frac12} = w^{-\frac a2}$ at infinity; hence $|w^{-\nu}J_\nu(w)|\le C(a)$ when $0\le a<1$, whereas $|w^{-\nu}J_\nu(w)|\le C(a)\bigl(1+w^{-\frac a2}\bigr)$ when $-1<a<0$. Together with the rapid decay of $\mathcal H_\nu(\psi)$ furnished by the last part of Lemma \ref{L:band}, and since $|e^{-is\la^2}|=1$, this gives, uniformly in $s\in\R$ and $z\ge0$,
\[
|S_a(s)\psi(z)| \le C(a,\psi)\ \ (0\le a<1),\ \ \ \ \ \ k(z)\,|S_a(s)\psi(z)| \le C(a,\psi)\ \ (-1<a<0),
\]
where in the second bound we have used $(z\la)^{-\frac a2}\le z^{-\frac a2}\la^{-\frac a2}$, then $k(z)\bigl(1+z^{-\frac a2}\bigr)\le 2$, the factor $\la^{-\frac a2}$ being absorbed by the rapid decay of $\mathcal H_\nu(\psi)$. Second, by the dispersive estimates \eqref{good} and \eqref{weight}, for $s\neq0$,
\[
|S_a(s)\psi(z)| \le C(a)\,|s|^{-\frac{a+1}2}\,||\psi||_{L^1_{a,z}},\ \ \ \ \ \ k(z)\,|S_a(s)\psi(z)| \le C(a)\left(|s|^{-\frac{a+1}2}+|s|^{-\frac12}\right)||\psi\, k^{-1}||_{L^1_{a,z}},
\]
both right-hand sides being finite since $\psi$ has compact support in $(0,\infty)$. For the $x$-factor, Proposition \ref{P:fundis} and the Sobolev embedding used in Step 2 give $||S(s)\vf||_{L^r_x}\le C(\vf)\min\{1, |s|^{-d(\frac12-\frac1r)}\}$. Multiplying the factors and using \eqref{acinfty}, respectively \eqref{uffaa}, exactly as in \eqref{decaybulk}, we conclude that for every $z\ge0$
\[
||\mathbb S_a(s)(\vf\psi)(\cdot,z)||_{L^r_x} \le C(\vf,\psi) \min\{1,|s|^{-\frac2q}\}\ \ (0\le a<1),
\]
with the corresponding two-exponent bound for $k(z)\,\mathbb S_a(s)(\vf\psi)$ when $-1<a<0$; in particular, $\sup_{z\ge0}\,||\mathbb S_a(\cdot)(\vf\psi)(\cdot,z,\cdot)||_{L^q_tL^r_x}\le C(\vf,\psi)$, respectively the same bound in $L^{q+q_\infty}_tL^r_x$ after multiplication by $k$. Moreover, for each fixed $s\neq0$ the map $z\mapsto S_a(s)\psi(z)$ is continuous on $[0,\infty)$, by dominated convergence in its defining integral, so that $z\mapsto \mathbb S_a(\cdot)(\vf\psi)(\cdot,z,\cdot)$ is continuous into $L^q_tL^r_x$ (respectively, times $k$, into $L^{q+q_\infty}_tL^r_x$) by the same dominated convergence argument in $t$ used above for $\mathbb T^\star_a$. Finally, by Minkowski's integral inequality in the $\tau$ variable of \eqref{Daa}, and since the norms involved are invariant under translations in $t$,
\[
||\mathbb D_a(F)(\cdot,z,\cdot) - \mathbb D_a(F)(\cdot,z_0,\cdot)||_{L^q_tL^r_x} \le \int_{\R} \left\|\mathbf 1_{\{t>\tau\}}\left[\mathbb S_a(t-\tau)(\vf\psi)(\cdot,z) - \mathbb S_a(t-\tau)(\vf\psi)(\cdot,z_0)\right]\right\|_{L^q_tL^r_x}\, |\eta(\tau)|\, d\tau ,
\]
in which the integrand tends to zero as $z\to z_0$ for each $\tau$ and is bounded by $2\,C(\vf,\psi)\,|\eta(\tau)|\in L^1_\tau(\R)$; dominated convergence in $\tau$ concludes. The anomalous range is treated in the same way, multiplying by $k$ and using the sum space.

\smallskip
\noindent \emph{Step 4: density.} For general data, let $u_0^{(j)}\in \mathcal D_a$ and $F^{(j)}\in \mathcal E_a$ converge to $u_0$ and $F$ in the norms of the statement. The estimates \eqref{bulk00} and \eqref{Dabound} (respectively \eqref{bulksum} and \eqref{crucial}), applied to the differences $u_0^{(j)} - u_0^{(l)}$ and $F^{(j)} - F^{(l)}$, control a priori only the \emph{essential} supremum in $z$ of the corresponding bulk terms. These differences, however, have data in $\mathcal D_a$ and $\mathcal E_a$, so that by Steps 2 and 3 the corresponding bulk terms are continuous in $z$ on $[0,\infty)$, and for a $z$-continuous function the essential supremum coincides with the genuine supremum. Consequently the sequences $\mathbb T^\star_a(u_0^{(j)})$ and $\mathbb D_a(F^{(j)})$ (multiplied by $k$ when $-1<a<0$) are Cauchy with respect to the supremum over $z\ge0$, i.e., they converge uniformly on $z\ge0$, and their limits are continuous in $z$. Since these limits agree for a.e.\ $z\ge0$ with $\mathbb T^\star_a(u_0)$ and $\mathbb D_a(F)$ -- the operators being defined on general data through the same estimates, see Remark \ref{R:density} -- they provide the continuous representatives asserted in the statement. This completes the proof.
\end{proof}

\section{Proof of Theorems \ref{T:main} and \ref{T:main2}}\label{S:main}

In this brief section we combine the results of Sections \ref{S:cpn}, \ref{S:stri} and \ref{S:bdry} and finally prove Theorems \ref{T:main} and \ref{T:main2}. 
We begin with recalling that, thanks to \eqref{nicea1} in Theorem \ref{T:U}, the mild solution $U$ to the problem \eqref{nozeroin}
is represented by \eqref{U0}:
\begin{equation}\label{UUrep}
U(X,t) = \mathbb T^\star_a(u_0)(X,t) + \mathbb D_a(F)(X,t) + \Theta^\star_a(\Phi)(X,t).
\end{equation}  

\medskip

\begin{proof}[Proof of Theorem \ref{T:main}]
For every $t\in \R$ we obtain from \eqref{UUrep}
\begin{align*}
||U(\cdot,t)||_{L^2_a(\RN)} & \le ||\mathbb T^\star_a(u_0)(\cdot,t)||_{L^2_a(\RN)} + ||\mathbb D_a(F)(\cdot,t)||_{L^2_a(\RN)} + ||\Theta^\star_a(\Phi)(\cdot,t)||_{L^2_a(\RN)}
\\
& \le C(d,a,r)\left[||u_0||_{L^2_a(\RN)} + ||F||_{L^1_t L^2_a(\RN)} + ||\Phi||_{L^{q'}_t L^{r'}_x}\right]
\end{align*}
where in the last inequality we have used Theorem \ref{T:Unistar} and \eqref{ok} in Theorem \ref{T:V}. By taking the supremum in $t\in \R$, we obtain \eqref{Umain}. 

The estimate \eqref{Umain2}, instead, follows by combining \eqref{UUrep} with \eqref{bulk00}, \eqref{Dabound} and with \eqref{okk} in Theorem \ref{T:V}.

\end{proof}

\medskip
We next turn to the

\begin{proof}[Proof of Theorem \ref{T:main2}] To establish \eqref{Umain20} we use again Theorem \ref{T:Unistar}, combined with \eqref{ok} in Theorem \ref{T:V} which we apply to the triple $(q_\infty,r,\infty)$ (recall that this triple satisfies the first equation in \eqref{uffaa}): 
\begin{align*}
||U(\cdot,t)||_{L^2_a(\RN)} & \le ||\mathbb T^\star_a(u_0)(\cdot,t)||_{L^2_a(\RN)} + ||\mathbb D_a(F)(\cdot,t)||_{L^2_a(\RN)} + ||\Theta^\star_a(\Phi)(\cdot,t)||_{L^2_a(\RN)}
\\
& \le C(d,a,r)\left[||u_0||_{L^2_a(\RN)} + ||F||_{L^1_t L^2_a(\RN)} + ||\Phi||_{L^{q'_\infty}_t L^{r'}_x}\right].
\end{align*}
Passing to the supremum in $t\in \R$, we obtain \eqref{Umain20}.

To prove \eqref{Umain22}, we combine \eqref{UUrep} with \eqref{bulksum} and \eqref{crucial}, to obtain
\begin{align}\label{Umain222}
||U\ k||_{L^\infty_z L^{q+q_\infty}_t L^r_x } \le C\ \left[||u_0||_{L^2_a(\RN)} +  ||F \ k^{-1}||_{L^1_{a,z} L^{q'\cap q'_\infty}_t L^{r'}_x }\right] + ||\Theta^\star_a(\Phi) k||_{L^\infty_z L^{q+q_\infty}_t L^r_x }  .
\end{align}
To estimate the third term in the right-hand side of \eqref{Umain222} we use the trivial observation that $0<k(z)\le 1$, and that moreover we have:
\[
L^\infty_z L^{q_\infty}_t L^r_x  \hookrightarrow L^\infty_z L^{q+q_\infty}_t L^r_x.
\]
This gives 
\begin{align}\label{trivi}
||\Theta^\star_a(\Phi) k||_{L^\infty_z L^{q+q_\infty}_t L^r_x } & \le ||\Theta^\star_a(\Phi)||_{L^\infty_z L^{q+q_\infty}_t L^r_x } \le ||\Theta^\star_a(\Phi)||_{L^\infty_z L^{q_\infty}_t L^r_x}
\\
& \le C(d,a,r) ||\Phi||_{L^{q'_\infty}_t L^{r'}_x},
\notag
\end{align}
where in the last inequality we have used \eqref{okk} in Theorem \ref{T:V} applied to the triple $(q_\infty,r,\infty)$. Substituting \eqref{trivi} in \eqref{Umain222}, we finally obtain the desired conclusion \eqref{Umain22}.

\end{proof}

\begin{remark}\label{R:density}
Theorems \ref{T:main}, \ref{T:main2} and \ref{T:V} are stated for Schwartz data, but all the estimates are stable under limits. Consequently, the operators $\mathbb T^\star_a$, $\mathbb D_a$ and $\Theta^\star_a$ extend uniquely, by density, to bounded operators between the spaces appearing in the respective estimates, and for $u_0\in L^2_a(\RN)$, $\Phi\in L^{q'}_tL^{r'}_x$ and $F$ in the corresponding mixed-norm classes, the function $U$ defined by \eqref{U0} -- which is what we mean by the \emph{mild solution} of \eqref{nozeroin} with such data -- satisfies \eqref{Umain}, \eqref{Umain2}, \eqref{Umain20} and \eqref{Umain22} verbatim, as well as the trace properties of Theorems \ref{T:gap} and \ref{T:gapbulk}. In Section \ref{S:well} these extensions are used without further comment.
\end{remark}
\vskip 0.3in

\section{Well-posedness}\label{S:well}

In this section we finally turn to the well-posedness in the nonlinear Cauchy problem \eqref{cp02}. We first introduce the notion of $L^2_a$-mass critical problem. Suppose $U$ solves \eqref{cp02} with $F= 0$ and, for every $\la>0$, consider the function
\[
U_\la(X,t) = \la^{-\gamma} U(\la^{-1} X, \la^{-2} t),\ \ \ \ \gamma>0.
\]
Clearly, for every $\la>0$, $U_\la$ is a solution of $\p_t U_\la - i(\Delta_x  + \Ba)U_\la = 0$. We have
\[
\p_z U_\la(X,t) = \la^{-(1+\gamma)} \p_z U(\la^{-1} X, \la^{-2} t),
\]
and therefore 
\begin{align*}
z^a \p_z U_\la(X,t)& = \la^{a-(1+\gamma)} (\la^{-1}z)^a \p_z U(\la^{-1} X, \la^{-2} t)
\\
& \ \underset{z\to 0^+}{\longrightarrow}\ \mu \la^{a-(1+\gamma)} |U(\la^{-1} x,0,\la^{-2} t)|^{p-1} U(\la^{-1} x,0,\la^{-2} t)
\\
& = \la^{a-1 - \gamma + \gamma p} \mu |U_\la(x,0,t)|^{p-1} U_\la(x,0,t).
\end{align*}
We infer that for $U_\la$ to also be a solution to \eqref{cp02} with $F=0$, we must have
\[
\gamma = \frac{1-a}{p-1}.
\]
Now, the function 
\[
U_\la(X,t) = \la^{-\frac{1-a}{p-1}} U(\la^{-1} X , \la^{-2} t)\ \underset{t\to 0}{\longrightarrow}\ u_{0,\la}(X) := \la^{-\frac{1-a}{p-1}} u_0(\la^{-1}X).
\]
Since we have from \eqref{rescale}
\[
||u_{0,\la}||_{L^2_a(\RN)} = \la^{-\frac{1-a}{p-1}+\frac{d+a+1}2} ||u_{0}||_{L^2_a(\RN)},
\]
we see that in order to have constant mass, we must have
\[
\frac{1-a}{p-1} = \frac{d+a+1}2.
\]
Solving for $p$ in this equation we reach the conclusion that
\[
p = p_c := 1 + \frac{2(1-a)}{d+a+1}.
\]
Note that, since we want $p_c>1$, we must assume that $a<1$. This explains \eqref{pcpos} in Definition \ref{D:mass} in the range  $0\le a<1$. 

When instead $-1<a<0$, then because of the anomalous behavior of the Bessel propagator $e^{i t\Ba}$ in \eqref{weight} of Proposition \ref{P:dis}, the critical exponent $p_c$ is no longer dictated by the scalings $(\la X,\la^2 t)$. It turns out that, in such negative range of values of $a$, the critical dimension ceases to be $D = d+a+1$, and is instead $D = d+1$. Accordingly, the critical exponent for \eqref{cp02} is given by
\[
p_c = 1 + \frac{2}{d+1},
\]
which explains \eqref{pcneg} in Definition \ref{D:mass}.
We will need the following.

\begin{lemma}\label{L:mass}
Let $a>-1$ and suppose that \eqref{cp02} is $L^2_a$-mass critical, with $p = p_c$. If the triple $(q,r,\infty)$ is admissible, and we let $\tilde q = p q'$, $\tilde r = p r'$, then also $(\tilde q,\tilde r,\infty)$ is admissible. 
\end{lemma}

\begin{proof}
We discuss the case $a\ge 0$, the regime $-1<a<0$ is dealt with in the same way. We must show that if $q, r$ satisfy \eqref{acinfty}, then we also have
\begin{equation}\label{tildes}
\frac 2{pq'} + \frac{d}{pr'} = \frac{d+a+1}2.
\end{equation}
The left-hand side of \eqref{tildes} equals 
\[
\frac{d+2}p - \frac 1p\left(\frac 2q + \frac dr\right) = \frac{d+2}p - \frac 1p \frac{d+a+1}2 = \frac 1p \frac{d+3-a}2,
\]
where in the second to the last equality we have used \eqref{acinfty}. Since by \eqref{pcpos} in Definition \ref{D:mass} we have $p = \frac{d+3-a}{d+a+1}$, the desired conclusion \eqref{tildes} follows.

\end{proof}

\begin{lemma}[Conservation of mass]\label{L:masscon}
Let $0\le a<1$ and let $q, r$ satisfy \eqref{bdryad} with $r\ge2$, $q>2$; we recall that in this range \eqref{bdryad} coincides with the admissibility relation \eqref{acinfty}, so that Theorem \ref{T:gapbulk} applies with the same pair $(q,r)$. (We do not state the lemma in the anomalous range $-1<a<0$, where the trace of the bulk term provided by Theorem \ref{T:gapbulk} lives in the sum space $L^{q+q_\infty}_t L^r_x$ and the pairing below would require a different functional setting; the lemma is used only in the proof of Theorem \ref{T:well}, where $a\ge0$.) Given $u_0\in L^2_a(\RN)$ and $\Phi\in L^{q'}_tL^{r'}_x$, let $U = \mathbb T^\star_a(u_0) + \Theta^\star_a(\Phi)$ be the mild solution of the linear problem \eqref{nozeroin} with $F=0$ and Neumann datum $\Phi$ (see Remark \ref{R:density}), and let $u = U(\cdot,0,\cdot)\in L^q_tL^r_x$ be its boundary trace, which exists by Theorems \ref{T:gap} and \ref{T:gapbulk}. Then for every $t>0$
\begin{equation}\label{timeder}
\int_{\RN} |U(X,t)|^2 d\omega_a(X) = \int_{\RN} |u_0(X)|^2 d\omega_a(X)\ +\ 2\int_0^t\int_{\Rd} \Im\left(\overline{u(x,s)}\, \Phi(x,s)\right) dx\, ds .
\end{equation}
In particular, if $U$ is a mild solution of \eqref{cp02} with $F=0$, so that $\Phi = -\mu\,|u|^{p-1}u$, then
\begin{equation}\label{timeder2}
\int_{\RN} |U(X,t)|^2 d\omega_a(X) = \int_{\RN} |u_0(X)|^2 d\omega_a(X)\ -\ 2\,\Im(\mu) \int_0^t\int_{\Rd} |u(x,s)|^{p+1}\, dx\, ds,
\end{equation}
and if $\Im(\mu) = 0$ the \emph{mass} is constant in time.
\end{lemma}

\begin{proof}
We first observe that both sides of \eqref{timeder} are continuous functionals of the pair $(u_0,\Phi)\in L^2_a(\RN)\times L^{q'}_tL^{r'}_x$. For the left-hand side this follows from the unitarity of $\mathbb T^\star_a$ on $L^2_a(\RN)$ and from \eqref{ok}; for the right-hand side, from H\"older's inequality
\[
\left|\int_0^t\int_{\Rd} \Im(\bar u\, \Phi)\, dx ds\right| \le ||u||_{L^q_tL^r_x}\, ||\Phi||_{L^{q'}_tL^{r'}_x},
\]
together with the continuous dependence of the trace $u$ on $(u_0,\Phi)$ furnished by \eqref{okk}, \eqref{bulk00} and Theorems \ref{T:gap}, \ref{T:gapbulk}. By density, it thus suffices to establish \eqref{timeder} for $u_0\in \mathcal D_a$ (the class in the proof of Theorem \ref{T:gapbulk}) and $\Phi\in C_0^\infty(\Rd\times(0,\infty))$.

For such data, the representation \eqref{nicea1} shows that $U$ is, for $t>0$, a classical solution of the equation in \eqref{nozeroin} satisfying $\lim_{z\to0^+}z^a\p_zU = \Phi$ (Theorem \ref{T:U} and Remark \ref{R:signcheck}). Moreover, $U(\cdot,z,t)$ is Schwartz in $x$ uniformly on compact $z, t$ sets (the Schr\"odinger group preserves $\mathscr S(\Rd)$), while, since the phase $\sigma\mapsto \frac{z^2}{4\sigma}$ in the boundary term of \eqref{nicea1} has no critical points, repeated integration by parts in $\sigma$ shows that the boundary term and its derivatives decay rapidly as $z\to\infty$; the bulk term decays rapidly in $z$ by the rapid decay of $\mathcal H_{\frac{a-1}2}$-band-limited data and of their evolutions, see Lemma \ref{L:band} and Remark \ref{R:bandt}. These properties justify the following integrations by parts. Omitting the variable of integration, for $t>0$ we compute
\begin{align*}
\frac{d}{dt} \int_{\RN} |U(X,t)|^2 d\omega_a(X) & = \int_{\RN} \left[U_t \bar U + U \bar U_t\right] d\omega_a(X)
\\
& = i\int_{\RN} \left[\bar U (\Delta_x U + \Ba U) - U (\Delta_x \bar U + \Ba \bar U)\right] d\omega_a(X)
\\
& = i \int_{\Rd} \int_0^\infty \left[\bar U (z^a U_z)_z - U (z^a \bar U_z)_z\right] dz\, dx
\\
& = -\, i \int_{\Rd} \left[\bar u\, \Phi - u\, \bar \Phi\right] dx\ =\ 2 \int_{\Rd} \Im\left(\bar u\, \Phi\right) dx,
\end{align*}
where the $x$-part vanishes upon integration by parts, and in the $z$-part the boundary terms at $z=\infty$ vanish by the decay just described, while those at $z=0$ produce $-(\bar u \Phi - u\bar\Phi)$ by the Neumann condition. Integrating in time (the identity \eqref{timeder} holds trivially at $t=0$, and $t\mapsto ||U(t)||^2_{L^2_a}$ is continuous), we obtain \eqref{timeder} for the dense class, and hence in general. Finally, \eqref{timeder2} follows by inserting $\Phi = -\mu|u|^{p-1}u\in L^{q'}_tL^{r'}_x$ (see \eqref{ini2}) and $\Im\left(\bar u\, (-\mu)|u|^{p-1}u\right) = -\Im(\mu)\, |u|^{p+1}$.
\end{proof}

We are ready to present the 

\begin{proof}[Proof of Theorem \ref{T:well}]
Assume that $a\ge 0$ and let $1<p\le p_c$, see \eqref{pcpos}. Suppose that $r = p+1$ and that $(q,r,\infty)$ satisfy the former equation in \eqref{acinfty}, with $q>2$. It is important to observe right away that, since $r-1 = p$, we must have $p r' = r$.

According to Theorems \ref{T:main} and \ref{T:gap}, and keeping \eqref{Cz} in mind, the appropriate Banach space for the problem is 
\begin{equation}\label{space}
\mathcal X= C(\R,L^2_a(\RN)) \cap C^{b}_z L^q_t L^r_x ,
\end{equation}
endowed with its natural norm 
\begin{equation}\label{norm}
||U||_{\Xs}= ||U||_{L^\infty_t L^2_a(\RN)} + ||U||_{L^\infty_z L^q_tL^r_x }.
\end{equation}
The couple $(\Xs,||\cdot||_{\Xs})$ is a complete Banach space. Indeed, each of the two spaces intersected in \eqref{space} is complete. For $C(\R,L^2_a(\RN))$, endowed with $||\cdot||_{L^\infty_t L^2_a}$, this is the standard fact that a uniform limit of continuous Banach space valued maps is continuous. The same argument in the variable $z$ applies to $C^b_z L^q_t L^r_x$, and shows that it is a closed subspace of $L^\infty_z L^q_t L^r_x$: here it is essential that, according to \eqref{Cz}, its norm be the genuine supremum over $z\ge 0$, and not an essential supremum. Finally, both spaces are continuously embedded in $L^1_{\operatorname{loc}}(\RN\times \R)$, and therefore a sequence which is Cauchy for \eqref{norm} has one and the same limit in the two of them. Consequently, $\Xs$ endowed with \eqref{norm} is complete.
In the sequel, given a function $U\in \Xs$, we denote for brevity  $|U|^{p-1} U(0) = |U(\cdot,0,\cdot)|^{p-1} U(\cdot,0,\cdot)$. Since $pr' = r$, we have
\begin{equation}\label{UU}
|||U|^{p-1} U(0)||_{L^{r'}_x} = ||U(0)||_{L^r_x}^p.
\end{equation}
Furthermore, if $U, V\in \Xs$, then we use the well-known inequality
\[
|\ |v|^{p-1} v - |w|^{p-1} w\ |\le C(p) \left[|v|^{p-1} + |w|^{p-1}\right] |v - w|,\ \ \ \ \ \ v, w\in \C,
\]
and H\"older inequality, to bound
\begin{align}\label{UUU}
|| |U|^{p-1} U(0) - |V|^{p-1} V(0)||_{L^{r'}_x} & \le C(p) \big\||U(0)| + |V(0)|\big\|^{p-1}_{L^{pr'}_x} ||U(0)-V(0)||_{L^{pr'}_x}
\\
& = C(p) \big\||U(0)| + |V(0)|\big\|^{p-1}_{L^{r}_x}  ||U(0)-V(0)||_{L^{r}_x}.
\notag
\end{align}

In view of \eqref{sol1} in Definition \ref{solution}, given $u_0\in L^2_a(\RN)$ and $F\in L^1_t L^2_a(\RN) \cap L^1_{a,z} L_t^{q'} L^{r'}_x$, we now consider the operator: 
\begin{align}\label{La}
\Lambda_a(U)(X,t) & := \mathbb T_a^\star(u_0)(X,t) + \mathbb D_a(F)(X,t) - \mu\, \Theta_a^\star(|U|^{p-1} U(0))(X,t).
\end{align}
Then a function $U\in \Xs$ is a mild solution of the problem \eqref{cp02} if it is a fixed point of \eqref{La}, i.e., we have for every $(X,t)\in \RN\times \R$, with $t\not= 0$:
\begin{equation}\label{LaU}
U(X,t) = \Lambda_a(U)(X,t).
\end{equation}
From \eqref{La}, arguing as in the proof of Theorem \ref{T:main} in Section \ref{S:main}, we obtain the estimate  
\begin{align}\label{ini}
& ||\Lambda_a(U)||_{L^\infty_t L^2_a(\RN)} \le ||u_0||_{L^2_a(\RN)} + ||F||_{L^1_t L^2_a(\RN)}  + |\mu|\ ||\Theta^\star_a(|U|^{p-1}U(0))||_{L^\infty_t L^2_a(\RN)}
\\
& \le ||u_0||_{L^2_a(\RN)} + ||F||_{L^1_t L^2_a(\RN)}  + C_1\ |\mu|\ |||U|^{p-1}U(0)||_{L^{q'}_t L^{r'}_x},
\notag
\end{align}
where in the second inequality we have used \eqref{ok} in Theorem \ref{T:V}, and we have let $C_1 = C_1(d,a,r)>0$. 

Furthermore, by \eqref{bulk00}  and \eqref{Dabound}, we have for some $C_2 = C_2(d,a,r)>0$,
\begin{align}\label{bulkotto}
||\mathbb T^\star_a(u_0)||_{L^\infty_z L^q_t L^r_x} + ||\mathbb D_a(F)||_{L^\infty_z L^q_t L^r_x}\le C_2 \left[||u_0||_{L^2_a(\RN)} + ||F||_{L^1_{a,z} L^{q'}_t L^{r'}_x}\right].
\end{align}
From \eqref{norm}, \eqref{La}, \eqref{ini} and \eqref{bulkotto} we find for some $C_3 = C_3(d,a,r)>0$
\begin{align}\label{stelle1}
||\Lambda_a(U)||_{\Xs} & = ||\Lambda_a(U)||_{L^\infty_t L^2_a(\RN)} + ||\Lambda_a(U)||_{L^\infty_z L^q_t L^r_x}
\\
& \le C_3 \left[||u_0||_{L^2_a(\RN)} + ||F||_{L^1_t L^2_a(\RN)} + ||F||_{L^1_{a,z} L^{q'}_t L^{r'}_x}\right]
\notag
\\
& +  C_1\ |\mu|\ |||U|^{p-1}U(0)||_{L^{q'}_t L^{r'}_x}.
\notag
\end{align}
At this point we divide the discussion into the two cases (1) and (2). 

\medskip 

\noindent \underline{(1) The $L^2_a$-mass critical case.} We begin by observing that, if $p = p_c$ in \eqref{pcpos}, then 
\[
r = p+1 = \frac{2(d+2)}{d+a+1},
\]
and the compatibility assumption \eqref{acinfty} of the triple $(q,r,\infty)$ gives
\[
q = \frac{2(d+2)}{d+a+1} = r,
\]
see Remarks \ref{R:rr} and \ref{R:pc}.
Therefore, we have $r = r(d,a)$ and also $pq' = p r' = q$. From this fact and \eqref{UU} we thus find
\begin{align}\label{ini2}
|||U|^{p-1}U(0)||_{L^{q'}_t L^{r'}_x} = ||U(\cdot,0,\cdot)||^p_{L^{q}_t L^{r}_x} \le ||U||^p_{L^\infty_z L^{q}_t L^{r}_x }\le ||U||_{\Xs}^p,
\end{align}
where in the second inequality we have used the trivial bound
\begin{equation}\label{triv}
||U(\cdot,0,\cdot)||_{L^{q}_t L^{r}_x} \le ||U||_{L^\infty_z L^{q}_t L^{r}_x }.
\end{equation} 
Inserting \eqref{ini2} into \eqref{stelle1}, we reach the conclusion that 
\begin{align}\label{stelle5}
&  ||\Lambda_a(U)||_{\Xs} \le C_3 \left[||u_0||_{L^2_a(\RN)} + ||F||_{L^1_t L^2_a(\RN)} + ||F||_{L^1_{a,z} L^{q'}_t L^{r'}_x }\right] + C_1\ |\mu|\ ||U||_{\Xs}^p.
\end{align}

Our goal is to prove that there exists $\rho = \rho(d,a,\mu)>0$ small enough, such that if 
\[
B_\rho = \{U\in \Xs\mid ||U||_{\Xs}\le \rho\},
\]
then:
\begin{itemize}
\item[(i)] $\Lambda_a : B_\rho\ \longrightarrow\ B_\rho$;
\\
\item[(ii)] $||\Lambda_a(U) -\Lambda_a(V)||_{\Xs} \le \frac 12\ ||U-V||_{\Xs}$,
\end{itemize}
i.e. $\Lambda_a$ is a contraction. 

\vskip 0.2in

\noindent \underline{Proof of (i):} Suppose we choose $\rho = \rho(d,a,\mu)>0$ such that
\begin{equation}\label{rho}
C_1\ |\mu|\ \rho^p \le \frac{\rho}2\ \Longleftrightarrow\ 0< \rho \le (2 C_1\ |\mu|)^{-\frac{1}{p-1}}.  
\end{equation}
If $U\in B_\rho$, we thus have 
\[
C_1\ |\mu|\ ||U||_{\Xs}^p \le C_1\ |\mu|\ \rho^p \le \frac{\rho}2.
\]
It is thus clear from \eqref{stelle5} that, if we set $\ve_0 := \frac{\rho}2$, then 
\[
||U||_{\Xs}\le \rho\ \Longrightarrow\ ||\Lambda_a(U)||_{\Xs}\le \rho,
\]
provided that
\begin{equation}\label{smallid}
C_3 \left[||u_0||_{L^2_a(\RN)} + ||F||_{L^1_t L^2_a(\RN)} + ||F||_{L^1_{a,z} L^{q'}_t L^{r'}_x }\right]\  \le\ \ve_0.
\end{equation}
We can thus achieve (i) if \eqref{smallid} is true. We now turn to the 

\vskip 0.2in

\noindent \underline{Proof of (ii):} From the definition \eqref{La} we have
\begin{align}\label{Larap}
\Lambda_a(U) - \Lambda_a(V) = -\mu\, \Theta_a^\star\left(|U|^{p-1} U(0) - |V|^{p-1} V(0)\right).
\end{align}
From the boundary Strichartz estimate \eqref{okk}, we find
\begin{equation}\label{okkk}
||\Theta_a^\star\left(|U|^{p-1} U - |V|^{p-1} V\right)||_{L^\infty_{z} L^q_t L^r_x } \le C\ |||U|^{p-1} U(0) - |V|^{p-1} V(0)||_{L^{q'}_t L^{r'}_x}.
\end{equation}
Using \eqref{UUU}, H\"older inequality, and the fact that $pq' = q$, we easily obtain
\begin{align}\label{U4}
 || |U|^{p-1} U(0) - |V|^{p-1} V(0)||_{L^{q'}_t L^{r'}_x} & \le  C(p) \big\||U(0)| + |V(0)|\big\|^{p-1}_{L^q_t L^{r}_x}  ||U(0)-V(0)||_{L^q_t L^{r}_x}
 \\
& \le C(p) \left[||U(0)||^{p-1}_{L^q_t L^{r}_x}  + ||V(0)||^{p-1}_{L^q_t L^{r}_x}\right]  ||U(0)-V(0)||_{L^q_t L^{r}_x}.
\notag
\end{align}
From \eqref{triv}, \eqref{Larap}, \eqref{okkk} and \eqref{U4}, we conclude for $C_5 = C_5(d,a)>0$,
\begin{equation}\label{okkkk}
||\Lambda_a(U) - \Lambda_a(V)||_{L^\infty_{z} L^q_t L^r_x } \le C_5\ |\mu|\ \left[||U||_{L^\infty_z L^{q}_t L^{r}_x } + ||V||_{L^\infty_z L^{q}_t L^{r}_x}\right]^{p-1} ||U-V||_{L^\infty_z L^{q}_t L^{r}_x},
\end{equation}

If instead we invoke \eqref{ok}, from \eqref{Larap} and proceeding exactly as above, we find
\begin{align}\label{stelle6}
||\Lambda_a(U) - \Lambda_a(V)||_{L^\infty_t L^2_a(\RN)} \le C_6\ |\mu| \left[||U||_{ L^\infty_z L^{q}_t L^{r}_x} + ||V||_{ L^\infty_z L^{q}_t L^{r}_x}\right]^{p-1} ||U-V||_{ L^\infty_z L^{q}_t L^{r}_x}. 
\end{align}
From \eqref{norm}, \eqref{okkkk} and \eqref{stelle6} we conclude that, if $U, V\in B_\rho$, then
\[
||\Lambda_a(U) - \Lambda_a(V)||_{\Xs} \le C_7 |\mu| (2\rho)^{p-1} ||U - V||_{\Xs},
\]
for some $C_7 = C_7(d,a)>0$.
To achieve (ii), it thus suffices to adjust the choice of $\rho$ so that
\[
C_7 |\mu| (2\rho)^{p-1}\le \frac 12.
\]

With (i) and (ii) in hands, the Banach-Caccioppoli theorem guarantees the existence of a unique fixed point $U\in B_\rho$. This is the sought for mild solution to \eqref{cp02}.

\vskip 0.2in 

\noindent \underline{The $L^2_a$-mass subcritical case.} 
We next consider the subcritical case $1<p<p_c$, and let $r = p+1$. Let $q>2$ be such that $(q,r,\infty)$ is admissible, i.e.
\[
\frac 2q +\frac dr = \frac{d+a+1}2.
\]
On a finite interval $[0,T]$, with $0<T<\infty$ to be chosen subsequently, we consider the space $\XT= C([0,T],L^2_a(\RN)) \cap C^{b}_z L^q_T L^r_x$, 
with its natural norm $||\cdot||_{\XT} = ||U||_{L^\infty_T L^2_a(\RN)} + ||U||_{L^\infty_z L^q_TL^r_x }$.
In the subcritical case it is no longer true that $q =r$. In contrast with the critical case, we will show that the problem \eqref{cp02} admits a solution on $[0,T]$, regardless the size of the initial datum $u_0$ and forcing term $F$. The smallness of $T$ will enable us to close the contraction argument. 

Since $1<p<p_c = 1 + \frac{2(1-a)}{d+a+1}$, see \eqref{pcpos}, we presently must have $r<q$. To see this critical fact, note that, in view of the admissibility of $(q,r,\infty)$, $\frac 2q < \frac 2r = \frac{2}{p+1}$ is equivalent to proving
\[
\frac 2q = \frac{d+a+1}2 - \frac{d}{p+1} < \frac{2}{p+1}.
\]
Now, it is easily checked that this latter inequality is equivalent to $p < 1 + \frac{2(1-a)}{d+a+1} = p_c$, which is what we are assuming. Our next observation is that $r<q$ implies (and is in fact equivalent to)
\begin{equation}\label{pq}
p q' < q.
\end{equation}
The validity of \eqref{pq} is obvious since it is equivalent to $p<q-1\ \Longleftrightarrow\  p+1 = r <q$. At this point, we can exploit \eqref{UU}, \eqref{pq} and the smallness of the interval  
to obtain from H\"older inequality:
\begin{align}\label{Upq}
& |||U|^{p-1} U(0)||_{L^{q'}_T L^{r'}_x} 
\le T^{\delta} ||U(\cdot,0,\cdot)||^p_{L^{q}_T L^r_x}
 \le T^\delta ||U||^p_{L^\infty_z L^{q}_T L^{r}_x } \le T^\delta ||U||_{\XT}^p,
\end{align}
where we have let $\delta = \frac{1}{q'} - \frac{p}{q}>0$.
Inserting this estimate in \eqref{stelle1}, we obtain
\begin{align}\label{stelleT}
||\Lambda_a(U)||_{\XT} & \le C_3 \left[||u_0||_{L^2_a(\RN)} + ||F||_{L^1_t L^2_a(\RN)} + ||F||_{L^1_{a,z} L^{q'}_t L^{r'}_x}\right]
 +  C_1\ |\mu|\ T^\delta ||U||_{\XT}^p.
\end{align}
We now define $\rho = \rho(d,a,r, u_0, F)>0$ by the equation
\begin{equation}\label{choose}
\frac{\rho}2\ :=\ C_3 \left[||u_0||_{L^2_a(\RN)} + ||F||_{L^1_t L^2_a(\RN)} + ||F||_{L^1_{a,z} L^{q'}_t L^{r'}_x}\right],
\end{equation}
and set $B^T_\rho = \{U\in \XT\mid ||U||_{\XT}\le \rho\}$.
Inserting this choice of $\rho$ in \eqref{stelleT}, we find
\begin{align}\label{stelleT2}
||\Lambda_a(U)||_{\XT} & \le \frac{\rho}2
 +  C_1\ |\mu|\ T^\delta ||U||_{\XT}^p.
\end{align}
We next select $T = T(d,a,r,p,\mu,u_0,F)>0$ in such a  way that
\[
C_1\ |\mu|\ T^\delta \rho^p \le \frac{\rho}2.
\]
It is clear from \eqref{stelleT2} that, with such a choice of $T$, if $U\in B^T_\rho$, then $\Lambda_a(U)\in B^T_\rho$. 

Next, we show that $\Lambda_a: B^T_\rho\to B^T_\rho$ is a contraction. Proceeding as in \eqref{Larap} and \eqref{okkk}, we find
\begin{equation}\label{okkk2}
||\Lambda_a(U) - \Lambda_a(V)||_{L^\infty_{z} L^q_T L^r_x } \le C |\mu| \ |||U|^{p-1} U(0) - |V|^{p-1} V(0)||_{L^{q'}_T L^{r'}_x}.
\end{equation}
Integrating \eqref{UUU} on $[0,T]$ and using H\"older inequality, we have
\begin{align*}
& |||U|^{p-1} U(0) - |V|^{p-1} V(0)||_{L^{q'}_T L^{r'}_x} \le C(p) \left(\int_0^T \big\||U(0)| + |V(0)|\big\|^{(p-1)q'}_{L^{r}_x} ||U(0) - V(0)||^{q'}_{L^{r}_x} dt\right)^{\frac{1}{q'}}
\\
& \le C(p) \left(\int_0^T \big\||U(0)| + |V(0)|\big\|^{(p-1)q' p'}_{L^{r}_x}\right)^{\frac{1}{q'p'}} \left(\int_0^T ||U(0) - V(0)||^{q' p}_{L^{r}_x}\right)^{\frac{1}{q'p}}
\\
& \le C(p) \left(||U(0)||^{p-1}_{L^{pq'}_T L^{r}_x} + ||V(0)||^{p-1}_{L^{pq'}_T L^{r}_x}\right) ||U(0) - V(0)||_{L^{pq'}_T L^{r}_x}.
\end{align*}
In view of \eqref{pq} we obtain as in \eqref{Upq}
\begin{align*}
& \left(||U(0)||^{p-1}_{L^{pq'}_T L^{r}_x} + ||V(0)||^{p-1}_{L^{pq'}_T L^{r}_x}\right) ||U(0) - V(0)||_{L^{pq'}_T L^{r}_x} 
\\
& \le T^{\delta} \left(||U(0)||^{p-1}_{L^{q}_T L^{r}_x} + ||V(0)||^{p-1}_{L^{q}_T L^{r}_x}\right) ||U(0) - V(0)||_{L^{q}_T L^{r}_x}
\\
& \le T^{\delta} \left(||U||^{p-1}_{L^\infty_z L^{q}_T L^{r}_x} + ||V||^{p-1}_{L^\infty_z L^{q}_T L^{r}_x}\right) ||U - V||_{L^\infty_z L^{q}_T L^{r}_x}
\\
& \le T^{\delta} \left(||U||^{p-1}_{\XT} + ||V||^{p-1}_{\XT}\right) ||U - V||_{\XT}.
\end{align*}

Inserting these inequalities in \eqref{okkk2}, we finally obtain for a certain constant $C_8 = C_8(d,a,p)>0$
\begin{equation}\label{okkk3}
||\Lambda_a(U) - \Lambda_a(V)||_{L^\infty_{z} L^q_T L^r_x } \le C_8 |\mu| T^{\delta} \left(||U||^{p-1}_{\XT} + ||V||^{p-1}_{\XT}\right) ||U - V||_{\XT},
\end{equation}
and similarly
\begin{align}\label{stelle7}
||\Lambda_a(U) - \Lambda_a(V)||_{L^\infty_T L^2_a(\RN)} \le C_8\ |\mu| T^{\delta} \left(||U||^{p-1}_{\XT} + ||V||^{p-1}_{\XT}\right) ||U - V||_{\XT}. 
\end{align}
Combining \eqref{okkk3} with \eqref{stelle7}, we finally arrive to the estimate
\begin{align}\label{stelle8}
||\Lambda_a(U) - \Lambda_a(V)||_{\XT} \le 2 C_8\ |\mu| T^{\delta} \left(||U||^{p-1}_{\XT} + ||V||^{p-1}_{\XT}\right) ||U - V||_{\XT}. 
\end{align}
If we keep in mind that $\rho$ has been fixed as in \eqref{choose}, it is evident that, by possibly adjusting $T>0$ so that
\begin{equation}\label{Tdelta}
4 C_8 |\mu| \rho^{p-1} T^{\delta} \le \frac 12,
\end{equation}
we have proved that for every $U, V \in B^T_\rho$, we have
\[
||\Lambda_a(U) - \Lambda_a(V)||_{\XT} \le \frac 12 ||U-V||_{\XT}.
\]
Therefore, $\Lambda_a$ admits a unique fixed point $U\in B^T_\rho$. This proves the existence of a mild solution $U$ to the problem \eqref{cp02} in the ball $B^T_\rho$. By standard arguments based on \eqref{stelle8}, one can show that such $U$ is actually the unique mild solution in the whole space $\XT$. 

Finally, assume that $F= 0$ and that $\Im(\mu) = 0$. We want to show that, for every $M>0$, the solution $U\in \XT$ constructed above can be extended to a mild solution on $[-M,M]$, i.e.\ an element of $\mathcal X_M$ (we caution that the concatenated solution need not belong to the global space $\Xs$, since its $L^q_tL^r_x$ norm may grow with $M$). With this objective in mind, we note that \eqref{choose} shows that  $\rho=\rho(\|u_0\|_{L^2_a(\RN)})$. As a consequence of this and \eqref{Tdelta}, also $T$ depends only on $\|u_0\|_{L^2_a(\RN)}$ (besides the parameters $d, a, p$ and $\mu$). If we now solve the Cauchy problem \eqref{cp02} in $\RN\times[T,2T]$, with initial datum $U_1 = U(\cdot,\cdot,T)$, then in view of Lemma \ref{L:masscon}, we have $\|U_1\|_{{L^2_a(\RN)}}= \|u_0\|_{L^2_a(\RN)}$, and thus we can advance up to time  $2T$; that the concatenation of the two fixed points is again a mild solution on $[0,2T]$ follows from the group property of $\mathbb T^\star_a$ and the additivity of the Duhamel term in \eqref{sol1} under splitting of the time interval. The same argument applies backward in time. As a consequence, for any fixed arbitrary time interval $M>0$, with about $2M/T$ steps we obtain a mild solution on $[-M,M]$, unique in $\mathcal X_M$ by the uniqueness statement just proved on each subinterval.  

\end{proof}

We finally present the 

\begin{proof}[Proof of Theorem \ref{T:nega}]

We presently have $-1<a<0$, $1<p\le p_c$, with $p_c := 1 + \frac{2}{d+1}$, see \eqref{pcneg}, and we are assuming that $r=p+1$. As before, this gives $p r' = r$. We also have the hypothesis that $q_\infty, q, r$ satisfy the equations \eqref{uffaa}, i.e.,
\[
\frac{2}{q_\infty} + \frac dr = \frac{d+a+1}2,\ \ \ \ \text{and}\ \ \ \ \ \frac 2{q} + \frac dr = \frac{d+1}2, 
\]
with $q>2$.
Note that these equations imply $q<q_\infty$, or equivalently $q_\infty'<q'$. To see this, observe that
\[
\frac{2}{q_\infty} = \frac{d+a+1}2 - \frac dr  = \frac 2q + \frac a2.
\]
Since $-1<a<0$, we have $\frac{2}{q_\infty} < \frac{2}{q}$.

 For a fixed $0<T<\infty$, we recall the mixed spaces defined by \eqref{LmixT}. We will continue to use the notation \eqref{Lqrp} for the mixed spaces on the whole line $t\in \R$. We also recall the spaces $L^m_{a,z} L^{q_1+q_2}_t L^r_x$ and $L^m_{a,z} L^{q_1\cap q_2}_t L^r_x $, see \eqref{LmixmixT}. 
 
Let $u_0\in L^2_a(\RN)$ and $F\in L^1_T L^2_a(\RN)$ be such that 
$F  k^{-1}\in L^1_{a,z} L^{q'\cap q'_\infty}_T L^{r'}_x $. Presently, the Banach space in which we seek a mild solution is
\begin{equation}\label{Xneg}
\XT = \{U\in C([0,T],L^2_a(\RN))\ \text{such that}\  z\to U(\cdot,z,\cdot) k(z)\ \in C^b_z L^q_T L^r_x \},
\end{equation}
with its natural norm
\begin{equation}\label{normU}
||U||_{\XT} = ||U||_{L^\infty_T L^2_a(\RN)}  + ||U k||_{L^\infty_z L^q_T L^r_x }.
\end{equation}
Note that, because of the presence of the weight $k$,  \eqref{Xneg} is different from the space \eqref{space} used in the proof of Theorem \ref{T:well}.

Given $U, V\in \XT$, we define two functions $\Phi, \Psi:\Rd\times \R\to \C$ by setting
\begin{equation}\label{F}  
\begin{cases}
\Phi(x,t) = |U(x,0,t)|^{p-1} U(x,0,t), \ \ \ \Psi(x,t) = |V(x,0,t)|^{p-1} V(x,0,t),\ \ x\in \Rd, 0\le t\le T,
\\
\Phi(x,t) = \Psi(x,t) = 0,\ \ \ \ \text{when}\ x\in \Rd,\ \text{and either}\ t<0\ \text{or}\ t>T.
\end{cases}
\end{equation}

Also, given a forcing term $F:\RN\times[0,T]\to \C$ such that $||F k^{-1}||_{L^1_{a,z} L^{q'}_T L^{r'}_x}<\infty$, we define
\[
\tilde F(X,t) := \begin{cases}
F(X,t),\ \ \ \  X\in \RN, 0\le t\le T,
\\
0,\ \ \ \ \text{when}\ X\in \RN,\ \text{and either}\ t<0\ \text{or}\ t>T.
\end{cases}
\]
Since $q_\infty'<q'$, we have
\[
||\tilde F k^{-1}||_{L^1_{a,z} L^{q_\infty'}_t L^{r'}_x } = ||F k^{-1}||_{L^1_{a,z} L^{q_\infty'}_T L^{r'}_x }\le T^{\frac{1}{q_\infty'}- \frac{1}{q'}} ||F k^{-1}||_{L^1_{a,z} L^{q'}_T L^{r'}_x }. 
\]
Noting that 
\[
\frac{1}{q_\infty'}- \frac{1}{q'} = \frac 1q - \frac{1}{q_\infty} := \gamma,
\]
we thus have 
\begin{align}\label{tildeF}
||\tilde F k^{-1}||_{L^1_{a,z} L^{q_\infty'\cap q'}_t L^{r'}_x } & \le ||\tilde F k^{-1}||_{L^1_{a,z}L^{q_\infty'}_t L^{r'}_x} + ||\tilde F k^{-1}||_{L^1_{a,z} L^{q'}_t L^{r'}_x}
\\
& \le (1+T^{\frac{1}{q_\infty'}- \frac{1}{q'}}) ||F k^{-1}||_{L^1_{a,z} L^{q'}_T L^{r'}_x }
\notag\\
& \le (1+T^{\gamma}) ||F k^{-1}||_{L^1_{a,z} L^{q'}_T L^{r'}_x } <\infty,
\notag
\end{align} 
by the assumption on $F$.
Moreover, as in \eqref{UU}
we have
\[
||\Phi||_{L^{r'}_x} = ||U(0)||_{L^r_x}^p,
\]
and therefore, using the fact that $q_\infty'<q'$, and that $k(z) = 1$, for $0\le z\le 1$ (see Definition \ref{D:mn}), we have 
\begin{align}\label{Phibound}
||\Phi||_{L^{q_\infty'}_t L_x^{r'}} & = ||\Phi||_{L^{q_\infty'}_T L_x^{r'}}\le T^{\frac{1}{q_\infty'}-\frac{1}{q'}} ||\Phi||_{L^{q'}_T L_x^{r'}} = T^{\gamma}  ||U(0)||_{L^{pq'}_T L^r_x}^p 
= T^{\gamma} ||U(0) k(0)||_{L^{pq'}_T L^r_x}^p.
\end{align}
An estimate similar to \eqref{Phibound} holds for the function $\Psi$ in \eqref{F}.

\medskip 

\noindent \underline{(1) The $L^2_a$-mass critical case.} If $p = p_c$ in \eqref{pcneg}, i.e., $p = 1 + \frac{2}{d+1}$, and 
$r = p +1$, we make the key observation that the second equation in \eqref{uffaa} implies that it must be
\[
q = r = \frac{2(d+2)}{d+1},
\]
see \eqref{same0} in Remark \ref{R:rr}. As a consequence, 
\[
p q' = p r' = r = q,
\]
and \eqref{Phibound} presently gives
\begin{equation}\label{Phiphi}
||\Phi||_{L^{q_\infty'}_t L_x^{r'}} \le T^{\gamma} ||U(0) k(0)||_{L^{q}_T L^r_x}^p \le  T^{\gamma} ||U k||_{L^\infty_z L^{q}_T L^r_x}^p <\infty,
\end{equation}
with the last inequality following from the assumption $U\in \XT$, see \eqref{normU}.

This shows that the functions $\Phi, \Psi$ defined by \eqref{F}, and the forcing term $\tilde F$, satisfy the hypothesis of the linear Theorem \ref{T:main2}, and therefore the functions $\Lambda_a(U), \Lambda_a(V)$, constructed by the formula \eqref{nicea1in}, satisfy \eqref{Umain20}, \eqref{Umain22}. In particular, from \eqref{Umain22} we obtain for some constant $C_9 = C_9(d,a,p)>0$, 
\begin{align}\label{bulkettoo}
& ||\Lambda_a(U) k||_{L^\infty_z L^{q+q_\infty}_t L^r_x } \le C_9\ \left[||u_0||_{L^2_a(\RN)} +  |\mu|\ ||\Phi||_{L^{q'_\infty}_t L^{r'}_x} + ||\tilde F k^{-1}||_{L^1_{a,z} L^{q'\cap q'_\infty}_t L^{r'}_x }\right]
\\
& \le C_9\ \left[||u_0||_{L^2_a(\RN)}  + (1+T^{\gamma}) ||F k^{-1}||_{L^1_{a,z} L^{q'\cap q'_\infty}_T L^{r'}_x }\right] + C_9\ |\mu| T^{\gamma}  ||U k||_{L^\infty_z L^q_T L^r_x }^p,
\notag
\end{align}
where in the second inequality we have used \eqref{tildeF} and \eqref{Phiphi}, together with the trivial bound $||\cdot||_{L^{q'}_T}\le ||\cdot||_{L^{q'\cap q'_\infty}_T}$.
By the second inequality in \eqref{sumT} of Lemma \ref{L:sumT}, applied with $E = L^r_x(\Rd)$ for each fixed $z\ge 0$, we finally obtain from \eqref{bulkettoo}
\begin{align}\label{bulkettooo}
||\Lambda_a(U) k||_{L^\infty_z L^{q}_T L^r_x } & \le  C_9 \max\{1,T^{\gamma}\}\left[||u_0||_{L^2_a(\RN)}  + (1+T^{\gamma}) ||F k^{-1}||_{L^1_{a,z} L^{q'\cap q'_\infty}_T L^{r'}_x }\right]
\\
& + C_9 |\mu| T^{\gamma}\  ||U k||_{L^\infty_z L^q_T L^r_x }^p.
\notag
\end{align}

Furthermore, we have from \eqref{Umain20}:
\begin{align}\label{Umain200}
||\Lambda_a(U)||_{L^\infty_t L^2_a(\RN)} & \le C_9 \left[||u_0||_{L^2_a(\RN)} + ||\tilde F||_{L^1_t L^2_{a}(\RN)} + |\mu| ||\Phi||_{L^{q'_\infty}_t L^{r'}_x} \right]
\\
& \le C_9 \left[||u_0||_{L^2_a(\RN)} + ||F||_{L^1_T L^2_{a}(\RN)} + |\mu| T^{\gamma}  ||U k||_{L^\infty_z L^q_T L^r_x }^p\right],
\notag
\end{align}
where we have used \eqref{Phiphi} again. In view of \eqref{normU}, the estimates  \eqref{bulkettooo}, \eqref{Umain200} prove that:
\begin{align}\label{ball}
||\Lambda_a(U)||_{\XT} & \le C_9 \max\{1,T^{\gamma}\}\left[||u_0||_{L^2_a(\RN)} + ||F||_{L^1_T L^2_a(\RN)} + (1+T^{\gamma}) ||F k^{-1}||_{L^1_{a,z} L^{q'\cap q'_\infty}_T L^{r'}_x }\right]
\\
& + C_9 |\mu| T^{\gamma}\  ||U||_{\XT}^p.
\notag
\end{align}
It is then clear that, if we choose $\rho = \rho(d,a,\mu,T)>0$ such that 
\[
C_9 |\mu| T^{\gamma}  \rho^p\le \frac{\rho}2\ \Longleftrightarrow\ \rho\le \left(2 C_9|\mu|T^{\gamma}\right)^{-\frac{1}{p-1}}, 
\]
and then we take $\ve_0=\ve_0(d,a,\mu,T) := \frac{\rho}2$, then if 
\[
 C_9 \max\{1,T^{\gamma}\} \left[\|u_0\|_{L^2_a(\RN)} + ||F||_{L^1_T L^2_a(\RN)} + (1+T^{\gamma}) ||F k^{-1}||_{L^1_{a,z} L^{q'\cap q'_\infty}_T L^{r'}_x }\right]\le \varepsilon_0,
\]
we have
\[
||U||_{\XT} \le \rho\ \Longrightarrow\ ||\Lambda_a(U)||_{\XT} \le \rho.
\]

Finally, setting $B_\rho = \{U\in \XT\mid ||U||_{\XT}\le \rho\}$, we want to prove that for any $U, V\in B_\rho$, we have
\begin{equation}\label{Lacon}
||\Lambda_a(U) - \Lambda_a(V)||_{\XT} \le \frac 12 ||U - V||_{\XT}.
\end{equation}
To prove \eqref{Lacon} we oserve that \eqref{Larap} and \eqref{F} give 
\begin{align}\label{Larap2}
\Lambda_a(U) - \Lambda_a(V) = -\mu\, \Theta_a^\star\left(\Phi - \Psi\right).
\end{align}
Arguing as for \eqref{Phibound}, but now using \eqref{UUU} and H\"older inequality as in \eqref{U4}, we obtain
\[
||\Phi - \Psi||_{L^{q_\infty'}_t L_x^{r'}} \le C(p)\, T^{\gamma} \left(||U k||_{L^\infty_z L^q_T L^r_x } + ||V k||_{L^\infty_z L^q_T L^r_x }\right)^{p-1} ||(U-V) k||_{L^\infty_z L^q_T L^r_x }.
\]
From Theorem \ref{T:main2} we thus find for $C_{10} = C_{10}(d,a,p)>0$
\[
||(\Lambda_a(U) - \Lambda_a(V))k||_{L^\infty_z L^{q+q_\infty}_t L^r_x } \le  C_{10} |\mu| T^{\gamma} \left(||U k||_{L^\infty_z L^q_T L^r_x } + ||V k||_{L^\infty_z L^q_T L^r_x }\right)^{p-1} ||(U - V)k||_{L^\infty_z L^q_T L^r_x }.
\]
From this bound, arguing as for \eqref{okkkk}, \eqref{bulkettooo}
we obtain
\begin{align*}
||(\Lambda_a(U) - \Lambda_a(V)) k||_{L^\infty_z L^{q}_T L^r_x } & \le  C_{10}\ |\mu| \max\{1,T^{\gamma}\} T^{\gamma}\\
& \times \left(||Uk||_{L^\infty_z L^q_T L^r_x } + ||Vk||_{L^\infty_z L^q_T L^r_x }\right)^{p-1} ||(U - V)k||_{L^\infty_z L^q_T L^r_x }.
\end{align*}
From this estimate it is finally clear that, provided that we further adjust our choice of $\rho = \rho(d,a,\mu,T)>0$ so that
\[
C_{10}\ |\mu| \max\{1,T^{\gamma}\} T^{\gamma} (2\rho)^{p-1} \le \frac 12,
\]
we achieve \eqref{Lacon} for all $U, V\in B_\rho$. This finally proves that $\Lambda_a : B_\rho\to B_\rho$ is a contraction, and therefore there exists a unique function $U\in B_\rho$ such that 
\[
\Lambda_a(U) = U.
\]

\medskip 

\noindent \underline{(2) The $L^2_a$-mass subcritical case.} If $1<p < p_c$, then it is no longer true that, with $r = p+1$ and $2<q$ satisfying \eqref{uffaa}, we have $q = r$. What we presently have is: $r<q$. In fact, the inequality $\frac 2q < \frac 2r$ is equivalent to $\frac 2q < \frac{2}{p+1}$ and, in view of the admissibility of $(q,r,\infty)$, this is equivalent to proving
\[
\frac 2q = \frac{d+1}2 - \frac{d}{p+1} < \frac{2}{p+1},
\]
which is equivalent to $p < 1 + \frac{2}{d+1} = p_c$. Our next observation is that $r<q$ implies (and is in fact equivalent to)
\begin{equation}\label{pq2}
p q' < q.
\end{equation}
The validity of \eqref{pq2} is obvious since it is equivalent to $p<q-1\ \Longleftrightarrow\  p+1 = r <q$. 

We now provide the details of the contraction argument. Let $0<T_0\le \min\{T,1\}$, to be selected, and note that, since $T_0\le1$, we have $\max\{1,T_0^{\gamma}\}=1$, so that $||\cdot||_{L^q_{T_0}}\le ||\cdot||_{L^{q+q_\infty}_{T_0}}$. Let $U\in \mathcal X_{T_0}$ and let $\Phi$ be as in \eqref{F} (with $T$ replaced by $T_0$). Since $k(0)=1$, the boundary trace satisfies $||U(\cdot,0,\cdot)||_{L^q_{T_0}L^r_x}\le ||Uk||_{L^\infty_zL^q_{T_0}L^r_x}$. By \eqref{pq2} and H\"older's inequality in time, exactly as in \eqref{Upq}, with $\delta := \frac{1}{q'}-\frac pq>0$ we have
\[
||\Phi||_{L^{q'}_{T_0}L^{r'}_x}\ \le\ T_0^{\delta}\, ||U(\cdot,0,\cdot)||^p_{L^q_{T_0}L^r_x}\ \le\ T_0^{\delta}\, ||U||^p_{\mathcal X_{T_0}},
\]
and consequently, since $q_\infty'<q'$,
\[
||\Phi||_{L^{q_\infty'}_{T_0}L^{r'}_x}\ \le\ T_0^{\gamma}\, ||\Phi||_{L^{q'}_{T_0}L^{r'}_x}\ \le\ T_0^{\gamma+\delta}\, ||U||^p_{\mathcal X_{T_0}} .
\]
Feeding these bounds into \eqref{Umain20} and \eqref{Umain22}, as in the critical case, we obtain for some $C_{11} = C_{11}(d,a,p)>0$
\[
||\Lambda_a(U)||_{\mathcal X_{T_0}}\ \le\ C_{11}\left[||u_0||_{L^2_a(\RN)} + ||F||_{L^1_{T_0}L^2_a(\RN)} + 2\, ||Fk^{-1}||_{L^1_{a,z}L^{q'\cap q'_\infty}_{T_0}L^{r'}_x}\right]\ +\ C_{11}\, |\mu|\, T_0^{\delta}\, ||U||^p_{\mathcal X_{T_0}} .
\]
We now define $\rho>0$ by letting $\frac\rho2$ equal the first bracket (times $C_{11}$), and let $B^{T_0}_\rho$ denote the ball of radius $\rho$ in $\mathcal X_{T_0}$. If $T_0$ is chosen so small that $C_{11}|\mu| T_0^{\delta}\rho^{p-1}\le \frac12$, then $\Lambda_a: B^{T_0}_\rho\to B^{T_0}_\rho$. For the contraction estimate we argue as in the critical case above, now keeping the H\"older-in-time gain: by \eqref{UUU} and H\"older's inequality, as in the subcritical case of Theorem \ref{T:well},
\[
||\Phi-\Psi||_{L^{q_\infty'}_{T_0}L^{r'}_x}\ \le\ C(p)\, T_0^{\gamma+\delta} \left(||Uk||_{L^\infty_zL^q_{T_0}L^r_x} + ||Vk||_{L^\infty_zL^q_{T_0}L^r_x}\right)^{p-1} ||(U-V)k||_{L^\infty_zL^q_{T_0}L^r_x},
\]
whence, by Theorem \ref{T:main2},
\[
||\Lambda_a(U)-\Lambda_a(V)||_{\mathcal X_{T_0}}\ \le\ C_{12}\, |\mu|\, T_0^{\delta}\, (2\rho)^{p-1}\, ||U-V||_{\mathcal X_{T_0}},\ \ \ \ \ U, V\in B^{T_0}_\rho .
\]
After possibly further shrinking $T_0$ -- note that $\rho$ does not depend on $T_0$, since the data norms over $[0,T_0]$ are dominated by those over $[0,T]$ -- we obtain a contraction, hence a unique fixed point in $B^{T_0}_\rho$; uniqueness in all of $\mathcal X_{T_0}$ then follows by the standard argument based on the last displayed estimate, as in the proof of Theorem \ref{T:well}. Since $\rho$ depends only on the data norms, and $T_0$ only on $d, a, p, |\mu|$ and $\rho$, the proof is complete. 

\end{proof}

\section{A restriction theorem}\label{S:res}

In our recent work \cite{GS} we have proved 
the following restriction result for the Fourier-Hankel transform. Let 
\[
\mathcal H_\nu(\vf)(z) = z^{-\nu}\int_0^\infty \vf(\zeta) J_\nu(z\zeta) \zeta^{\nu+1} d\zeta,\ \ \ \ \Re \nu>-1,
\]
be the modified Hankel transform of a measurable function $\vf:(0,\infty)\to \overline{\C}$
and denote by 
\begin{equation}\label{FH}
\widetilde{\mathcal H}_\nu(u)(\zeta,\tau) = \int_{\R} e^{-2\pi i \tau t} \mathcal H_\nu(u(\cdot,t))(\zeta) dt,
\end{equation}
the Fourier-Hankel transform of a function $u:(0,\infty)\times \R \to \overline{\C}$.

\begin{theorem}\label{T:restriction}
Let $a\ge 0$ and suppose that $q, r$ satisfy 
\[
\frac 2q + \frac{a+1}r = \frac{a+1}2.
\]
For every $u\in L^{q'}_t L^{r'}_a$, we have 
\begin{equation}\label{re1}
\left(\int_0^\infty |\widetilde{\mathcal{H}}_{\frac{a-1}2}(u)(\zeta,-\frac{\zeta^2}{2\pi})|^2 d\omega_a(\zeta) \right)^{1/2}  \le C(a,r)\ ||u||_{L^{q'}_t L^{r'}_a}.
\end{equation}
In particular, when $q'=r' = \frac{2(a+3)}{a+5}$, we obtain the following fractal version of the \emph{Tomas-Stein restriction inequality} for the half-parabola
$P = \{(\zeta,\tau)\in \R^2\mid \tau = - 2\pi \zeta^2,\ \zeta\ge 0\}$
\begin{equation}\label{re2}
\left(\int_0^\infty |\widetilde{\mathcal{H}}_{\frac{a-1}2}(u)(\zeta,-\frac{\zeta^2}{2\pi})|^2 d\omega_a(\zeta) \right)^{1/2}  \le C(a)\ \left(\int_{\R} \int_0^\infty |u(z,t)|^{\frac{2(a+3)}{a+5}} d\omega_a(z) dt\right)^{\frac{a+5}{2(a+3)}}.
\end{equation}
\end{theorem}

In this section we generalize Theorem \ref{T:restriction} as follows.

\begin{theorem}\label{T:resgen}
Let $a\ge 0$ and assume that the triple $(q,r,r)$ is \emph{admissible} according to \eqref{ac3}, i.e.,
\begin{equation}\label{tre}
\frac 2q = (d+a+1)(\frac 12 - \frac 1r),
\end{equation}
with $q, r >2$.
Then  $\mathbb T_a: L^{q'}_t L^{r'}_a(\RN) \to L^2_a(\RN)$, and moreover there exists $C(d,a,r)>0$ such that for every $F\in L^{q'}_t L^{r'}_a(\RN)$ one has
\begin{equation}\label{resgen}
||\mathbb T_a(F)||_{L^2_a(\RN)} \le C(d,a,r) ||F||_{L^{q'}_t L^{r'}_a(\RN)}.
\end{equation} 
In particular, when $(q,q)$ satisfy \eqref{tre}, we obtain for any $F\in L^{\frac{2(d+a+3)}{d+a+5}}_a(\RN\times \R)$
\begin{equation}\label{resgen2}
||\mathbb T_a(F)||_{L^2_a(\RN)} \le C(d,a) ||F||_{L^{\frac{2(d+a+3)}{d+a+5}}_a(\RN\times \R)},
\end{equation} 
for some $C(d,a)>0$.
\end{theorem}

\begin{proof}
Let $F\in \snn$. Arguing as in \eqref{pop}, using H\"older inequality, we have

\begin{align}\label{popone}
||\mathbb T_a(F)||^2_{L_a^2(\RN)} & = \sa \mathbb T_a(F),\mathbb T_a(F)\da = \sa\sa \mathbb T_a^\star \mathbb T_a(F),F\da\da \le ||\mathbb T_a^\star \mathbb T_a(F)||_{L^q_t L_a^r(\RN)} ||F||_{L^{q'}_t L^{r'}_a(\RN)}.
\end{align}
To complete the proof of \eqref{resgen} it suffices to show that there exists $C(d,a,r)>0$ such that 
\begin{equation}\label{popone2}
||\mathbb T_a^\star \mathbb T_a(F)||_{L^q_t L_a^r(\RN)} \le C(d,a,r) ||F||_{L^{q'}_t L^{r'}_a(\RN)}.
\end{equation}
Keeping \eqref{stars} in mind, we see that, for any $2\le r\le \infty$ and every $t\in \R$, we have 
\begin{equation}\label{nova}
||\mathbb T_a^\star \mathbb T_a(F)(\cdot,t)||_{L^r_a(\RN)} \le \int_\R ||\mathbb S_a(t-\tau)(F(\cdot,\tau))||_{L^r_a(\RN)} d\tau.
\end{equation}
Since $a\ge 0$, from \eqref{grandeSaS0}, \eqref{prop2} and \eqref{good}, we obtain
\begin{align}\label{nova2}
|\mathbb S_a(X,Y,t)|
\le 
\frac{C(d,a)}{|t|^{\frac{d+a+1}{2}}},
\qquad \ \ \ \ t\neq 0.
\end{align}
In view of \eqref{prop}, the estimate \eqref{nova2}  immediately gives for every $t\not= 0$
\begin{equation}\label{nova3}
||\mathbb S_a(t)(F)||_{L^\infty(\RN)} \le \frac{C(d,a)}{|t|^{\frac{d+a+1}{2}}} ||F||_{L^1_a(\RN)}.
\end{equation}
Combining Proposition \ref{P:Uni} with \eqref{nova3} and the Riesz-Thorin interpolation theorem, we conclude that for any $2\le r\le \infty$ and every $t\not= 0$, we have
\begin{equation}\label{nova4}
||\mathbb S_a(t)(F)||_{L^r_a(\RN)} \le \frac{C(d,a,r)}{|t|^{(d+a+1)(\frac{1}{2}-\frac 1r)}} ||F||_{L^{r'}_a(\RN)}.
\end{equation}
Using \eqref{nova4} in \eqref{nova}, we find
\begin{equation}\label{nova5}
||\mathbb T_a^\star \mathbb T_a(F)(\cdot,t)||_{L^r_a(\RN)} \le C(d,a,r) \int_\R \frac{||F(\cdot,\tau)||_{L^{r'}_a(\RN)}}{|t-\tau|^{(d+a+1)(\frac{1}{2}-\frac 1r)}} d\tau = C(d,a,r)\ I_\beta(h)(t),
\end{equation}
where we have let $h(\tau) = ||F(\cdot,\tau)||_{L^{r'}_a(\RN)}$, and $I_\beta$ denotes as before the fractional integration operator \eqref{marcel} on $\R$, with 
\[
1-\beta = (d+a+1)(\frac{1}{2}-\frac 1r)= \frac 2q,
\] 
where in the last equality we have used \eqref{tre}. The assumption $q, r>2$ implies $0<\beta<1$, and we thus obtain \eqref{popone2} from \eqref{nova5} and the Hardy-Littlewood-Sobolev theorem. This proves \eqref{resgen}. The estimate \eqref{resgen2} is an immediate corollary of \eqref{resgen} and of the observation that $(q,q,q)$ satisfies \eqref{ac3} if and only if 
\[
q = \frac{2(d+a+3)}{d+a+1}\ \ \ \ \text{and}\ \ \ \ q' = \frac{2(d+a+3)}{d+a+5}.
\]

\end{proof}

We close the section by showing that, similarly to the proof of Theorem \ref{T:restriction} in \cite{GS}, the inequalities \eqref{resgen} and \eqref{resgen2} imply a generalization of the Tomas-Stein restriction theorem for the Fourier-Hankel transform in $\RN\times \R$. Given $U\in C^\infty_0(\RN\times \R)$, we define
\[
\tilde{\mathcal H}_\nu(U)(\xi,\zeta,\tau) = \mathcal H_{\nu}(\mathscr F_{(x,t)\to(\xi,\tau)} U)(\zeta),
\]
where $\mathscr F_{(x,t)\to(\xi,\tau)} U$ indicates the Fourier transform in the variables $(x,t)\in\Rd\times \R$. 
Consider the portion of paraboloid 
\[
\Sigma = \{(\xi,\zeta,\tau)\in \RN\times \R\mid \tau = - \frac{4\pi^2 |\xi|^2 +\zeta^2}{2\pi}\}.
\]
Finally, we set
\[
\hat U(\xi,z,t) = \int_{\Rd} e^{-2\pi i\sa\xi,x\da} U(x,z,t) dx.
\]
With these notations, the restriction of $\tilde{\mathcal H}_\nu(U)$ to $\Sigma$ is given by
\begin{equation}\label{restrSigma}
\tilde{\mathcal H}_\nu(U)(\xi,\zeta,- \frac{4\pi^2 |\xi|^2 +\zeta^2}{2\pi}) = \int_{\R} e^{2\pi i t (\frac{4\pi^2 |\xi|^2 +\zeta^2}{2\pi})}  \mathcal H_{\nu}(\hat U(\xi,\cdot,t))(\zeta) dt.
\end{equation}
We can now state the announced result, which is the analogue in the present setting of the Tomas-Stein restriction theorem.

\begin{corollary}[A Tomas-Stein restriction theorem]\label{C:TS}
Let $a\ge 0$ and let $(q,r,r)$ be admissible as in \eqref{tre}, with $q,r>2$. Then, with $\nu = \frac{a-1}2$, the restriction to the paraboloid $\Sigma$ of the Fourier-Hankel transform $\tilde{\mathcal H}_{\nu}(U)$ is a bounded operator from $L^{q'}_t L^{r'}_a(\RN)$ into $L^2_a(\RN)$. More precisely, there exists $C(d,a,r)>0$ such that for every $U\in C^\infty_0(\RN\times \R)$ one has
\begin{equation}\label{resgen4}
\Big\|\tilde{\mathcal H}_{\frac{a-1}2}(U)\Big(\cdot,\cdot,- \frac{4\pi^2 |\cdot|^2 +(\cdot)^2}{2\pi}\Big)\Big\|_{L^2_a(\RN)}\le C(d,a,r) ||U||_{L^{q'}_t L^{r'}_a(\RN)}.
\end{equation}
In particular, when $q = r = \frac{2(d+a+3)}{d+a+1}$, the right-hand side of \eqref{resgen4} can be replaced by $C(d,a) ||U||_{L^{\frac{2(d+a+3)}{d+a+5}}_a(\RN\times \R)}$.
\end{corollary}

\begin{proof}
Using \eqref{Ta} in Lemma \ref{L:T}, the formula from \cite{GS}
\begin{equation}\label{hank}
\mathcal H_{\frac{a-1}2}(S_a(t) \vf)(\zeta) = e^{-i t \zeta^2} \mathcal H_{\frac{a-1}2}(\vf)(\zeta),
\end{equation}
and \eqref{prop2}, we obtain
\[
\mathcal H_{\frac{a-1}2}(\widehat{\mathbb T_a(U)})(\xi,\zeta) = \int_\R e^{4\pi^2 i t |\xi|^2 + i t \zeta^2} \mathcal H_{\frac{a-1}2}(\hat U(\cdot,\cdot,t))(\xi,\zeta) dt,
\]
and the right-hand side is precisely the right-hand side of \eqref{restrSigma} with $\nu = \frac{a-1}2$. This proves the identity
\begin{equation}\label{resgen30}
\mathcal H_{\frac{a-1}2}(\widehat{\mathbb T_a(U)})(\xi,\zeta) = \tilde{\mathcal H}_{\frac{a-1}2}(U)\Big(\xi,\zeta,- \frac{4\pi^2 |\xi|^2 +\zeta^2}{2\pi}\Big),
\end{equation}
which identifies the restriction to $\Sigma$ with the Fourier-Hankel transform of $\mathbb T_a(U)$. On the other hand, by the Plancherel theorem for the Fourier-Hankel transform, see e.g. \cite{GS}, we have
\begin{equation}\label{resgen3}
||\mathbb T_a(U)||_{L^2_a(\RN)} = ||\mathcal H_{\frac{a-1}2}(\widehat{\mathbb T_a(U)})||_{L^2_a(\RN)}.
\end{equation}
Combining \eqref{resgen30} and \eqref{resgen3} with the estimate \eqref{resgen} of Theorem \ref{T:resgen}, we obtain \eqref{resgen4}. The final assertion follows in the same way from \eqref{resgen2}.
\end{proof}

\vskip 0.3in

\noindent \textbf{Declarations:}
\begin{itemize}
\item[(1)]  This manuscript does not report data generation or analysis.
\item[(2)] Both authors contributed equally to this work.
\item[(3)] G. Staffilani is funded in part by the NSF grant DMS-2306378 and the Simons Foundation through the
Simons Collaboration Grant on Wave Turbulence.
\item[(4)] Corresponding author: Nicola Garofalo
\end{itemize}




\bibliographystyle{amsplain}

\end{document}